\documentclass{amsart}

\usepackage[T1]{fontenc}
\usepackage[utf8]{inputenc}
\usepackage{amsmath,amsthm,amssymb,mathtools,mathrsfs,thmtools}
\usepackage{tikz-cd,enumitem}

\PassOptionsToPackage{hyphens}{url}
\usepackage[style=ext-alphabetic, backend=biber, sorting=nyt, doi=false, isbn=false, url=false, hyperref=true, backrefstyle=none, maxalphanames=3, minalphanames=3, maxcitenames=3, mincitenames=3, maxbibnames=99]{biblatex}
\usepackage{hyperref}
\usepackage{xurl}
\hypersetup{breaklinks=true}

\definecolor{webgreen}{rgb}{0,.5,0}
\definecolor{webbrown}{rgb}{.6,0,0}
\definecolor{webblue}{rgb}{0,0,.8}
\providecommand\urlcolor{webbrown}
\providecommand\citecolor{webgreen}

\providecommand\linkcolor{webblue}
\hypersetup{colorlinks, citecolor=\citecolor, urlcolor=\urlcolor,
  anchorcolor=\urlcolor, linkcolor=\linkcolor}

\usepackage[capitalize, noabbrev]{cleveref}

\renewbibmacro{in:}{}
\DeclareFieldFormat
[article,inbook,incollection,inproceedings,patent,thesis,unpublished]
{titlecase:title}{\MakeSentenceCase{#1}}
\newbibmacro{string+doiurlisbn}[1]{%
  \iffieldundef{doi}{%
      \iffieldundef{url}{%
          #1%
        }{%
      \href{\thefield{url}}{#1}%
    }%
  }{%
    \href{https://dx.doi.org/\thefield{doi}}{#1}%
  }%
}
\DeclareFieldFormat{title}{\usebibmacro{string+doiurlisbn}{\mkbibemph{#1}}}
\DeclareFieldFormat[online]{title}{\mkbibemph{#1}}
\DeclareFieldFormat[article,incollection,inproceedings,unpublished,misc,book]{title}%
{\usebibmacro{string+doiurlisbn}{\mkbibquote{#1}}}
\AtBeginBibliography{\small}

\newcommand{\stackcite}[1]{%cite tags on the stack project
\cite[\href{https://stacks.math.columbia.edu/tag/#1}{Tag #1}]{thestacksprojectauthorsStacksProject2018}}
\bibliography{isogenies-genus-2,refs-with-urls}

\newtheorem{thm}{Theorem}[section]
\Crefname{thm}{Theorem}{Theorems}
\newtheorem{lem}[thm]{Lemma}
\newtheorem{prop}[thm]{Proposition}
\Crefname{prop}{Proposition}{Propositions}
\newtheorem{cor}[thm]{Corollary}
\theoremstyle{definition}
\newtheorem{rem}[thm]{Remark}

\newtheorem{defn}[thm]{Definition}
\newtheorem{algo}[thm]{Algorithm}
\numberwithin{equation}{section}

\newcommand{\overbar}[1]{\mkern 1.5mu\overline{\mkern-1.5mu#1\mkern-1.5mu}\mkern 1.5mu}
\newcommand{\from}{\colon}
\newcommand{\defby}{\colon}
\newcommand{\st}{\colon}
\newcommand{\conj}[1]{\overbar{#1}}
\newcommand{\wideconj}[1]{\overline{#1}}
\newcommand{\embeds}{\hookrightarrow}

\newcommand{\setdiv}{\backslash}
\newcommand{\tp}{t}

\newcommand{\iso}{\ensuremath{\simeq}}
\newcommand{\defi}{\coloneqq}

\newcommand{\dualv}{\vee}
\newcommand{\sym}{\ensuremath{\mathrm{sym}}}

\newcommand\tangent[2]{T_{#1}(#2)}

\newcommand{\different}{\Z_K^\dualv}

\newcommand{\mat}[4]{\left(\begin{matrix}#1&#2\\#3&#4\end{matrix}\right)}
\newcommand{\tmat}[4]{\left(\begin{smallmatrix}#1&#2\\#3&#4\end{smallmatrix}\right)}
\newcommand{\vectwo}[2]{\left(\begin{matrix}#1\\#2\end{matrix}\right)}

\newcommand{\R}{\mathbb{R}}
\newcommand{\C}{\mathbb{C}}
\newcommand{\Q}{\mathbb{Q}}
\newcommand{\Z}{\mathbb{Z}}
\newcommand{\F}{\mathbb{F}}

\newcommand{\Pvar}{\mathbb{P}}
\newcommand{\Avar}{\mathbb{A}}
\newcommand{\Half}{\mathbb{H}}
\newcommand{\Crv}{{\mathcal{C}}}
\newcommand{\M}{\mathcal{M}}
\newcommand{\Lb}{{\mathcal{L}}}
\newcommand{\A}{\mathcal{A}}
\newcommand{\Otilde}{\smash{\widetilde{O}}}
\newcommand{\Zhat}{\widehat{\Z}}
\newcommand{\kbar}{\overline{k}}
\newcommand{\betabar}{\overline{\beta}}
\newcommand{\XX}{\mathscr{X}}
\newcommand{\CC}{\mathscr{C}}
\newcommand{\MM}{\mathscr{M}}

\newcommand{\YY}{\mathscr{Y}}
\newcommand{\OO}{\mathcal{O}}

\newcommand{\DeclareMyOperator}[1]{%
  \expandafter\DeclareMathOperator\csname #1\endcsname{#1} }
\newcommand{\DeclareMathOperators}{\forcsvlist{\DeclareMyOperator}}
\DeclareMathOperators{ord,sng,Card,rank,Hom,Ker,End,Aut,Sym,Alt,id,Id,lcm,pgcd,ppcm,argch,argsh,argth,Math,GL,Gl,Sp,SL,Sl,PGL,PSL,SO,SU,PSU,Spec,Var,Proj,Jac,Gal,Div,PDiv,Pic,Ext,Tor,Cl,Tr,NS,Cov,Diag,Lie,Isom,Discr,Mat}
\DeclareMathOperator{\chr}{char}

\newcommand{\pol}{\mathcal{L}}
\newcommand{\prc}{\nu}
\newcommand{\djdtau}{DJ}

\newcommand{\stack}{\mathcal{A}}
\newcommand{\coarse}{\mathbf{A}}

\newcommand{\av}{A}
\newcommand{\isog}{\varphi}
\newcommand{\disog}{d\varphi}
\newcommand{\defo}{\mathscr{D}}

\newcommand{\Hodge}{\mathsf{H}}

\newcommand{\Ag}{\coarse_{g}}
\newcommand{\Agn}{{\coarse}_{g,n}}
\newcommand{\Atwo}{\coarse_2}
\newcommand{\Mtwo}{\mathbf{M}_2}
\newcommand{\Agl}{{{\coarse}_{g}(\ell)}}

\newcommand{\Htwo}{{\mathbf{H}_2}}
\newcommand{\Hg}{{\mathbf{H}_{g}}}
\newcommand{\Hgbeta}{\mathbf{H}_{g}(\beta)}
\newcommand{\PHI}{\mathbf{\Phi}}
\newcommand{\PHII}{\Psi_0} %{\symbfsf{\Phi_0}} % image of PHI
\newcommand{\PHIIb}{\Psi_{\beta}} %image of PHI(beta)
\newcommand{\PHIIH}{\Psi_{\beta,\betabar}} %{\symbfsf{\Phi_1}} % PHII on the Humbert side
\newcommand{\AAg}{{\stack_{g}}}
\newcommand{\AAgn}{\stack_{g,n}}
\newcommand{\AAGL}{{\stack_{g,\ell}}}
\newcommand{\AAgl}{{\stack_{g}(\ell)}}
\newcommand{\AAtwo}{\stack_2}
\newcommand{\XXg}{{\mathscr{X}_{g}}}
\newcommand{\XXgl}{{\mathscr{X}_{g}(\ell)}}
\newcommand{\HHg}{{\mathscr{H}_{g}}}

\newcommand{\HHgbeta}{\mathscr{H}_{g}(\beta)}

\newcommand{\PPHI}{\mathit{\Phi}}
\newcommand{\poly}{W}

\title[Computing isogenies from modular equations in genus two]{Computing isogenies from modular equations\\ in genus two}

\author{Jean Kieffer}

\author{Aurel Page}

\author{Damien Robert}

\subjclass{14K02, 14K10, 14Q20}
\keywords{Abelian varieties, isogenies, modular equations, algorithms}

\begin{document}

\begin{abstract}
  Consider two genus~$2$ curves over a field whose Jacobians are linked by an
  isogeny of known type: either an~$\ell$-isogeny or, in the real
  multiplication case, an isogeny with cyclic kernel. We present a completely
  algebraic algorithm to compute this isogeny using modular equations of either
  Siegel or Hilbert type. An essential step of independent interest is to
  construct an explicit Kodaira--Spencer isomorphism for principally polarized
  abelian surfaces.
\end{abstract}

\maketitle

\section{Introduction}

Since the pioneering work of Vélu~\cite{veluIsogeniesEntreCourbes1971} in the
case of elliptic curves, several algorithms are available to solve the
following problem: given a principally polarized~(p.p.) abelian variety~$A$ and
a torsion subgroup~$K$ of~$A$ such that~$A/K$ is also principally polarizable,
compute the quotient isogeny~$A\to A/K$. Some of these algorithms work with
Jacobians of curves, of genus~$2$ in
particular~\cite{cossetComputingEllEll2015,
  couveignesComputingFunctionsJacobians2015}; others use theta functions and
apply in every dimension~\cite{lubiczComputingSeparableIsogenies2015,
  dudeanuCyclicIsogeniesAbelian2022, lubiczFastChangeLevel2022}.

In this paper, we are interested in the reverse question: given two
p.p.~abelian varieties~$A$ and~$A'$ linked by an isogeny~$\isog$ of a known
type and degree but unknown kernel, compute~$\isog$. We present a completely
algebraic algorithm for this task that generalizes Elkies's isogeny algorithm
for elliptic curves~\cite{elkiesEllipticModularCurves1998}, and thus solve a
longstanding open problem in isogeny computations
\cite[§1.1.2]{ballentineIsogeniesPointCounting2016}.

\subsection{Main results}

Elkies's algorithm uses an explicit equation for the modular curve of
level~$\Gamma_0(\ell)$ to compute $\ell$-isogenies between elliptic curves,
where~$\ell$ is a prime. More generally, we explain how algebraic equations
encoding the presence of isogenies of a given type between abelian varieties,
called \emph{modular equations}, can be used to compute isogenies in every
dimension. In the case of Jacobians of genus~2 curves, we describe the
resulting algorithm completely.  Let us state a simplified version of our main
result (\cref{thm:main_proved}) in the case of~$\ell$-isogenies (of
degree~$\ell^2$) where~$\ell$ is a prime, described by modular equations of
Siegel type~\cite{brokerModularPolynomialsGenus2009,
  milioQuasilinearTimeAlgorithm2015}.

\begin{thm}
  \label{thm:main}
  Let~$\ell$ be a prime, and let~$k$ be a field such that
  $\chr k = 0$ or $\chr k > 8\ell+ 1$. Then, given the data of
  \begin{enumerate}
  \item two generic~$\ell$-isogenous p.p.~abelian surfaces~$A$ and~$A'$
    over~$k$, and
  \item the derivatives of modular equations of Siegel type and level~$\ell$
    at~$(A,A')$,
  \end{enumerate}
  one can compute an~$\ell$-isogeny~$\isog\from A\to A'$.  This algorithm
  costs~$\Otilde(\ell)$ elementary operations and~$O(1)$ square roots in~$k$.
\end{thm}

We also obtain a similar result (\cref{thm:proved-main-hilbert}) for cyclic
isogenies between p.p. abelian surfaces with real multiplication. The algorithm
is then based on modular equations of Hilbert type~\cite{
  martindaleHilbertModularPolynomials2020,
  milioModularPolynomialsIlbert2020}. Note that, as in the case of elliptic
curves, computing roots of modular equations (over finite fields in particular)
is a typical way of generating suitable input for our isogeny algorithms.

\subsection{Comparison with previous works}

Other polynomial-time algorithms to compute an isogeny $\isog\from A\to A'$
exist, in every dimension~$g$. For instance, one could compute $k$-rational
subgroups of the $\ell$-torsion group $A[\ell]$ and apply an algorithm to
compute quotient isogenies. However, the torsion subgoups~$A[\ell]$ are
difficult to manipulate as~$\ell$ grows, due to their large
size~$\ell^{2g}$. In another direction, for abelian surfaces specifically, van
Wamelen~\cite{vanwamelenPoonenQuestionConcerning2000,
  vanwamelenComputingAnalyticJacobian2006} describes an isogeny algorithm
using complex approximations; these ideas were later generalized to Jacobians
of arbitrary dimensions in~\cite{costaRigorousComputationEndomorphism2019}.
However, this numerical approach is inherently restricted to subfields of~$\C$
and lacks clear complexity estimates.
In comparison, the isogeny algorithm of \cref{thm:main} reconstructs the tangent
map of the isogeny exactly, and is extremely
efficient. Its practical cost is hidden in the evaluation of modular equations
and their derivatives, but these evaluations are still less costly than
manipulating the full torsion subgroups, both in the case of elliptic
curves~\cite{engeComputingModularPolynomials2009,
  sutherlandEvaluationModularPolynomials2013} and p.p.~abelian
surfaces~\cite{kiefferEvaluatingModularEquations2022}. In fact, computing
$\ell$-isogenies provides an efficient way of obtaining maximal isotropic
subgroups in~$A[\ell]$. This remark is at the heart of the
Schoof--Elkies--Atkin (or SEA) point-counting
algorithm~\cite{schoofEllipticCurvesFinite1985} for elliptic curves over
finite fields. In genus~$2$, one can similarly obtain asymptotic speedups over
point-counting methods that only rely on kernels of endomorphisms to construct
rational subgroups~\cite{gaudryCountingPointsGenus2011,
  gaudryGenusPointCounting2012}: we refer
to~\cite{kiefferCountingPointsAbelian2022} for a detailed analysis.

\subsection{Outline of the algorithm}

From a geometric point of view, we compute $\ell$-isogenies in arbitrary
dimension~$g$ as follows. Denote by~$\AAgl$ the moduli stack of p.p.~abelian
schemes of dimension~$g$ endowed with the kernel of an $\ell$-isogeny, and
by~$\AAg$ the moduli stack of p.p.~abelian schemes of dimension~$g$. Consider
the map
\begin{align*}
  \PPHI_{\ell} = (\PPHI_{\ell,1}, \PPHI_{\ell,2}) \from \AAgl &\to \AAg \times \AAg\\
  (A,K)&\mapsto (A, A/K).
\end{align*}
Both~$\PPHI_{\ell,1}$ and~$\PPHI_{\ell,2}$ are étale maps. Let
$\isog\from \av\to \av'$ be an $\ell$-isogeny, and let~$x,x'$ be the points
of~$\AAg$ corresponding to~$A$ and~$A'$. Then the \emph{Kodaira--Spencer
  isomorphism} between~$T_x(\AAg)$ and~$\Sym^2 T_0(\av)$ yields a close
relation between two maps:
\begin{itemize}
\item the \emph{deformation map}
  $\defo(\isog) \defi d {\PPHI_{\ell,2}} \circ d {\PPHI_{\ell,1}}^{-1} \from
  T_x(\AAg) \to T_{x'}(\AAg)$, and
\item the \emph{tangent map} $d\isog\from T_0(\av)\to T_0(\av')$.
\end{itemize}
Therefore, in any dimension~$g$, an isogeny algorithm could run as
follows.
\begin{enumerate}
\item \label{step:defo} Compute the deformation map by differentiating
  certain modular equations giving a local model of~$\AAgl$
  and~$\AAg$.
\item \label{step:tangent} Compute~$d\isog$ from the deformation map
  by using an explicit version of the Kodaira--Spencer isomorphism.
\item \label{step:formal} Finally, compute~$\isog$ by solving a differential
  system in the formal group of~$\av$ and performing a rational reconstruction,
  as in~\cite{couveignesComputingFunctionsJacobians2015,
    costaRigorousComputationEndomorphism2019}.
\end{enumerate}
The whole method, when applied to elliptic curves, is indeed a reformulation of
Elkies's isogeny algorithm.

In practice, working with stacks would involve adding a level structure and
keeping track of automorphisms, which is not computationally
convenient. Therefore, in order to make everything explicit in the case~$g=2$,
we replace the stack~$\AAtwo$ by its coarse moduli scheme~$\Atwo$. We even work
up to birationality, by considering the map from~$\Atwo$ to~$\Avar^3$ defined
by the three Igusa invariants~$(j_1,j_2,j_3)$. These modifications simplify the
computations considerably, but have the drawback of introducing the genericity
assumptions in \Cref{thm:main}. In particular, we only consider abelian
surfaces~$A$ that are the Jacobian of a genus~$2$ curve~$\Crv$.

Working with genus~$2$ curves allows us to encode a basis of~$T_0(\av)$ in the
choice of an equation of~$\Crv$. Then, the explicit Kodaira--Spencer
isomorphism of Step~\eqref{step:tangent} is simply an expression for certain
Siegel modular functions, namely the derivatives of the Igusa invariants, in terms
of the coefficients of the curve equation. We compute these formulas building
on work of Cléry, Faber, and van der
Geer~\cite{cleryCovariantsBinarySextics2017}: see
\cref{thm:vector-identification}. This result of independent interest
generalizes the classical formula
\begin{displaymath}
  \frac{1}{2\pi i}\, \frac{dj}{d\tau} = -\frac{E_4^2 E_6}{\Delta}
\end{displaymath}
used in Elkies's isogeny algorithm for elliptic curves.

Finally, in Step~\eqref{step:formal}, we use the fact that~$\Crv$ embeds in its
Jacobian to compute with power series in one variable only, and use Newton
iterations to solve the differential system in quasi-linear time. The
hypothesis on $\chr k$ appears in this step, but is not essential: a standard
workaround in small characteristic would be to lift the isogeny to
characteristic zero, following~\cite{eidFastComputationHyperelliptic2021}.

\subsection{Organization of the paper}

In \Cref{sec:mf,sec:cov}, we work over $\C$: \Cref{sec:mf} is devoted to the
necessary background on modular forms and isogenies, and \Cref{sec:cov} is
devoted to the explicit Kodaira--Spencer isomorphism.  In \Cref{sec:moduli}, we
adopt the language of algebraic stacks to show that the calculations over~$\C$
remain in fact valid over any base.  We present the computation of the isogeny
from its tangent map in \Cref{sec:alg}, and review the whole algorithm in
\Cref{sec:summary}.  Finally, in \Cref{sec:Qr5}, we present variants in
the algorithm in the case of real multiplication by $\Q(\sqrt{5})$ and compute
an example of cyclic isogeny of degree~$11$.

\subsection{Acknowledgements}

A.P. and D.R. were supported by the ANR grant CIAO (French Agence Nationale de
la Recherche, number ANR-19-CE48-0008.) J.K. was supported by CIAO and
the Simons Foundation grant 550031 (to Noam D.~Elkies.)

\section{Background on modular forms and isogenies}
\label{sec:mf}

We present the basic facts about Siegel and Hilbert modular forms only in the
genus~$2$ case. References for this section
are~\cite{vandergeerSiegelModularForms2008} for Siegel modular forms,
and~\cite{bruinierHilbertModularForms2008} for Hilbert modular forms, where
the general case is treated.

We write~$4\times 4$ matrices in block notation using~$2\times 2$
blocks.  We write~$m^\tp$ for the transpose of a matrix~$m$, and use
the notations
\begin{displaymath}
  m^{-\tp} \defi (m^{-1})^\tp, \qquad \Diag(x,y) \defi \mat{x}{0}{0}{y}.
\end{displaymath}

\subsection{Siegel modular forms}
\label{subsec:siegel}

Denote by~$\Half_2$ the set of complex symmetric~$2\times 2$ matrices
with positive definite imaginary part. 
For every $\tau\in\Half_2$, the quotient
\begin{displaymath}
  A(\tau) \defi \C^2 / \Lambda(\tau) \quad \text{where}\quad
  \Lambda(\tau) = \Z^2 \oplus \tau\Z^2
\end{displaymath}
is naturally endowed with the structure of a principally polarized (p.p.)
abelian surface over~$\C$. A basis of $\Omega^1(A(\tau))$ is given by
\begin{displaymath}
  \omega(\tau) \defi (2\pi i\, dz_1,2\pi i\, dz_2)
\end{displaymath}
where $z_1, z_2$ are the coordinates on~$\C^2$.

The symplectic group~$\Sp_{4}(\Z)$ acts on~$\Half_{2}$ as follows:
for $\gamma = \tmat{a}{b}{c}{d}\in\Sp_{4}(\Z)$ and $\tau\in \Half_2$, we write
\begin{displaymath}
  \gamma\tau \defi (a\tau + b)(c\tau + d)^{-1}.
\end{displaymath}

The quotient space~$\Atwo(\C) = \Sp_4(\Z)\setdiv\Half_2$ is the set of complex
points of the coarse moduli space~$\Atwo$ mentioned in the introduction: for
every p.p.~abelian surface $A$ over~$\C$, there exists~$\tau\in\Half_2$, unique
up to the action of~$\Sp_4(\Z)$, such that~$A$ and~$A(\tau)$ are
isomorphic~\cite[Prop.~8.1.3]{birkenhakeComplexAbelianVarieties2004}.
For~$\gamma\in \Sp_4(\Z)$ as above, the linear map $z\mapsto (c\tau+d)^{-t}z$
yields an isomorphism
$A(\tau) \to
A(\gamma\tau)$~\cite[Rem.~8.1.4]{birkenhakeComplexAbelianVarieties2004}

Let~$\rho\from \GL_{2}(\C)\to\GL(V)$ be a finite-dimensional and irreducible
holomorphic representation of~$\GL_2(\C)$. A \emph{Siegel modular function} of
weight~$\rho$ is a meromorphic map~$f\from\Half_{2}\to V$ satisfying the
transformation rule
\begin{displaymath}
  f(\gamma\tau) = \rho(c\tau + d)f(\tau).
\end{displaymath}
for all $\gamma = \tmat{a}{b}{c}{d}\in\Sp_{4}(\Z)$ and~$\tau\in \Half_2$.
We say that~$f$ is \emph{scalar-valued} if~$\dim V=1$, and
\emph{vector-valued} otherwise.  A \emph{Siegel modular form} is a holomorphic
Siegel modular function.

If~$A$ is a p.p.~abelian surface over~$\C$ endowed with a basis~$\omega$ of
$\Omega^1(A)$ and~$f$ is a Siegel modular form of weight~$\rho$, then one can
evaluate~$f$ on the pair $(A,\omega)$:
see~\cite[p.\,141]{faltingsDegenerationAbelianVarieties1990} or
§\ref{subsec:moduliav} for a geometric interpretation of this fact. To compute
$f(A,\omega)$, choose~$\tau\in\Half_2$ and an
isomorphism~$\eta \from A \to A(\tau)$.  Let~$r\in\GL_2(\C)$ be the matrix of
the pullback map~$\eta^*\from \Omega^1(A(\tau))\to\Omega^1(A)$ in the
bases~$\omega(\tau)$ and $\omega$.  Then
\begin{displaymath}
  f(A,\omega) = \rho(r) f(\tau).
\end{displaymath}
One can directly check that~$f(A,\omega)$ does not depend on the choice
of~$\tau$ and~$\eta$.

\subsection{An explicit view on Siegel modular forms in genus 2}
\label{subsec:siegel-g2}

In genus~$2$, the possible weights of Siegel modular forms can be listed
explicitly: each representation~$\rho$ as above is isomorphic to
$\det^k\otimes\Sym^n$ for some~$k\in \Z$
and~$n\geq 0$~\cite[Prop.~15.47]{fultonRepresentationTheoryFirst1991}. We will
omit the tensor symbol.  Explicitly,~$\Sym^n$ is a representation on the vector
space~$V = \C_n[x]$ of polynomials of degree at most~$n$, and for all
$E\in \C_n[X]$ and $r = \tmat{a}{b}{c}{d}\in \GL_2(\C)$, we have
\begin{displaymath}
  \Sym^n(r)\, E = (bx + d)^n\ 
  E\left(\frac{ax + c}{bx + d}\right).
\end{displaymath}
We take $(x^n,\ldots,x,1)$ as the standard basis of~$\C_n[x]$, so that we can
write an endomorphism of~$\C_n[x]$ as a matrix. In particular we have
\begin{displaymath}
  \Sym^2(r) =
  \left(
    \begin{matrix}
      a^2 & ab & b^2 \\
      2ac & ad+bc & 2bd \\
      c^2 & cd & d^2
    \end{matrix}
  \right).
\end{displaymath}

The weight of a nonzero scalar-valued Siegel modular form~$f$ is of the
form~$\det^k$ for a unique $k\in\Z$, and in fact $k\geq 0$. We also say
that~$f$ is a scalar-valued Siegel modular form of \emph{weight~$k$}.
Writing~$\Sym^n$ as a representation on~$\C_n[x]$ allows us to multiply Siegel
modular forms. Thus, the graded vector space generated by Siegel modular forms
is also naturally a graded $\C$-algebra, called the \emph{graded algebra of
  Siegel modular forms.}\footnote{Under our definitions, not all elements of
  this graded algebra are modular forms: for instance, if~$f_1$ and~$f_2$ are
  nonzero modular forms of distinct weights, then $f_1+f_2$ is not a modular
  form.}

In order to represent a modular form explicitly, we use Fourier expansions.
Let~$f$ be a Siegel modular form on~$\Half_2$ of weight~$\det^k\Sym^n$,
with underlying vector space~$V = \C^{n+1}$.  If we write
\begin{displaymath}
  \tau = \mat{\tau_1}{\tau_2}{\tau_2}{\tau_3}
  \quad \text{and} \quad
  q_j = \exp(2\pi i \tau_j)\quad \text{for } 1\leq j\leq 3,
\end{displaymath}
then~$f$ has a Fourier expansion of the form
\begin{displaymath}
  f(\tau) = \sum_{n_1,n_2,n_3\in\Z} c_f(n_1,n_2,n_3)\,
  q_1^{n_1} q_2^{n_2} q_3^{n_3}.
\end{displaymath}
The Fourier coefficients~$c_f(n_1,n_2,n_3)$ belong to~$V$, and can be nonzero
only when $n_1\geq 0, n_3\geq 0$ and~$n_2^2\leq 4 n_1 n_3$ (note that~$n_2$
can still be negative). To compute with~$q$-expansions, we work in the power
series ring~$\C[q_2,q_2^{-1}][[q_1, q_3]]$ modulo an ideal of the
form~$\bigl(q_1^\prc, q_3^\prc\bigr)$ for some precision~$\prc\geq 0$.

Now we can describe the structure of the graded~$\C$-algebra of Siegel modular
forms. While the full algebra is not finitely
generated~\cite[Lem.~4]{vandergeerSiegelModularForms2008}, the subalgebra of
scalar-valued modular forms is.

\begin{thm}[{\cite{igusaSiegelModularForms1962,
    igusaModularFormsProjective1967}}]
  \label{thm:siegel-structure}
  The graded $\C$-algebra of scalar-valued even-weight Siegel modular
  forms in genus~$\,2$ is generated by four algebraically independent
  elements~$\psi_4, \psi_6, \chi_{10}$, and $\chi_{12}$ of respective
  weights~$4,\, 6,\, {10},\, {12}$, and~$q$-expansions
  \begin{displaymath}
    \begin{aligned}
      \psi_4(\tau) &= 1 + 240(q_1 + q_3) \\
      & \quad  + \bigl(240 q_2^2 + 13440
      q_2 + 30240 + 13340 q_2^{-1} + 240 q_2^{-2} \bigr) q_1 q_3 +
      O\bigl(q_1^2,q_3^2 \bigr), \\
      \psi_6(\tau) &= 1 - 504(q_1 + q_3) \\
      & \quad + \bigl(-504 q_2^2 + 44352 q_2 +
      166320 + 44352 q_2^{-1} - 504 q_2^{-2}\bigr) q_1q_3 + O\bigl(q_1^2,q_3^2 \bigr), \\
      \chi_{10}(\tau) &= \bigl(q_2 - 2 + q_2^{-1}\bigr) q_1 q_3 + O(q_1^2,q_3^2),
      % (-2q_2^2 - 16 q_2 + 36 - 16q_2^{-1} - 2 q_2^{-2})(q_1 q_3^2 +
      % q_1^2 q_3) + O(q_1^2,q_3^2)
      \\
      \chi_{12}(\tau) &= \bigl(q_2 + 10 + q_2^{-1}\bigr) q_1 q_3 + O\bigl(q_1^2,q_3^2 \bigr).
      % (x q_2^2 - 88 q_2 -132 - 88 q_2^{-1} + x q_2^{-2}) (q_1 q_3^2
      % + q_1^2 q_3) + O(q_1^2,q_3^2)
      \\
    \end{aligned}
  \end{displaymath}
  The graded~$\C$-algebra of scalar-valued Siegel modular forms in genus~$2$ is
  \begin{displaymath}
    \C[\psi_4,\psi_6,\chi_{10},\chi_{12}]
    \oplus \chi_{35} \C[\psi_4,\psi_6,\chi_{10},\chi_{12}] 
  \end{displaymath}
  where $\chi_{35}$ is a modular form of weight~$35$ and
  $q$-expansion
  \begin{displaymath}
    \chi_{35}(\tau) = q_1^2 q_3^2 (q_1 - q_3) (q_2 - q_2^{-1}) + O(q_1^4, q_3^4).
  \end{displaymath}
\end{thm}

The $q$-expansions in \Cref{thm:siegel-structure} are easily computed from
expressions in terms of theta functions
\cite[§7.1]{strengComputingIgusaClass2014},
\cite[p.\,493]{bolzaDarstellungRationalenGanzen1887}, and their Fourier
coefficients are integers. We warn the reader that different normalizations
appear in the literature: for instance, our~$\chi_{10}$ is~$4$ times the
modular form~$\chi_{10}$ appearing in Igusa's papers, and our~$\chi_{12}$
is~$12$ times Igusa's~$\chi_{12}$.

The equality~$\chi_{10}(\tau) = 0$ occurs exactly when~$A(\tau)$ is isomorphic
to a product of elliptic curves (with the product polarization).
When~$\chi_{10}(\tau) \neq 0$, the p.p.~abelian surface~$A(\tau)$ is isomorphic to the Jacobian of a hyperelliptic
curve. Following~\cite[§2.1]{strengComputingIgusaClass2014} and our choice of
normalizations, we define the \emph{Igusa invariants} to be
\begin{displaymath}
  j_1 \defi 2^{-8} \frac{\psi_4\psi_6}{\chi_{10}},\quad j_2 \defi
  2^{-5} \frac{\psi_4^2\chi_{12}}{\chi_{10}^2}, \quad
  j_3 \defi 2^{-14} \frac{\psi_4^5}{\chi_{10}^2}.
\end{displaymath}
The Igusa invariants $j_1,j_2,j_3$ are Siegel modular functions of weight~$0$,
and together define a birational map $\Atwo(\C)\to\C^3$.

\begin{rem}
  \label{rem:invariants_bielliptic}
  Generically, giving~$(j_1, j_2, j_3)\in \C^3$ uniquely specifies an
  isomorphism class of p.p.~abelian surfaces over~$\C$. This correspondence
  only holds on an open set: the Igusa invariants are not defined on products
  of elliptic curves, and do not represent a unique isomorphism class
  when~$\psi_4 = 0$. To consider these points nonetheless, it is best to use
  other invariants: for instance the invariants
  \begin{displaymath}
    h_1 \defi \dfrac{\psi_6^2}{\psi_4^3},\quad h_2 \defi
    \dfrac{\chi_{12}}{\psi_4^3}, \quad h_3 \defi \dfrac{\chi_{10}\psi_6}{\psi_4^4}
  \end{displaymath}
  are generically well-defined on products of elliptic curves.  See
  \cite[Thm.~1.V]{liuCourbesStablesGenre1993} for the expression of these
  invariants in terms of $j(E_1)+j(E_2)$ and $j(E_1)j(E_2)$ when evaluated on a
  product $E_1 \times E_2$.
\end{rem}

We conclude this paragraph by describing key examples of vector-valued
forms. First, if~$f$ is a Siegel modular function of weight~$0$, then its
derivative
\begin{displaymath}
  Df \defi \frac{1}{2\pi i} \Bigl(\dfrac{\partial
    f}{\partial\tau_1} x^2 + \dfrac{\partial
    f}{\partial\tau_2} x + \dfrac{\partial
    f}{\partial\tau_3} \Bigr): \Half_2\to \C_2[x]
\end{displaymath}
is a Siegel modular function of weight~$\Sym^2$. This property stems from the
existence of the Kodaira--Spencer isomorphism; it can also be seen as a special
case of Rankin--Cohen operators~\cite[§25]{vandergeerSiegelModularForms2008},
or be checked directly by differentiating the
relation~$f(\gamma\tau) = f(\tau)$ with respect to~$\tau$.

The second key example is the modular form~$\chi_{6,8}$ of
weight~$\det^8\Sym^6$~\cite{ibukiyamaVectorvaluedSiegelModular2012,
  cleryCovariantsBinarySextics2017}, with Fourier expansion
\begin{displaymath}
  \begin{aligned}
    \chi_{6,8}(\tau) = & \quad \left((4 q_2^2 - 16 q_2 + 24 - 16
      q_2^{-1} + 4 q_2^{-2}) q_1^2 q_3 + \cdots\right) x^6 \\
    & + \left((12 q_2^2 - 24 q_2 + 24 q_2^{-1} - 12 q_2^{-2}) q_1^2
      q_3 + \cdots \right) x^5 \\
    & + \left((- q_2 + 2 - q_2^{-1}) q_1 q_3 + \cdots \right) x^4 \\
    & + \left((-2 q_2 + 2 q_2^{-1}) q_1 q_3 + \cdots \right) x^3 \\
    & + \left((-q_2 + 2 - q_2^{-1}) q_1 q_3 + \cdots \right) x^2 \\
    & + \left((12 q_2^2 - 24 q_2 + 24 q_2^{-1} - 12 q_2^{-2}) q_1
      q_3^2 + \cdots \right) x \\
    & + \left((4 q_2^2 - 16 q_2 + 24 - 16 q_2^{-1} + 4 q_2^{-2})q_1
      q_3^2 + \cdots \right).
  \end{aligned}
\end{displaymath}
The modular form $\chi_{6,8}$ is in a sense ``universal'', as it provides a
link with equations of genus 2 curves: see \cref{sec:cov}.

\subsection{Hilbert modular forms}
\label{subsec:hilbert}

In the context of Hilbert surfaces and abelian surfaces with real
multiplication, we consistently use the following notation:

\begin{center}
  \begin{tabular}{ccl}
    $\Half_1$ & & the upper half plane in~$\C$ \\
    $K$ & & a real quadratic number field (embedded in~$\R$)\\
    $\Delta$ & & the discriminant of~$K$, so that $K = \Q\bigl(\sqrt{\Delta}\bigr)$ \\
    $\Z_K$ & & the ring of integers in~$K$ \\
    $\Z_K^\dualv$ & & the trace dual of~$\Z_K$, in other words
                      $\Z_K^\dualv = 1/\sqrt{\Delta}\ \Z_K$ \\
    $x\mapsto \conj{x}$ & & real conjugation in~$K$ \\
    $\Sigma$ & & the embedding~$x\mapsto (x,\conj{x})$ from~$K$ to~$\R^2$ \\
    $\sigma$ & & the involution $(t_1,t_2) \mapsto (t_2,t_1)$ of~$\Half_1^2$.
  \end{tabular}
\end{center}
Finally, the Hilbert modular group~$\Gamma_K$ is defined as follows:
\begin{displaymath}
  \Gamma_K = \SL\bigl(\Z_K\oplus \Z_K^\dualv\bigr) = \left\{\mat{a}{b}{c}{d} \in\SL_2(K)
    \,\st\, a,d\in\Z_K,\, b\in \bigl(\Z_K^\dualv\bigr)^{-1},\, c\in \Z_K^\dualv
  \right\}.
\end{displaymath}

Let~$A$ be a p.p.~abelian surface. We denote by~$\End^{\dagger}(A)$ the set of
endomorphisms of~$A$ that are invariant under the Rosati involution
(see~\cite[§17]{milneAbelianVarieties1986} for a definition). A \emph{real
  multiplication structure} by~$\Z_K$ on~$A$ is an embedding
\begin{displaymath}
  \iota\from \Z_K\embeds \End^{\dagger}(A).
\end{displaymath}
We say that $A$ has \emph{real multiplication by~$\Z_K$} if it is endowed with
a real multiplication structure. We sometimes use this terminology when $\iota$
is not explicitly given: we then make an implicit choice of a real multiplication
embedding.

As in the Siegel case, the coarse moduli space of p.p.~abelian surfaces
over~$\C$ with real multiplication by~$\Z_K$ can be constructed
complex-analytically. For each $t = (t_1,t_2)\in\Half_1^2$, the complex torus
\begin{displaymath}
  A_K(t) \defi \C^2/\Lambda_K(t) \quad\text{where} \quad
  \Lambda_K(t) = \Sigma\bigl(\Z_K^\dualv\bigr) \oplus \Diag(t_1,t_2)\,\Sigma\bigl(\Z_K\bigr)
\end{displaymath}
can be endowed with the structure of a p.p.~abelian surface over~$\C$, and
admits a real multiplication embedding~$\iota_K(t)$ given by multiplication
via~$\Sigma$. It is also endowed with the basis of differential forms
\begin{displaymath}
\omega_K(t) \defi (2\pi i\, dz_1, 2\pi i\, dz_2).
\end{displaymath}
The embedding~$\Sigma$ induces a map $\Gamma_K \embeds \SL_2(\R)^2$. The
group~$\Gamma_K$ thus acts on~$\Half_1^2$ by the usual action of~$\SL_2(\R)$
on~$\Half_1$ on each coordinate. The
quotient~$\Htwo(\C) = \Gamma_K\setdiv \Half_1^2$ is the moduli space we are
looking for: for each~$(A,\iota)$ as above, there exists~$t\in\Half_1^2$ such
that~$(A,\iota)$ is isomorphic to~$\bigl(A_K(t),\iota_K(t)\bigr)$, and~$t$ is
uniquely determined up to the action
of~$\Gamma_K$~\cite[§9.2]{birkenhakeComplexAbelianVarieties2004}. The
involution~$\sigma$ descends to~$\Htwo(\C)$ and exchanges the real
multiplication embedding with its conjugate.  In fact, the quotient~$\Htwo(\C)$
is the set of complex points of an algebraic variety~$\Htwo$ defined over~$\Q$,
called the \emph{Hilbert surface} attached to~$K$.

Let~$k_1, k_2\in \Z$. A \emph{Hilbert modular function} of weight $(k_1, k_2)$ is a
meromorphic function~$f\from\Half_1^2\to\C$ such that for all
$\gamma = \bigl(\begin{smallmatrix}a & b \\ c &
  d\end{smallmatrix}\bigr)\in\Gamma_K$ and all~$t\in\Half_1^2$,
\begin{displaymath}
  f(\gamma t) = \bigl(c\, t_1 + d\bigr)^{k_1}
  \bigl(\wideconj{c}\,t_2 + \conj{d}\bigr)^{k_2} f(t).
\end{displaymath}
Note that all irreducible finite-dimensional representations of~$\GL_1(\C)^2$ have
dimension~$1$, so there is no need to consider vector-valued forms.  We say
that~$f$ is \emph{symmetric} if~$f\circ\sigma = f$.  If~$f$ is nonzero and
symmetric, then its weight~$(k_1, k_2)$ is automatically \emph{parallel},
meaning~$k_1 = k_2$. A \emph{Hilbert modular form} is a holomorphic Hilbert
modular function.

\subsection{The Hilbert embedding}
\label{subsec:hilbert-siegel}

Forgetting the real multiplication structure yields a map
$\Htwo(\C)\to \Atwo(\C)$ from the Hilbert surface to the Siegel threefold. This
forgetful map comes from a linear map
%\begin{displaymath}
$H\from \Half_1^2 \to \Half_2$
%\end{displaymath}
called the \emph{Hilbert embedding}, which we now describe explicitly. Let
$(e_1,e_2)$ be a $\Z$-basis of~$\Z_K$. To make a deterministic choice, we take
$e_1=1$ and $e_2 = \frac12({1 - \sqrt{\Delta}})$ (resp.~$e_2 = \sqrt{\Delta}$) when
$\Delta$ is~$1\bmod{4}$ (resp.~$0 \bmod{4}$). Set
%\begin{displaymath}
$R = \tmat{e_1}{e_2}{\conj{e_1}}{\conj{e_2}}$,
%\end{displaymath}
and define
\begin{displaymath}
  H\from \Half_1^2\to \Half_2,\qquad t = (t_1,t_2)\mapsto R^\tp\, \Diag(t_1,t_2)\, R.
\end{displaymath}
Then, for every $t\in\Half_1^2$, the left multiplication by~$R^\tp$ on~$\C^2$
induces an isomorphism $A_K(t) \to A \bigl(H(t) \bigr)$
\cite[p.\,209]{vandergeerHilbertModularSurfaces1988}. Indeed we have
\begin{displaymath}
  \Lambda_K(t) = R^{-t}\Z^2 \oplus R^{-t} \bigl(R^t \Diag(t_1,t_2) R\bigr) \Z^2 = R^{-t}\Lambda(H(t)).
\end{displaymath}

The Hilbert embedding is compatible with the actions of the modular groups, as
follows.  Let~$\Gamma_K$ act on~$\Half_2$ by means of the morphism
$\Gamma_K\to \Sp_4(\Z)$ given by
\begin{displaymath}
  \mat{a}{b}{c}{d} \mapsto \mat{R^\tp}{0}{0}{R^{-1}}
  \mat{a^*}{b^*}{c^*}{d^*}
  \mat{R^{-\tp}}{0}{0}{R}
\end{displaymath}
where we write~$x^* = \Diag(x,\conj{x})$ for~$x\in K$. The Hilbert
embedding~$H$ is then equivariant for the actions of~$\Gamma_K$ on~$\Half_1^2$
and~$\Half_2$. The involution~$\sigma$ of~$\Half_1^2$ also corresponds via~$H$
to an element~$M_\sigma\in \Sp_4(\Z)$, namely
\begin{displaymath}
  M_\sigma = \mat{
    \begin{matrix}
      1&0\\ \delta &-1
    \end{matrix}
  }{(0)}{(0)}{
    \begin{matrix}
      1& \delta\\0&-1
    \end{matrix}
  }
\end{displaymath}
where~$\delta = 1$ if~$\Delta = 1\ \mathrm{mod}\ 4$, and~$\delta=0$ otherwise
\cite[Prop.~3.1]{lauterComputingGenusCurves2011}.

Using this compatibility, we can directly check that pulling back a Siegel
modular form via the Hilbert embedding yields Hilbert modular forms.

\begin{prop}
  \label{prop:mf-pullback}
  Let $k\in\Z$, $n\in\Z_{\geq 0}$, and let~$f \from \Half_2\to \C_n[x]$ be a
  Siegel modular form of weight~$\rho = \det^k\Sym^n$. Define the
  functions~$g_i\from\Half_1^2\to\C$ for $0\leq i\leq n$ by
  \begin{displaymath}
    \sum_{i=0}^n g_i(t)\, x^i =
    \rho(R) f \bigl(H(t) \bigr) \quad \text{for all } t\in \Half_1^2.
  \end{displaymath}
  Then each~$g_i$ for $0\leq i\leq n$ is a Hilbert modular form of
  weight~$(k+i,\, k+n-i)$, and we have $g_i\circ\sigma = g_{n-i}$. In
  particular, if $n=0$ and~$f$ is a scalar-valued Siegel modular form of
  weight~$\,\det^k$, then the function
  $H^*f\defby t\mapsto f \bigl(H(t) \bigr)$ is a symmetric Hilbert modular form
  of parallel weight~$(k,k)$.
\end{prop}

The image of the Hilbert embedding~$H$ in~$\Atwo(\C)$ is called the
\emph{Humbert surface} attached to~$K$. The pullback of~$\chi_{10}$ by the
Hilbert embedding is nonzero because a generic p.p.~abelian
surface over~$\C$ with real multiplication by~$\Z_K$ is not a product of two
elliptic curves~\cite[IX,
Prop.~1.2]{vandergeerHilbertModularSurfaces1988}. Moreover, the pullback
of~$\psi_4$ is nonzero, since its Fourier expansion as a Hilbert modular form
has a nonzero constant
term~\cite[Prop.~3.1]{lauterComputingGenusCurves2011}. As a consequence, the
Igusa invariants define a birational map from the Humbert surface to its image
in~$\C^3$. The squarefree polynomial cutting out this image is called the
\emph{Humbert equation}. This equation grows quickly in size with the
discriminant~$\Delta$, but can be computed in small
cases~\cite{gruenewaldComputingHumbertSurfaces2010}.

\subsection{Isogenies between abelian surfaces}
\label{subsec:isogenies}

Let~$A$ be a p.p.~abelian surface over~$k$. Denote its dual by~$A^\dualv$ and
its principal polarization by $\pi\from A\to A^\dualv$. For every line
bundle~$\Lb$ on~$A$, there is a morphism $\phi_{\Lb}\from A\to A^\dualv$
defined by $\phi_\Lb(x) = t_x^* \Lb\otimes\Lb^{-1}$, where~$t_x$ denotes
translation by~$x$ on~$A$. Let~$\NS(A)$ denote the Néron--Severi group of~$A$,
consisting of algebraic equivalence classes of line bundles. A fundamental fact
is that $\NS(A)$ is completely described in terms of endomorphisms of~$A$
over~$k$.

\begin{thm}[{\cite[Thm.~2 p.\,188, Thm.~3 p.\,231 and Application~III
    p.\,209]{mumfordAbelianVarieties1970}}]
  \label{thm:NS-End}
  For every $\xi\in\smash{\End^{\dagger}}(A)$, there exists a line
  bundle~$\Lb_A(\xi)$ (possibly defined over an extension of~$k$) such that
  $\phi_{\Lb_A(\xi)} = \pi\circ\xi$.  The map~$\xi \mapsto \Lb_A(\xi)$ induces
  an isomorphism of groups
  %\begin{displaymath}
  $(\smash{\End^{\dagger}}(A),+) \simeq (\NS(A),\otimes)$.
  % \end{displaymath}
  The morphism $\phi_{\Lb_A(\xi)}$ is a polarization on~$A$ if and only
  if~$\xi\in \smash{\End^{\dagger}}(A)$ is totally positive.
\end{thm}
In this notation,~$\Lb_A(1)$ is the line bundle associated with the
polarization~$\pi$.

Now, let~$\isog:A\to A'$ be an isogeny between p.p.~abelian surfaces. The line
bundle $\isog^* \Lb_{A'}(1)$ defines another polarization on~$A$, hence is
algebraically equivalent to~$\Lb_A(\xi)$ for some totally
positive~$\xi\in \smash{\End^\dagger}(A)$. Provided that~$A$ is simple, there
are two possibilities~\cite[p.\,202]{mumfordAbelianVarieties1970}:
either~$\Q(\xi)=\Q$, in which case~$\xi$ is a positive integer; or~$\Q(\xi)$ is
a real quadratic field~$K$. For simplicity, we assume in this paper that~$\xi$
is a prime, and~$A$ has real multiplication by the maximal order~$\Z_K$ in the
latter case. These assumptions often hold in practice, and our techniques would
also apply with suitable modifications to more exotic cases.
Then~$\isog: A\to A'$ is an isogeny of one of the two following types.

\begin{defn}
  \label{def:beta-isog}
  Let~$k$ be a field, and let~$A, A'$ be p.p.~abelian surfaces over~$k$.
  \begin{enumerate}
  \item Let $\ell\in\Z_{\geq 0}$. An isogeny $\isog\from A\to A'$ is called an
    \emph{$\ell$-isogeny} if
    \begin{displaymath}
      \isog^* \Lb_{A'}(1) = \Lb_A({\ell}) \quad \text{in } \NS(A).
    \end{displaymath}
  \item Let~$K$ be a real quadratic field, and let $\beta\in\Z_K$ be a totally
    positive prime. Assume that~$A,A'$ have real multiplication by~$\Z_K$,
    given by embeddings~$\iota$ and~$\iota'$. An isogeny
    $\isog\from A\to A'$ is called a \emph{$\beta$-isogeny} if
    \begin{displaymath}
      \isog^* \Lb_{A'}(1) = \Lb_A({\,\iota(\beta)}) \quad\text{in }\NS(A)
    \end{displaymath}
    and the real multiplication embeddings $\iota$ and~$\iota'$ are compatible
    under~$\isog$, meaning that for all $\alpha\in\Z_K$, we have
    $ \isog\circ\iota(\alpha) = \iota'(\alpha)\circ\isog$.
  \end{enumerate}
\end{defn}

An~$\ell$-isogeny $\isog:A\to A'$ has degree~$\ell^2$; its kernel is a maximal
isotropic subgroup in the $\ell$-torsion subgroup~$A[\ell]$ for the Weil
pairing, and isomorphic to~$(\Z/\ell\Z)^2$ as an abstract group \cite[(1)
p.\,228 and Thm.~4 p.\,233]{mumfordAbelianVarieties1970}. In the real
multiplication case, $\beta$-isogenies are even smaller. The kernel of a
$\beta$-isogeny $\isog:A\to A'$ is maximal isotropic in~$A[\beta]$, thus
$\deg(\isog) = N_{K/\Q}(\beta)$, and~$\ker(\isog)$ is cyclic when the ideal
$(\beta)$ lies above a split prime in $K/\Q$.

Both~$\ell$- and $\beta$-isogenies are easily described over~$\C$. Up to
isomorphism, every $\ell$-isogeny is of the form
\begin{displaymath}
  A(\tau)\to A(\tau/\ell)
\end{displaymath}
(induced by the identity on~$\C^2$) for
some~$\tau\in
\Half_2$~\cite[Thm.~3.2]{brokerModularPolynomialsGenus2009}. Similarly, write
$t/\beta \defi \bigl(t_1/\beta, t_2/\conj{\beta}\bigr)$ for
$t = (t_1,t_2)\in\Half_1^2$. Then every $\beta$-isogeny is of the form
\begin{displaymath}
  \bigl(A_K(t),\iota_K(t)\bigr)\to \bigl(A_K(t/\beta),\iota_K(t/\beta)\bigr)
\end{displaymath}
for some choice of~$t$ \cite[Lem.~4.9]{martindaleHilbertModularPolynomials2020}.

\subsection{Modular equations}
\label{subsec:modpol}

Modular equations encode the presence of an isogeny between p.p.~abelian
surfaces, and generalize the classical modular polynomials that are widely used
to compute isogenies between elliptic curves.

In the Siegel case, let $\Gamma^0(\ell)\subset \Sp_4(\Z)$ be the subgroup
consisting of matrices whose upper right $2\times 2$ block is divisible
by~$\ell$, and consider the map
\begin{align*}
  \PHI_{\ell,\C}: \Gamma^0(\ell)\backslash \Half_2 &\to \Atwo(\C)\times \Atwo(\C)\\
  \tau &\mapsto (\tau, \tau/\ell).
\end{align*}
The map $\PHI_{\ell,\C}$ is the analytification of the map $\PPHI_\ell$
described in the introduction, which exists at the level of algebraic stacks
over~$\Q$. The Siegel modular equations are equations for the image
of~$\PHI_{\ell,\C}$ in $\C^3\times \C^3$ via the Igusa invariants; we consider
them as elements of~$\Q[J_1,J_2,J_3,J_1',J_2',J_3']$. Any such set of equations
would work in the context of the isogeny algorithm.  We can nonetheless define
the Siegel modular equations uniquely, using the fact that the extension of the
field $\C\bigl(j_1(\tau), j_2(\tau), j_3(\tau)\bigr)$ constructed by
adjoining~$j_1(\tau/\ell)$, $j_1(\tau/\ell)$, and~$j_3(\tau/\ell)$ is finite
and generated
by~$j_1(\tau/\ell)$~\cite[Lem.~4.2]{brokerModularPolynomialsGenus2009}.

\begin{defn}
  \label{def:modeq}
  Let $\ell$ be a prime. The \emph{Siegel modular equations of level~$\ell$}
  are the three following irreducible polynomials
  $\Psi_{\ell,1},\Psi_{\ell,2},\Psi_{\ell,3}\in \Q[J_1,J_2,J_3,J_1',J_2',J_3']$:
  \begin{itemize}
  \item $\Psi_{\ell,1} \in \Q[J_1,J_2,J_3,J_1']$ is the (non-monic) minimal
    polynomial of the function~$j_1(\tau/\ell)$ over
    $\C\bigl(j_1(\tau),j_2(\tau),j_3(\tau) \bigr)$.
  \item For $i\in\{2, 3\}$, we have
    $\Psi_{\ell,i}\in \Q[J_1,J_2,J_3,J_1',J_i']$, with
    $\deg_{J_i'}\Psi_{\ell,i} = 1$, and an equality of meromorphic functions
    \begin{displaymath}
      \Psi_{\ell,i}\bigl(j_1(\tau),j_2(\tau),j_3(\tau),j_1(\tau/\ell),j_i(\tau/\ell)\bigr) = 0.
    \end{displaymath}
  \end{itemize}
\end{defn}

In the Hilbert case, we let $\Gamma^0(\beta)\subset \Gamma_K$ be the subgroup
of matrices whose upper right entry $b$ lies in $\beta (\different)^{-1}$, and consider the map
\begin{align*}
  \PHI_{\beta,\C} : \Gamma^0(\beta)\backslash \Half_1^2 \to \Atwo(\C)\times \Atwo(\C)\\
  t \mapsto \bigl(H(t), H(t/\beta) \bigr).
\end{align*}
We call \emph{Hilbert modular equations of level~$\beta$} any set of three
irreducible polynomials $\Psi_{\beta,k}\in \Q[J_1,J_2,J_3,J_1',J_2',J_3']$ for
$1\leq k\leq 3$ which, together with the Humbert equation in~$\Q[J_1,J_2,J_3]$,
are equations for the image of $\PHI_{\beta,\C}$ in~$\C^3\times \C^3$ via the
Igusa invariants. One can adapt \cref{def:modeq} to also define the Hilbert
modular equations uniquely: see
\cite[Prop.~4.11]{milioModularPolynomialsIlbert2020} and
\cite[§3.2]{kiefferDegreeHeightEstimates2022}.

Since the Igusa invariants are symmetric by \Cref{prop:mf-pullback}, the
Hilbert modular equations encode~$\beta$- and $\conj{\beta}$-isogenies
simulaneously~\cite[Ex.~4.17]{milioModularPolynomialsIlbert2020}. It would be
better to consider modular equations in terms non-symmetric invariants;
however, we know of no explicit choice of such invariants in general.

From a practical point of view, modular equations in genus~$2$ are very large
polynomials. This is especially true for the Siegel modular equations of
level~$\ell$. For each~$1\leq k\leq 3$, the degree of~$\Psi_{\ell,k}$ in each
variable is~$O(\ell^3)$, and the height of the coefficients
is~$O(\ell^3\log\ell)$, for a total size
of~$O(\ell^{15}\log\ell)$~\cite{kiefferDegreeHeightEstimates2022}. The
situation is less desperate for Hilbert modular equations of level~$\beta$:
their total size is $O_K(\ell^{4} \log \ell)$ where~$\ell =
N_{K/Q}(\beta)$. Modular equations have only been computed in full (using
different invariants) up to $\ell = 7$ in the Siegel case, and up to
$N(\beta) = 97$ in the Hilbert case for
$K = \Q(\sqrt{2})$~\cite{milioDatabaseModularPolynomials2016}.

Luckily, directly evaluating modular equations and their derivatives at a given
point is much cheaper than writing them down in
full~\cite{kiefferEvaluatingModularEquations2022}: for example, over a prime
finite field~$\F_p$, the evaluation cost is only~$\Otilde(\ell^6\log p)$ and
$\Otilde(\ell^2\log p)$ binary operations for the Siegel and Hilbert modular
equations, respectively.  These evaluations are all we need to apply the
isogeny algorithm.

\section{Explicit Kodaira--Spencer over \texorpdfstring{$\C$}{C}}
\label{sec:cov}

In~§\ref{subsec:hyperelliptic}, we explain how a choice of genus~$2$ curve
equation~$\Crv_E\defby y^2 = E(x)$ over~$\C$ naturally encodes a basis~$\omega_E$ of
differential forms on the Jacobian of~$\Crv_E$. If~$f$ is a Siegel
modular form, this gives rise to a map
\begin{displaymath}
  \Cov(f)\defby E\mapsto f\bigl(\Jac(\Crv_E), \omega_E\bigr)
\end{displaymath}
Following~\cite{cleryCovariantsBinarySextics2017}, we show that~$\Cov(f)$ is a
polynomial in the coefficients of~$E$ in~§\ref{subsec:cov}. We describe an
algorithm to obtain this polynomial from the $q$-expansion of~$f$
in~§\ref{subsec:identification}, and apply it to the derivatives of the Igusa
invariants to obtain the explicit Kodaira--Spencer isomorphism.  This allows us
to compute the deformation map and the tangent map of a generic $\ell$-isogeny
over~$\C$ in~§\ref{subsec:norm-matrix}. Finally, we adapt these methods to the
Hilbert case in~§\ref{subsec:explicit-hilbert}.

\subsection{Genus~2 curve equations}
\label{subsec:hyperelliptic}

Let~$E\in \C_6[x]$ be a polynomial with six distinct roots
in~$\Pvar^1(\C)$ (hence $\deg(E)\in\{5,6\}$). We associate to~$E$ the
genus~$2$ curve
\begin{displaymath}
  \Crv_E: y^2 = E(x).
\end{displaymath}
We refer to~$E$ as a \emph{genus~$2$ curve equation}. Choosing~$E$ not only
specifies~$\Crv_E$ up to isomorphism: indeed,~$\Crv_E$ is also endowed with the
basis of differential forms
\begin{displaymath}
  \omega_E \defi \Bigl(\frac{x\,dx}{y},\frac{dx}{y}\Bigr).
\end{displaymath}
Any choice of base point $P$ on a genus~$2$ curve~$\Crv$ gives an embedding
%\begin{displaymath}
$\eta_P \from \Crv \embeds \Jac(\Crv)$
% \end{displaymath}
sending~$Q$ to the divisor class $[Q-P]$.
Then~$\eta_P^*: \Omega^1(\Jac(\Crv)) \to \Omega^1(\Crv)$ is an isomorphism and
is independent
of~$P$~\cite[Prop.~5.3]{milneJacobianVarieties1986}. Throughout, we identify
$\Omega^1(\Jac(\Crv))$ and~$\Omega^1(\Crv)$ via this isomorphism, so that we
may also view~$\omega_E$ as a basis of differential forms
on~$\Jac(\Crv_E)$. The following lemma (a simple calculation:
see~\cite[§4]{cleryCovariantsBinarySextics2017}) justifies why our choice
of~$\omega_E$ is convenient.

\begin{lem}
  \label{lem:hyperell-isomorphism}
  Let~$E$ be a genus~$2$ curve equation, and let
  %\begin{displaymath}
    $r = \tmat{a}{b}{c}{d} \in \GL_2(\C)$.
  %\end{displaymath}
  Let~$E' = \det^{-2}\Sym^6(r)\,E$, and let $\eta\from\Crv_{E'}\to\Crv_E$ be
  the isomorphism defined by
  \begin{displaymath}
    \eta(x,y) = \left( \dfrac{a x + c}{b x + d},\ \dfrac{(\det
        r)\,y}{(bx + d)^3} \right).
  \end{displaymath}
  Then the matrix of $\eta^*\from \Omega^1(\Crv_E)\to \Omega^1(\Crv_{E'})$
  in the bases~$\omega_E$ and~$\omega_{E'}$ is~$r$.
\end{lem}

By \Cref{lem:hyperell-isomorphism} and Torelli's theorem, if~$A$ is a
p.p.~abelian surface over~$\C$ that is not the product of two elliptic curves,
and if~$\omega$ be a basis of~$\Omega^1(A)$, then there exists a unique
genus~$2$ curve equation~$E$ such that the pairs
$\bigl(\Jac(\Crv_E),\omega_E \bigr)$ and~$(A,\omega)$ are isomorphic. We can
thus make the following definition.

\begin{defn}
  \label{def:std-curve}
  Let $\tau\in\Half_2$, and assume that $\chi_{10}(\tau)\neq 0$. We
  define~$E(\tau)$ to be the unique genus~$2$ curve equation such that
  \begin{displaymath}
    \bigl(\Jac(\Crv_{E(\tau)}), \omega_{E(\tau)} \bigr) \simeq
    \bigl(A(\tau),\omega(\tau) \bigr),
  \end{displaymath}
  and call it the \emph{standard curve equation} attached to~$\tau$. We define
  the meromorphic functions~$a_i(\tau)$ for $0\leq i\leq 6$ to be the
  coefficients of~$E(\tau)$:
  \begin{displaymath}
    E(\tau) = \sum_{i=0}^6 a_i(\tau) x^i.
  \end{displaymath}
\end{defn}

\begin{lem}
  \label{lem:E-as-smf}
  The function
  $\tau\mapsto E(\tau)$ is a vector-valued Siegel modular function of
  weight~$\det^{-2}\Sym^6$ which has no poles on the open set
  $\{\chi_{10}\neq 0\}$.
\end{lem}

\begin{proof}
  The function $\tau\mapsto E(\tau)$ is well-defined on~$\{\chi_{10}\neq 0\}$
  and is holomorphic on this open set. To prove the transformation rule,
  fix~$\gamma=\tmat{a}{b}{c}{d}\in \Sp_4(\Z)$ and ~$\tau\in \Half_2$ such that
  $\chi_{10}(\tau)\neq 0$. Let~$\eta:A(\tau)\to A(\gamma\tau)$ be the
  isomorphism $z\mapsto (c\tau+d)^{-t}z$. Then the matrix of
  $\eta^*: \Omega^1(A(\gamma\tau))\to \Omega^1(A(\tau))$ in the bases
  $\omega(\gamma\tau)$ and~$\omega(\tau)$ is~$(c\tau + d)^{-1}$. On the other
  hand, writing $E' = \det^{-2}\Sym^6(c\tau+d) E(\tau)$, we have an isomorphism
  $\eta': \Jac(\Crv_{E(\tau)}) \to \Jac(\Crv_{E'})$ such that the matrix
  of~$\eta'^*$ in the bases $\omega_{E'}$ and $\omega_{E(\tau)}$ is
  $(c\tau + d)^{-1}$ by \Cref{lem:hyperell-isomorphism}. Thus $E'$ satisfies
  the equality of \cref{def:std-curve} at~$\gamma\tau$, so
  $E(\gamma\tau) = E' = \det^{-2}\Sym^6(c\tau+d) E(\tau)$.
\end{proof}

\subsection{Covariants}
\label{subsec:cov}

Let~$f$ be a Siegel modular form of weight~$\rho$. The construction of
§\ref{subsec:hyperelliptic} yields an algebraic map
\begin{align*}
  \Cov(f): E &\mapsto f\bigl(\Jac(\Crv_E), \omega_E \bigr).
\end{align*}
The map~$\Cov(f)$ is then a \emph{covariant} of~$E$. These are classical
objects, studied in the 19th century by
Clebsch~\cite{clebschTheorieBinaerenAlgebraischen1872}. A more modern
reference for covariants is Mestre's
article~\cite{mestreConstructionCourbesGenre1991}. In light of
\Cref{lem:hyperell-isomorphism}, we use the following terminology.

\begin{defn}
  \label{def:cov}
  Let $\rho:\GL_2(\C)\to \GL(V)$ be a finite-dimensional holomorphic
  representation of~$\GL_2(\C)$ on a vector space~$V$. A \emph{fractional
    covariant} of weight~$\rho$ is a rational map
  %\begin{displaymath}
    $C\from \C_6[x] \to V$
  %\end{displaymath}
  that satisfies the following transformation rule: for all $r\in\GL_2(\C)$
  and $E\in\C_6[x]$,
  \begin{displaymath}
    C\bigl(\det\nolimits^{-2}\Sym^6(r) \,E\bigr) =
    \rho(r)\, C(E).
  \end{displaymath}
  If $\dim V\geq 2$, then~$C$ is said to be \emph{vector-valued}, and otherwise
  \emph{scalar-valued}.  A \emph{covariant} is a fractional covariant that is
  also a polynomial map.
\end{defn}

It is enough to consider covariants of weight~$\det^k\Sym^n$, for $k\in\Z$ and
$n\in\Z_{\geq 0}$. As in the case of Siegel modular forms, multiplication of
polynomials allows us to consider (fractional) covariants as elements of a
graded~$\C$-algebra.  What we call a vector-valued covariant of weight
$\det^k\Sym^n$ is in Mestre's paper a covariant of order~$n$ and degree
$k + n/2$; what we call a scalar-valued covariant of weight~$\det^k$ is in
Mestre's paper an invariant of degree~$k$.

A precise correspondence between Siegel modular forms and covariants is
established in~\cite{cleryCovariantsBinarySextics2017} by studying how modular
forms and covariants extend to the toroidal compactification of $\Atwo$. We
reformulate some of these results as follows.

\begin{thm}[{\cite[§4 and §6]{cleryCovariantsBinarySextics2017}}]
  \label{thm:cov-mf}
  The map~$f\mapsto \Cov(f)$ induces a weight-preserving bijection between the
  graded algebras of Siegel modular functions and fractional covariants. Its
  inverse bijection is
  \begin{displaymath}
    C \mapsto \bigl(f: \tau\mapsto C(E(\tau))\bigr).
  \end{displaymath}
  Further, if~$f$ is a Siegel modular form, then~$\Cov(f)$ is a covariant.
  If~$f$ is a cusp form, then~$\Cov(f/\chi_{10})$ is a also a covariant.
\end{thm}

A second key input is the structure of the graded algebra of covariants which,
unlike the graded algebra of Siegel modular forms, is finitely generated.

\begin{thm}[{\cite[p.\,296]{clebschTheorieBinaerenAlgebraischen1872}}]
  \label{thm:cov-structure}
  The graded $\C$-algebra of covariants is generated by~$26$ elements
  defined over~$\Q$. The number of generators of weight $\det^k\Sym^n$
  is indicated in the following table:
\end{thm}
\begin{center}
  \begin{tabular}{c|cccccccccccccccc}
    $n\ \backslash\ k$ & -3 & -2 & -1 & 0 & 1 & 2 & 3 & 4 & 5 & 6 & 7 & 8
    & 9 & 10 & 11 & 15 \\\hline
    0 & & & & & &1& &1& &1& & & &1& &1 \\
    2 & & & & & &1& &1& &1&1& &1& &1 \\
    4 & & & &1& &1&1& &1& & 1 \\
    6 & &1& &1&1& &2 \\
    8 & &1&1& &1 \\
    10& & &1 \\
    12&1 \\
  \end{tabular}
\end{center}

We will only manipulate a small number of these generators. Take our scalar
generators of even weight to be the Igusa--Clebsch invariants
$I_2, I_4, I_6, I_{10}$, in Mestre's notation $A',B',C',D'$, and set
\begin{displaymath}
  I_6' \defi (I_2 I_4 - 3 I_6)/2.
\end{displaymath}
Other generators can be computed
following~\cite[§1]{mestreConstructionCourbesGenre1991} (in this reference,
the integers~$m$ and~$n$ on page~315 should be the orders of~$f$ and~$g$, not
their degrees).  Denote the generator of weight~$\det^{15}$ by~$S$, and denote
by $y_1, y_2, y_3$ the generators of weights~$\det^2\Sym^2$, $\det^4\Sym^2$,
and~$\det^6\Sym^2$ respectively. Finally, the generator of
weight~$\det^{-2}\Sym^6$ is the degree 6 polynomial itself. To help the reader
check their computations, we mention that the coefficient of~$a_1^5 a_4^{10}$
in~$S$ is~$2^{-2} 3^{-6} 5^{-10}$.

\subsection{From \texorpdfstring{$q$}{q}-expansions to covariants}
\label{subsec:identification}

We now explain how to compute the polynomial covariant associated with a Siegel
modular form of known $q$-expansion. The works of Igusa already provide the answer
in the scalar-valued case.

\begin{thm}
  \label{thm:scalar-identification}
  We have
  \begin{displaymath}
    \begin{aligned}
      4\, \Cov(\psi_4) &= I_4, \qquad\qquad\ \;
      4\, \Cov(\psi_6) = I_6', \\
      -2^{12}\, \Cov(\chi_{10}) &= I_{10},\qquad\quad
      2^{15} \,\Cov(\chi_{12}) = I_2 I_{10}, \\
      2^{32} 3^{-9} 5^{-10} \Cov(\chi_{35}) &= I_{10}^2 S.
    \end{aligned}
  \end{displaymath}
\end{thm}

\begin{proof}
  By~\cite[p.\,848]{igusaModularFormsProjective1967}, there exists a constant
  $\lambda\in\C^\times$ such that these relations hold up to a
  factor~$\lambda^k$, for $k\in\{4,6,10,12,35\}$ respectively. (Note that
  Igusa's covariant~$E$ is $2^5 3^9 5^{10} S$.) To determine~$\lambda$, we
  apply Thomae's formula~\cite[Thm.~IIIa.8.1]{mumfordTataLecturesTheta1984} on
  the genus~$2$ curve\footnote{The even more obvious choice
    $y^2 = \prod_{j=1}^6 (x-j)$ has a vanishing~$S$.}
  \begin{displaymath}
    \Crv_E: y^2 = E(x) =  x\prod_{j=2}^6 (x-j)
  \end{displaymath}
  whose Weierstrass points are ordered in the obvious way.
  Let~$\tau\in \Half_2$ be a period matrix of~$\Jac(\Crv_E)$, choose an
  isomorphism~$\eta:\Jac(\Crv_E)\to A(\tau)$, and let~$\sigma$ be the matrix
  of~$\eta^*$ in the bases~$\omega(\tau)$
  and~$\omega$. By~\cite[Thm.~IIIa.8.1]{mumfordTataLecturesTheta1984}, up to a
  common factor~$\mu\in \C^\times$ with~$\mu^2 = \det(\sigma)$, the ten even
  theta constants at~$\tau$ are
  % Values of theta^8: 230400, 104976, 82944, 57600, 20736, 3600, 32400, 9216, 9216, 144
  \begin{displaymath}
    % This is for y^2 = \prod(x-i) which has a vanishing R!
    %4,\ \sqrt{6}\cdot 5^{1/4} \text{ (twice)},\ 2\sqrt{3},\ 2^{3/4}\sqrt{3}\text{ (twice)},
    %\ 2^{3/4}5^{1/4} \text{ (twice)},\ \sqrt{6},\ \sqrt{2}.
    2\sqrt[4]{30},\ 3\sqrt{2},\ 2\sqrt[4]{18},\ 2\sqrt[4]{15},\ 2\sqrt{3},\
    \sqrt[4]{60},\ \sqrt[4]{180},\ 2\sqrt[4]{6} \text{ (twice)},\ \sqrt[4]{12}.
  \end{displaymath}
  (The correct roots of unity can be computed by noticing that these values are
  positive real numbers \cite[pp.\,216--217]{thomaeBeitragZurEstimmung1870},
  or by analytic computations as in \cref{rem:vector-identification-numcheck}
  below.)  Using the formulas from \cite[§7.1]{strengComputingIgusaClass2014}
  and~\cite[p.\,493]{bolzaDarstellungRationalenGanzen1887}, the values of the
  modular forms~$\psi_4,\ldots,\chi_{35}$ at~$\tau$ are
  \begin{align*}
    \psi_4(\tau) &= 345168\det(\sigma)^4, \qquad\qquad\quad
                   \psi_6(\tau) = 78382080\det(\sigma)^6,\\
    \chi_{10}(\tau) &= -128595600\det(\sigma)^{10},
                      \qquad \chi_{12}(\tau) = 129720811500 \det(\sigma)^{12},\\
    \chi_{35}(\tau) &=  57046688433310783937336006400000\det(\sigma)^{35}.
  \end{align*}
  On the other hand, using the formulas
  in~\cite{mestreConstructionCourbesGenre1991}, we obtain
  \begin{align*}
    I_4(E) &= 1380672, \qquad\qquad\qquad I_6'(E) = 313528320,\\
    I_{10}(E) &= 526727577600, \qquad\ I_2I_{10}(E) = 4250691551232000,\\
    I_{10}^2S(E) &= 3983354751469532799105506450866176/3125.
  \end{align*}
  Thus~$\lambda^4 = \lambda^6 = \lambda^{10} = \lambda^{12} = \lambda^{35} =
  1$, hence~$\lambda= 1$.
\end{proof}

Therefore, the Igusa invariants satisfy, in accordance
with~\cite[§2.1]{strengComputingIgusaClass2014}:
\begin{displaymath}
  \Cov(j_1) = \dfrac{I_4 I_6'}{I_{10}},\quad \Cov(j_2) = \dfrac{I_2I_4^2}{I_{10}},\quad
  \Cov(j_3) = \dfrac{I_4^5}{I_{10}^2}.
\end{displaymath}
In order to obtain similar formulas for vector-valued modular forms, we first
compute the $q$-expansion of the standard curve $\Crv(\tau)$ from
\Cref{def:std-curve}.

\begin{prop}
  \label{prop:f86}
  The following equality of Siegel modular functions holds:
  \begin{displaymath}
    \Crv(\tau) = \frac{\chi_{6,8}(\tau)}{\chi_{10}(\tau)}.
  \end{displaymath}
\end{prop}

\begin{proof}
  The modular form~$\chi_{6,8}$ introduced in §\ref{subsec:siegel-g2} is a cusp
  form. By \Cref{thm:cov-mf}, $\Cov(\chi_{6,8}/\chi_{10})$ is a covariant of
  weight~$\det^{-2}\Sym^6$, and this space of covariants is $1$-dimensional by
  \Cref{thm:cov-structure}. Therefore, the claimed equality holds up to a
  certain factor~$\lambda\in\C^\times$. This yields $q$-expansions for the
  coefficients~$a_i(\tau)$ of~$\Crv(\tau)$ up to a factor~$\lambda$. Then,
  \cref{thm:scalar-identification} implies that
  $\lambda^4=\lambda^6=\lambda^{35}=1$, hence~$\lambda=1$.
\end{proof}

\Cref{prop:f86} improves slightly on
\cite[§6]{cleryCovariantsBinarySextics2017} (which follows the same proof
strategy) in that we determine the correct scalar factor.

Given a Siegel modular form~$f$ of weight~$\rho$ whose $q$-expansion
can be computed, the following algorithm now recovers the expression
of~$\Cov(f)$ as a polynomial.

\begin{algo}~
  \label{alg:qexp-to-cov}
  \begin{enumerate}
  \item Compute a generating family for the vector space of
    polynomial covariants of weight~$\rho$ using
    \Cref{thm:cov-structure}, and extract a basis~$\mathcal{B}$ using the
      embedding into~$\C[a_0,\dots,a_6]$.
  \item \label{step:qexp} Choose a precision~$\prc$ and compute the
    $q$-expansion of~$f$ modulo $(q_1^\prc, q_3^\prc)$.
  \item For every $B\in \mathcal{B}$, compute the $q$-expansion of the Siegel
    modular function $\tau\mapsto B\bigl(\Crv(\tau)\bigr)$ modulo
    $(q_1^\prc, q_3^\prc)$ using \Cref{prop:f86}.
  \item Solve a linear system to write~$\Cov(f)$ as a linear combination of the
    elements of~$\mathcal{B}$; if the matrix does not have full rank, go back
    to step~\ref*{step:qexp} with a larger~$\prc$.
  \end{enumerate}
\end{algo}

We now apply \Cref{alg:qexp-to-cov} to the derivatives of the Igusa invariants,
denoted by~$Dj_k$ for $1\leq k\leq 3$ following the notation
of~§\ref{subsec:siegel}.

\begin{thm}
  \label{thm:vector-identification}
  We have
  \begin{align*}
    \Cov(Dj_1)
    &=
    %   \dfrac{1}{I_{10}}\Bigl( \frac{153}{8} I_2^2 I_4 y_1 -
    %   \dfrac{135}{2} I_2 I_6 y_1 + \dfrac{135}{2} I_4^2 y_1 +
    %   \dfrac{46575}{4} I_2 I_4 y_2\\
    % &\qquad\qquad - 30375\, I_6 y_2 + 1366875\, I_4 y_3
    % \Bigr), \\
    \dfrac{1}{8I_{10}}\bigl(153\, I_2^2 I_4 y_1 -
      540\, I_2 I_6 y_1 + 540\, I_4^2 y_1 +
      93150\, I_2 I_4 y_2\\
    &\qquad\qquad - 243000\, I_6 y_2 + 10935000\, I_4 y_3
    \bigr), \\
    \Cov(Dj_2)
    &=
      \dfrac{1}{I_{10}}\bigl( 90\, I_2^2 I_4 y_1 + 900\, I_2^2 y_1 +
      40500\, I_2 I_4 y_2
      \bigr), \\
    \Cov(Dj_3)
    &=
      \dfrac{1}{I_{10}^2}\bigl( 225\, I_2 I_4^4 y_1 + 101250\, I_4^4
      y_2 \bigr).
  \end{align*}
\end{thm}

\begin{proof}
  Let $1\leq k\leq 3$. The function~$\chi_{10}^2 j_k$ has no poles
  on~$\Atwo(\C)$, so
  %\begin{displaymath}
    $f_k \defi \chi_{10}^3\, Dj_k$
  %\end{displaymath}
  is a Siegel modular form. Its $q$-expansion can be computed from
  the $q$-expansion of~$j_k$ by formal differentiation. Since
  \begin{displaymath}
     \frac{1}{2\pi i}\dfrac{\partial}{\partial \tau_l} = q_l \frac{\partial}{\partial q_l}
  \end{displaymath}
  for $1\leq l\leq 3$, we check that~$f_k$ is a cusp form. By
  \Cref{thm:cov-mf}, $\Cov(f_k/\chi_{10})$ is a polynomial covariant of
  weight~$\det^{20}\Sym^2$, and by \Cref{thm:cov-structure}, a basis of this
  space of covariants is given by covariants of the form~$I y$ where
  $y\in\{y_1,y_2,y_3\}$ and~$I$ is a scalar-valued covariant of the appropriate
  even weight. \Cref{alg:qexp-to-cov} succeeds with $\prc=3$; the computations
  were done using Pari/GP~\cite{theparigroupPariGPVersion2019}.
\end{proof}

\begin{rem}
  \label{rem:vector-identification-numcheck}
  \Cref{thm:scalar-identification,thm:vector-identification} can be checked
  numerically. Computing big period matrices of genus~$2$ curves (see for
  instance~\cite{molinComputingPeriodMatrices2019}) provides pairs
  $\bigl(\tau,\, \Crv(\tau)\bigr)$ with $\tau\in\Half_2$. We can evaluate the Igusa
  invariants at a given $\tau$ to high precision using their expression in
  terms of theta
  functions~\cite{labrandeComputingThetaFunctions2016}. Therefore, we can also
  evaluate their derivatives numerically with high precision and compute the
  associated covariant using floating-point linear algebra. We used the
  libraries \texttt{hcperiods}~\cite{molinHcperiodsPeriodMatrices2018} and
  \texttt{cmh}~\cite{engeCMHComputationGenus2014} for these computations.
\end{rem}

Using \Cref{thm:vector-identification} and linear algebra, one can obtain
similar formulas for the derivatives of other invariants such as the
invariants~$h_k$ defined in \cref{rem:invariants_bielliptic}.

\subsection{Deformation matrix and action on tangent spaces}
\label{subsec:norm-matrix}

Let $E$ and~$F$ be genus~$2$ curve equations over~$\C$, let $A$ and~$A'$ be
the Jacobians of~$\Crv_E$ and~$\Crv_{F}$, and let $\isog\from A\to A'$ be an
$\ell$-isogeny. Taking the dual bases of~$\omega_E$ and~$\omega_{F}$ defines
bases of the tangent spaces~$T_0(A)$ and~$T_0(A')$. If the pair~$(A,A')$ is
sufficiently generic, then there exists only one~$\ell$-isogeny $\isog:A\to A'$
up to sign, and we show how to compute, up to sign, the matrix of the tangent map
$d\isog\from T_0(A)\to T_0(A')$ in the above bases from the data
of the curve equations and modular equations of level~$\ell$. First, we
introduce the following matrix notations.

\begin{defn}
  \label{def:djdtau}
  For $\tau\in\Half_2$, we define
  \begin{displaymath}
    \djdtau(\tau) \defi \left(\dfrac{1}{2\pi i}\dfrac{\partial
        j_k}{\partial\tau_l}(\tau)\right)_{1\leq k, l\leq 3}
    \cdot \left(
      \begin{matrix}
        2&0&0\\0&1&0\\0&0&2
      \end{matrix} \right).
  \end{displaymath}
  In other words, if we set 
  \begin{displaymath}
    v_1 = \mat{2}{0}{0}{0},\quad v_2 = \mat{0}{1}{1}{0},\quad
    v_3 = \mat{0}{0}{0}{2},
  \end{displaymath}
  then for each $1\leq l\leq 3$, the $l$-th column of~$\djdtau(\tau)$ contains
  (up to dividing by~$2\pi i$) the derivatives of the Igusa invariants at~$\tau$ in
  the direction~$v_l$.
\end{defn}

\hspace{-2pt}The next two lemmas summarize the properties of the matrix-valued function~$DJ$.

\begin{lem}
  \label{lem:sym2-DJ}
  Let~$\tau\in \Half_2$ be a point where the Igusa invariants are defined, and
  let~$r\in \GL_2(\C)$. Then the columns of~$DJ(\tau)\Sym^2(r)$ contain the
  derivatives of the three Igusa invariants at~$\tau$ in the
  directions~$r v_l r^t$ for~$1\leq l\leq 3$, divided by~$2\pi i$.
\end{lem}

\begin{proof}
  This relation comes from the fact that the representation of~$\GL_2(\C)$ on
  the space of symmetric $2\times2$ matrices for which~$r$ acts
  by~$v\mapsto r v r^t$ is isomorphic to~$\Sym^2$. Here we check it by a direct
  calculation.  Write $r = \tmat{a}{b}{c}{d}$. We have
  \begin{align*}
    r v_1 r^t &= a^2v_1 + 2acv_2 + c^2v_3,\\
    r v_2 r^t &= ab v_1 + (ad+bc)v_2 + cd v_3,\\
    r v_3 r^t &= b^2 v_1 + 2bd v_2 + d^2 v_3.
  \end{align*}
  This matches the entries of the matrix~$\Sym^2(r)$ defined
  in~§\ref{subsec:siegel-g2}.
\end{proof}

\begin{lem}
  \label{lem:weight-DJ}
  Let~$\rho$ be the representation of~$\GL_2(\C)$ on~$V = \Mat_{3\times 3}(\C)$
  given by
  \begin{displaymath}
    \rho(r): M\mapsto M \Sym^2(r^t), \quad\text{for all } r\in \GL_2(\C).
  \end{displaymath}
  Then~$DJ$ is a vector-valued Siegel modular function on~$V$ of weight~$\rho$.
\end{lem}

\begin{proof}
  We know that for each~$1\leq k\leq 3$, the function~$Dj_k$ is a vector-valued
  modular function of weight~$\Sym^2$ as defined
  in~§\ref{subsec:siegel-g2}. Hence each column of the matrix $DJ(\tau)^t$ is a
  vector-valued modular form for the representation
  \begin{displaymath}
    \rho: r\mapsto \Diag(2,1,2) \Sym^2(r) \Diag(2,1,2)^{-1} = \Sym^2(r^t)^t,
  \end{displaymath}
  and the conclusion follows by transposing.
\end{proof}

We also denote by $\Cov(\djdtau)$ the associated ``matrix-valued'' fractional
covariant. For a given curve equation~$E$, \Cref{thm:vector-identification}
expresses the entries of the $3\times 3$ matrix $\Cov(\djdtau)(E)$ in terms of
the coefficients of~$E$.

\begin{defn}
  \label{defn:modeq-matrix}
  Consider the Siegel modular
  equations~$\Psi_{\ell,1}, \Psi_{\ell,2}, \Psi_{\ell,3}$ of
  level~$\ell$ as elements of the ring
  $\Q[J_1,J_2,J_3,J_1',J_2',J_3']$. We define
  \begin{displaymath}
    D\Psi_{\ell,L} \defi
    \left(\dfrac{\partial\Psi_{\ell,n}}{\partial J_k}\right)_{1\leq n,
      k\leq 3}
    \quad\text{and} \quad
    D\Psi_{\ell,R} \defi \left(\dfrac{\partial\Psi_{\ell,n}}{\partial J_k'}\right)_{1\leq n,
      k\leq 3}.
  \end{displaymath}
  They are $3\times 3$ matrices with coefficients in $\Q[J_1,J_2,J_3,J_1',J_2',J_3']$.
\end{defn}

With these notations in place, we can define what a generic isogeny is in the
context of \Cref{thm:main}, and define its attached deformation
matrix~$\defo(\isog)$.

\begin{defn}
  \label{def:generic}
  Let~$\isog:A\to A'$ be an $\ell$-isogeny as above. Write~$j$ as a shorthand
  for the Igusa invariants $(j_1,j_2,j_3)$ of~$A$, and~$j'$ for the Igusa invariants
  $(j_1',j_2',j_3')$ of~$A'$. We say that~$(A,A')$ is \emph{generic}, or
  that~$\isog$ is generic, when the complex $3\times 3$ matrices
  $D\Psi_{\ell,L}(j,j')$, $D\Psi_{\ell,R}(j,j')$, $\Cov(\djdtau)(E)$ and
  $\Cov(\djdtau)(F)$ are invertible. In this case, we define the
  \emph{deformation matrix}~$\defo(\isog)$ of~$\isog$ as
  \begin{displaymath}
    \defo(\isog) \defi - \Cov(\djdtau)(F)^{-1}\cdot D\Psi_{\ell,R}(j, j')^{-1}
    \cdot  D\Psi_{\ell,L}(j, j') \cdot \Cov(\djdtau)(E).
  \end{displaymath}
\end{defn}

The deformation matrix~$\defo(\isog)$ has a geometric interpretation that we
detail in~\cref{sec:moduli}: if~$x,x'$ are the points of~$\AAtwo$ corresponding
to~$A,A'$, then~$\defo(\isog)$ is the matrix of the deformation map of~$\isog$
in the bases of~$\tangent{x}{\AAtwo}$ and~$\tangent{x'}{\AAtwo}$ associated
with~$\omega_E$ and~$\omega_{F}$ via the Kodaira--Spencer isomorphism.

Now we can relate the deformation matrix~$\defo(\isog)$ to the tangent
map~$d\isog$, also identified with its matrix in the specified bases
of~$T_0(A)$ and~$T_0(A')$.

\begin{prop} 
  \label{prop:norm-matrix}
  With the above notation, assume that~$(A,A')$ is generic. Then there exists
  only one $\ell$-isogeny $\isog: A\to A'$ up to sign, and we have
  \begin{displaymath}
    \Sym^2(d\isog) = \ell\, \defo(\isog).
  \end{displaymath}
\end{prop}

\begin{proof} Choose $\tau\in\Half_2$ and isomorphisms $\eta, \eta'$ giving a
  commutative diagram
  \begin{displaymath}
    \begin{tikzcd}
      A \ar[r,"\isog"] \ar[d, "\eta"]& A'
      \ar[d, "\eta'"]
      \\ A(\tau) \ar[r, "z\,\mapsto z"] & A(\tau/\ell).
    \end{tikzcd}
  \end{displaymath}
  Let~$r$ be the matrix of~$\eta^*$ in the bases $\omega(\tau)$
  and~$\omega_E$, and define~$r'$ similarly.  Then we have
  $d\isog = r'^\tp r^{-\tp}$. By the definition of modular equations, we have
  \begin{displaymath}
    \Psi_{\ell,k}\bigl(j_1(\tau), j_2(\tau), j_3(\tau), j_1(\tau/\ell),
    j_2(\tau/\ell), j_3(\tau/\ell)\bigr) = 0\quad \text{for } 1\leq k\leq 3.
  \end{displaymath}
  We differentiate these equalities with respect to $\tau_i$ for
  $1\leq i\leq 3$. This yields
  \begin{displaymath}
    \sum_{n=1}^3 \frac{\partial\Psi_{\ell,k}}{\partial J_n}(j,j') \frac{\partial j_n}{\partial \tau_i}(\tau) + \frac{1}{\ell} \sum_{n=1}^3 \frac{\partial \Psi_{\ell,k}}{\partial J_n'}(j,j') \frac{\partial j_n}{\partial\tau_i}(\tau/\ell) = 0
  \end{displaymath}
  for all $1\leq i,k\leq 3$, which corresponds to the coefficient $(k,i)$ of
  the matrix relation
  \begin{displaymath}
    D\Psi_{\ell,L}(j, j') \cdot \djdtau(\tau) + \dfrac{1}{\ell} D\Psi_{\ell,R}(j, j')
    \cdot \djdtau(\tau/\ell) = 0.
  \end{displaymath}
  We rewrite this last relation as
  \begin{displaymath}
    -\ell \, D\Psi_{\ell,L}(j, j') \cdot \Cov(\djdtau)(E) \cdot \Sym^2(r^\tp)
    =  D\Psi_{\ell,R}(j, j')
    \cdot \Cov(\djdtau)(F) \cdot \Sym^2(r'^\tp),
  \end{displaymath}
  and the expression of~$\Sym^2(d\isog)$ follows.

  This determines~$d\isog$ up to sign, so~$\pm \isog$ are the only
  $\ell$-isogenies from~$A$ to~$A'$, as all isogenies in characteristic zero
  are separable.
\end{proof}

\subsection{The Hilbert case}
\label{subsec:explicit-hilbert}

We now adapt our methods to recover the tangent matrix of a generic isogeny in
the Hilbert case, for any real multiplication field~$K$. If the attached ring
of Hilbert modular forms is known, several improvements to this general
strategy can be made: see \cref{sec:Qr5} for the case~$K = \Q(\sqrt{5})$.

A crucial difference with the Siegel case is that we cannot directly compute
the tangent matrix of a $\beta$-isogeny, where~$\beta\in \Z_K$ is a totally
positive prime, from an arbitrary choice of curve equations attached to~$A$ and~$A'$:
the real multiplication embedding has to play a role. The convenient notion for
us will be the following.

\begin{defn}
  \label{def:hilbert-normalized}
  Let~$(A,\iota)$ be a p.p.~abelian surface with real multiplication
  by~$\Z_K$. We say that a basis~$\omega$ of~$\Omega^1(A)$ is
  \emph{Hilbert-normalized} if for every~$\alpha\in \Z_K$, the matrix of
  $\iota(\alpha)^*: \Omega^1(A)\to\Omega^1(A)$ in the basis~$\omega$
  is~$\Diag(\alpha,\conj{\alpha})$. We say that a genus~$2$ curve equation~$E$
  such that $A = \Jac(\Crv_E)$ is \emph{Hilbert-normalized} if~$\omega_E$ is.
\end{defn}

In other words, a basis~$\omega$ of~$\Omega^1(A)$ is Hilbert-normalized if and
only if its dual basis consists of eigenvectors for the action of~$\Z_K$
on~$\tangent{0}{A}$. Hilbert-normalized bases are the right notion to consider
in the context of evaluating a Hilbert modular form on a pair~$(A,\omega)$, in
analogy with covariants in the Siegel case: we refer to \cref{sec:Qr5} for a
detailed discussion.

For the moment, assume that we have a $\beta$-isogeny
$\isog\from (A,\iota)\to (A',\iota')$ between abelian surfaces with real
multiplication by~$\Z_K$, and that we are given Hilbert-normalized curve
equations~$E$ and~$F$. We use the notation $D\Psi_{\beta,L}$
and~$D\Psi_{\beta,R}$ in the Hilbert case in
analogy with \cref{defn:modeq-matrix}. We also write
\begin{displaymath}
  T \defi \left(
    \begin{matrix}
      1&0 \\ 0&0 \\ 0&1
    \end{matrix}
  \right).
\end{displaymath}

\begin{lem}
  \label{lem:dt-dtau}
  Let~$E$ be a genus~$2$ curve equation such that~$\Jac(\Crv_E)$ has real
  multiplication by~$\Z_K$. Choose an isomorphism
  $\eta:\Jac(\Crv_E) \to A_K(t)$ for some $t\in \Half_1^2$, and let
  $r\in \GL_2(\C)$ be the matrix
  of~$\eta^*:\Omega^1(A_K(t))\to \Omega^1(\Jac(\Crv_E))$ in the bases
  $\omega_K(t)$ and $\omega_E$. Finally, let~$\tau = H(t)$. Then we have
  \begin{displaymath}
    \Cov(DJ)(E) = DJ(\tau) \Sym^2(R^t r^t).
  \end{displaymath}
  In other words, by \cref{lem:sym2-DJ}, the columns of~$\Cov(DJ)(E)$ contain
  the derivatives of the Igusa invariants at~$\tau$ in the directions
  \begin{displaymath}
    \frac{1}{\pi i} R^t r^t \mat{1}{0}{0}{0} r R,
    \quad \frac{1}{2\pi i} R^t r^t \mat{0}{1}{1}{0} r R \quad\text{and}\quad
    \frac{1}{\pi i} R^t r^t \mat{0}{0}{0}{1} r R.    
  \end{displaymath}
\end{lem}

\begin{proof}
  Let~$\zeta: A_K(t)\to A(\tau)$ be the isomorphism induced by left
  multiplication by~$R^t$ on~$\C^2$. The matrix of~$\zeta^*$ in the bases
  $\omega(\tau)$ and~$\omega_K(t)$ is~$R$, so the action
  of~$(\zeta \circ \eta)^*$ on differential forms is given by the matrix~$r
  R$. The conclusion follows from the definition of covariants and
  \cref{lem:weight-DJ}.
\end{proof}

\begin{prop} 
  \label{prop:norm-matrix-hilbert}
  Let $\isog\from A\to A'$ be a $\beta$-isogeny and $E, F$ be
  Hilbert-normalized curve equations as above. Then the tangent matrix~$d\isog$
  is diagonal, and we have
  \begin{align*}
    D\Psi_{\beta,L}(j, j') \cdot \Cov(\djdtau)(E) \cdot T \Diag(\beta, \conj{\beta})
    = - D\Psi_{\beta,R}(j,j')\cdot \Cov(\djdtau)(F) \cdot T\, (d\isog)^{2}.
  \end{align*}
\end{prop}

\begin{proof}
  Choose $t\in\Half_1^2$ and isomorphisms $\eta, \eta'$ giving a commutative
  diagram
  \begin{displaymath}
    \begin{tikzcd}
      \bigl(A,\iota\bigr) \ar[r,"\isog"] \ar[d, "\eta"] &
      \bigl(A',\iota'\bigr) \ar[d, "\eta'"] \\
      \bigl(A_K(t),\iota_K(t)\bigr) \ar[r, "z\mapsto z"] &
      \bigr(A_K(t/\beta),\iota_K(t/\beta)\bigr).
    \end{tikzcd}
  \end{displaymath}
  Let~$r$ be the matrix of~$\eta^*$ in the bases $\omega_K(t),\omega$, and
  define~$r'$ similarly; they are diagonal.  We have
  $d\isog = r'^\tp r^{-\tp} = r'r^{-1}$. We differentiate the modular equations
  \begin{displaymath}
    \Psi_{\beta,k}\bigl(j_1(H(t)),j_2(H(t)),j_3(H(t)),
    j_1(H(t/\beta)), j_2(H(t/\beta)), j_3(H(t/\beta))\bigr) = 0
  \end{displaymath}
  with respect to $t\in \Half_1^2$. Using \Cref{lem:dt-dtau}, the resulting
  equality can be written as
  \begin{align*}
    &D\Psi_{\beta,L}(j, j') \cdot \Cov(\djdtau)(E) \cdot \Sym^2(r^t) \cdot T\\
    &\quad + D\Psi_{\beta,R}(j,j') \cdot \Cov(\djdtau)(F)
      \cdot \Sym^2(r'^t) \cdot T\cdot  \Diag(1/\beta,1/\conj{\beta}) = 0.    
  \end{align*}
  We can reorganize this equality into the claimed result as~$r$ and~$r'$ are diagonal.
\end{proof}

In view of \cref{prop:norm-matrix-hilbert}, we say that the pair $(A,A')$ is
\emph{generic} if the $3\times 2$ matrices
$D\Psi_{\beta,L}(j, j') \cdot \Cov(\djdtau)(E) \cdot T$ and
$D\Psi_{\beta,R}(j,j')\cdot \Cov(\djdtau)(F) \cdot T$ have rank~$2$. In this
case, we can indeed recover~$(d\isog)^2$ from the derivatives of modular
equations. However, in contrast with the Siegel case, we obtain two possible
candidates for~$\pm d\isog$ as we have to extract two uncorrelated square
roots.

We now address the question of constructing a Hilbert-normalized curve equation
from the input of the Igusa invariants $(j_1,j_2,j_3)$ of a p.p.~abelian
surface~$(A,\iota)$ with real multiplication by~$\Z_K$. Note that we are
missing some information, as the two pairs $(A,\iota)$ and $(A,\overline{\iota})$,
where~$\overline{\iota}$ denotes the real conjugate of~$\iota$, have the same Igusa
invariants. The best we can hope for is thus to obtain a \emph{potentially
  Hilbert-normalized} curve in the following sense.

\begin{defn}
  We say that a genus~$2$ curve equation~$E$ is \emph{potentially
    Hilbert-normalized} if there exists a real multiplication embedding
  $\iota:\Z_K\embeds \End^\dagger(\Jac(\Crv_E))$ such
  that~$(\Jac(\Crv_E),\iota,\omega_E)$ is Hilbert-normalized.
\end{defn}

Generically, we can use the derivatives of the Igusa invariants to characterize
potentially Hilbert-normalized curve equations.

\begin{prop}
  \label{prop:tangent-humbert}
  Let~$E$ be a genus~$2$ curve equation such that~$\Jac(\Crv_E)$ has real
  multiplication by~$\Z_K$. Let~$(j_1,j_2,j_3)$ denote its Igusa invariants,
  and assume that the matrix $\Cov(\djdtau)(E)$ is invertible. Then~$E$ is
  potentially Hilbert-normalized if and only if the two columns of
  the~$3\times 2$ matrix $\Cov(\djdtau)(E)\cdot T$ are tangent vectors to the
  Humbert surface at $(j_1,j_2,j_3)$.
\end{prop}

\begin{proof}
  Let $t, \tau, \eta$ and~$r$ be as in
  \cref{lem:dt-dtau}. Since~$\Cov(\djdtau)(E)$ is invertible, the directions
  \begin{displaymath}
    R^t r \mat{1}{0}{0}{0} r^{\tp} R \quad\text{and}\quad
    R^t r \mat{0}{0}{0}{1} r^{\tp} R
  \end{displaymath}
  are tangent to the Humbert surface at~$\tau$ (i.e.~lie inside the image
  of~$\Half_1^2$ by the Hilbert embedding) if and only if the two columns
  $\Cov(\djdtau)(E)\cdot T$ are tangent to the algebraic Humbert surface at
  $(j_1,j_2,j_3)$. By the expression of the Hilbert embedding, this happens if
  and only if both $r (\begin{smallmatrix} 1&0\\0&0 \end{smallmatrix})r^\tp$
  and $r (\begin{smallmatrix} 0&0\\0&1 \end{smallmatrix})r^\tp$ are
  diagonal. This is equivalent to saying that~$r$ is is either diagonal or
  anti-diagonal, in other words~$E$ is potentially Hilbert-normalized.
\end{proof}

Assume that we are given the equation of the Humbert surface for~$K$ in terms of the Igusa
invariants: this precomputation depends only on~$K$.  Given a tuple of
Igusa invariants $(j_1,j_2,j_3)$ on the Humbert surface such that the
genericity condition of \cref{prop:tangent-humbert} is satisfied, the following
algorithm reconstructs a potentially Hilbert-normalized curve equation; its
correctness follows from \cref{lem:hyperell-isomorphism}.

\begin{algo}~
  \label{algo:hilb-curve-2}
  \begin{enumerate}
  \item Construct a curve equation~$E_0$ such that $\Jac(\Crv_{E_0})$ has
    Igusa invariants $(j_1,j_2,j_3)$ using Mestre's
    algorithm~\cite{mestreConstructionCourbesGenre1991}.
  \item \label{step:tangent-humbert} Find $r\in \GL_2(\C)$ such
    that the two columns of the matrix
    \begin{displaymath}
      \Cov(\djdtau)(E_0)\cdot \Sym^2(r^\tp)\cdot T
    \end{displaymath}
    are tangent to the Humbert surface at $(j_1,j_2,j_3)$.
  \item Output $\det^{-2}\Sym^6(r)\,E_0$.
  \end{enumerate}
\end{algo}

In step~\ref*{step:tangent-humbert}, if $a,b,c,d$ denote the entries of~$r$, we
only have to solve a quadratic equation in $a,c$, and a quadratic equation in
$b, d$. Therefore, \Cref{algo:hilb-curve-2} costs~$O_K(1)$ field operations
and~$O(1)$ square roots.

In practice, when computing a $\beta$-isogeny $\isog\from A\to A'$ in the
Hilbert case, we are only given the Igusa invariants of~$A$ and~$A'$, or
possibly a genus~$2$ curve equation. Constructing potentially
Hilbert-normalized curve equations $E,F$ then amounts to making a choice
of real multiplication embedding for each abelian surface (namely, the
embeddings for which~$E$ and~$F$ are Hilbert-normalized). If these embeddings
are incompatible via~$\isog$, we obtain antidiagonal matrices when attempting
to compute the tangent matrix with \cref{prop:norm-matrix-hilbert}; in this
case, we apply the change of variables $x\mapsto 1/x$ on~$E$ or~$F$ to make
them compatible. After that, $\isog$ will be either a $\beta$- or a
$\conj{\beta}$-isogeny depending on the choices of real multiplication
embeddings. In total, we obtain four possible candidates for the tangent matrix
up to sign.

\section{Moduli spaces and the deformation map}
\label{sec:moduli}

In this section, we use the language of moduli stacks to give an algebraic
interpretation of the results in \cref{sec:cov} and to generalize them to
isogenies between abelian schemes of any dimension over any base. We also give
precise conditions guaranteeing genericity in the sense of \cref{def:generic}.

Another way to generalize the previous computations to arbitrary fields (say)
would be to lift the isogeny to characteristic zero and invoke the
complex-analytic computations there. The reader who is satisfied with this
direct argument (and the genericity assumption) may directly skip to
\Cref{sec:alg}. However, we think that the moduli-theoretic approach provides
more geometric insight.

In~§\ref{subsec:moduliav}, we recall general facts on moduli stacks of
p.p.~abelian varieties. In~§\ref{subsec:moduli-deformation}, we formally define
the deformation map attached to an isogeny and compare its incarnations at the
levels of stacks and coarse spaces, thereby obtaining precise conditions for
genericity. In~§\ref{subsec:kodaira-spencer}, we introduce the Kodaira--Spencer
isomorphism and use it to reinterpret results from \cref{sec:cov}, in
particular the relation between the tangent and deformation matrices
(\cref{prop:norm-matrix}). In~§\ref{subsec:moduli-covariants}, we recast the
definition of covariants in the algebraic setting to show that the formulas to
evaluate $\Cov(DJ)(E)$ hold over any base. Finally, we treat the Hilbert case
in~§\ref{subsec:moduli-hilbert}.

\subsection{Moduli stacks of abelian varieties}
\label{subsec:moduliav}

We denote by~$\AAg$ the moduli stack of p.p.~abelian varieties of
dimension~$g$, and by~$\AAgn$ the moduli stack of p.p.~abelian varieties of
dimension~$g$ with a level~$n$ symplectic structure, defined
over~$\Z[1/n]$~\cite{faltingsDegenerationAbelianVarieties1990}.  Both~$\AAg$
and~$\AAgn$ are separated Deligne--Mumford stacks, and~$\AAgn$ is smooth
over~$\Z[1/n]$ with~$\phi(n)$ geometrically irreducible fibers.

We denote by $\Ag$, $\Agn$ their corresponding coarse moduli spaces. By
Mumford's geometric invariant theory
\cite{mumfordGeometricInvariantTheory1994}, they are quasi-projective schemes.
We can extend $\Agn$ over~$\Z$ by taking the normalization of~$\Ag$
in~$\Agn/\Z[1/n]$, as in \cite{mumfordStructureModuliSpaces1971,
  deligneSchemasModulesCourbes1973, dejongModuliSpacesPolarized1993}.
Over~$\C$, the analytification of~$\AAg$ is the Siegel
space~$\Half_g/\Sp_{2g}(\Z)$ seen as an orbifold, generalizing the setting
of~§\ref{subsec:siegel}. If $n \geq 3$, then~$\AAgn$ has trivial inertia, so
$\AAgn$ is isomorphic to its coarse space~$\Agn$, and~$\Agn$ is smooth
over~$\Z[1/n]$.  If $n \leq 2$, then the generic inertia group on~$\AAgn$
is~$\mu_2 = \{\pm 1\}$.

The moduli stack~$\AAgl$ parametrizing $\ell$-isogenies can be constructed as
follows. Let~$\Gamma^0(\ell) \subset \Sp_{2g}(\smash{\Zhat})$ be the congruence
subgroup encoding $\ell$-isogenies, defined as
in~§\ref{subsec:modpol}. Then~$\AAgl$ is the quotient stack~$[\AAGL/\Gamma']$,
where~$\Gamma'$ denotes the image of~$\Gamma^0(\ell)$
in~$\Sp_{2g}(\Z/\ell\Z)$. It is smooth over~$\Z[1/\ell]$. The maps
$\AAGL \to \AAgl$ and $\AAgl \to \AAg$ are finite, étale, and representable
\cite[\S IV.2 and~\S IV.3]{deligneSchemasModulesCourbes1973}.  We can extend
the coarse space~$\Agl$ to~$\Z$ by normalization, as we did for~$\Agn$.

One can also define Siegel modular forms algebraically on~$\AAg$.  Let
$\pi\from \XXg \to \AAg$ be the universal abelian variety. The vector bundle
\begin{displaymath}
  \Hodge = \pi_{\ast} \Omega^1_{\XXg/\AAg}
\end{displaymath}
over~$\AAg$, which is dual to~$\Lie_{\XXg/\AAg}$, is called the \emph{Hodge
  bundle}. If~$\rho$ is a representation of $\GL_g$, a Siegel modular form of
weight~$\rho$ is a section of $\rho(\Hodge)$; in particular, a scalar-valued
modular form of weight~$k$ is a section of $(\wedge^g \Hodge)^{\otimes k}$. In
other words, a Siegel modular form~$f$ can be seen as a map
\begin{displaymath}
  (A, \omega) \mapsto f(A, \omega)
\end{displaymath}
where~$A$ is a point of~$\AAg$ and~$\omega$ is a basis of differential forms
on~$A$, with the following property: if $\eta\from A \to A'$ is an isomorphism,
and~$r\in\GL_g$ is the matrix of~$\eta^*$ in the bases~$\omega',\omega$, then
$f(A', \omega)=\rho(r) f(A,\omega')$.  The link with classical modular forms
over~$\C$ is the following: if $\tau \in \Half_g$, then we define
\begin{displaymath}
  f(\tau)=f\bigl(\C^g/(\Z^g+\tau \Z^g),
  (2 \pi i \,dz_1, \ldots, 2 \pi i\,dz_g)\bigr).
\end{displaymath}
This choice of basis is made so that the $q$-expansion principle holds
\cite[p.\,141]{faltingsDegenerationAbelianVarieties1990}. We already
used it to define~$f(A,\omega)$ over~$\C$ in §\ref{subsec:siegel}.
The canonical line bundle
$\wedge^g \Hodge$ is ample, so modular
forms give local coordinates on $\Ag$.

In the case $g=2$, the structure of the coarse moduli space~$\Atwo$ has been
worked out explicitly~\cite{igusaArithmeticVarietyModuli1960,
  igusaRingModularForms1979}. In particular, the modular
forms~$\psi_4,\psi_6,\chi_{10},\chi_{12}$ from \cref{thm:siegel-structure} are
defined over~$\Z$. The Jacobian locus $\Mtwo$ consisting of Jacobians of
hyperelliptic curves is the open subscheme of $\Atwo$ defined
by~$\chi_{10}\neq 0$. The Igusa invariants $j_1, j_2, j_3$ have bad reduction
modulo~$2$ and do not generate the function field of~$\Mtwo$ modulo~$3$.
Over~$\Z[1/6]$ however, they define a birational map, and more precisely an
isomorphism from $\mathbf{U} = \{\psi_4\chi_{10} \ne 0\}\subset \Mtwo$ to
$\{j_3 \ne 0\}\subset\Avar^3$.

\subsection{The deformation map}
\label{subsec:moduli-deformation}

Consider the map
\begin{align*}
  \PPHI_{\ell} = (\PPHI_{\ell,1},\PPHI_{\ell,2})\from \AAgl &\to \AAg \times \AAg\\
  A &\mapsto (A,A/K).
\end{align*}
It induces a map at the level of coarse moduli spaces, denoted by
\begin{displaymath}
  \PHI_{\ell} = (\PHI_{\ell,1},\PHI_{\ell,2})\from \Agl \to \Ag\times \Ag.
\end{displaymath}
We now study the relations between~$\PPHI_\ell$, $\PHI_\ell$ and modular
equations in detail in order to give precise conditions that guarantee the
genericity of an isogeny in the sense of \cref{def:generic} over any
field~$k$. An overview is as follows:
\begin{enumerate}
\item \label{overview:etale}
  At the level of stacks over~$\Z[1/\ell]$, $\PPHI_{\ell,1}$
  and~$\PPHI_{\ell,2}$ are always finite étale, so there exists a deformation
  map $d\PPHI_{\ell,2}\circ d\PPHI_{\ell,1}^{-1}$ attached to every
  $\ell$-isogeny~$\isog$.
\item \label{overview:stabilizer}
  At points where~$\PPHI_{\ell,1}$ and~$\PPHI_{\ell,2}$ are
  \emph{stabilizer-preserving}, we can compute this deformation map directly at
  the level of the coarse space~$\Agl$.
\item \label{overview:generic}
  If further the domain and codomain of~$\isog$ have \emph{generic
    automorphisms}, then we can compute the deformation map as
  $d\PHI_{\ell,2}\circ d\PHI_{\ell,1}^{-1}$.
\item \label{overview:modeq}
  Under the assumptions of~\eqref{overview:generic}, the deformation map
  can be computed from a suitable normalization of the Siegel modular equations. In
  particular, if~$\isog$ corresponds to a normal point in the image
  of~$\PHI_\ell$, then~$\isog$ is generic.
\end{enumerate}

Item~\eqref{overview:etale} concretely means that the deformation map can
always be computed after adding sufficient structure to rigidify the stacks
involved, a costly procedure in general. The additional assumptions listed make
the computations more and more tractable, at the expense of introducing
new exceptions.

We begin with definitions, assuming all our stacks to be separated
Deligne--Mumford stacks. We denote by $I_\XX$ the inertia stack of a
stack~$\XX$. If~$x$ is a point of~$\XX$, we denote by~$I_x$ the fiber of
$I_\XX$ at~$x$, in other words the finite group of automorphisms of~$x$. We say
that a point~$x$ of~$\AAg$ has \emph{generic automorphisms} if $I_x = \mu_2$,
or equivalently if the abelian variety~$A$ corresponding to~$x$ satisfies
$\Aut(A) = \{\pm 1\}$. Points with generic automorphisms form an open substack
of~$\AAg$.

Let $f\from \XX \to \YY$ be a morphism of stacks. Then~$f$ is representable if
and only if the map $I_{\XX} \to \XX \times_{\YY} I_{\YY}$ induced by~$f$ is a
monomorphism \stackcite{04YY}.  We then say that~$f$ is
\emph{stabilizer-preserving} at~$x$ if the monomorphism on inertia
$I_x \to I_{f(x)}$ induced by~$f$ is an isomorphism.

The following proposition accounts for step~\eqref{overview:etale} of the
overview, and characterizes points where the maps~$\PPHI_{\ell,i}$ are
stabilizer-preserving.

\begin{prop} Let~$\ell$ be a prime.
  \label{prop:moduli-etale}
  \begin{enumerate}
    \item The maps~$\PPHI_{\ell,1}$ and~$\PPHI_{\ell,2}$ are finite, étale and
    representable over~$\Z[1/\ell]$.
  \item Let~$k$ be a field of characteristic distinct from~$\ell$. Let
    $x \in \AAgl(k)$ be a point represented by~$(A,K)$, and let $K'\subset A/K$
    be the kernel of the dual isogeny. Then~$\PPHI_{\ell,1}$ is
    stabilizer-preserving at~$x$ if and only if all automorphisms of~$A$
    stabilize~$K$, and~$\PPHI_{\ell,2}$ is stabilizer-preserving at~$x$ if and
    only if all automorphisms of~$A/K$ stabilize~$K'$.
  \end{enumerate}
\end{prop}

\begin{proof}
  Let~$x$ be a point of~$\AAgl$ corresponding to a pair~$(A,K)$ in the moduli
  interpretation. The automorphisms of~$x$ in $\AAgl$ are exactly the
  automorphisms of~$A$ stabilizing~$K$. In particular~$\PPHI_{\ell,1}$ is
  representable, and it is stabilizer-preserving at~$x$ if and only if all
  automorphisms of~$A$ stabilize~$K$. The map~$\PPHI_{\ell,1}$ is finite étale
  by construction of~$\AAgl$.

  Any automorphism of $(A,K)$, descends to $A' = A/K$, so~$\PPHI_{\ell,2}$ is
  representable as well. An automorphism of~$A'$ comes from an automorphism
  of~$(A,K)$ if and only if it stabilizes~$K'$, hence the condition
  for~$\PPHI_{\ell,2}$ to be stabilizer-preserving.  We finally prove that
  $\PPHI_{\ell,2}$ is finite étale. Denote by $\pi_1\from \XXg \to \AAg$
  the universal abelian scheme, and by $\pi_{\ell}\from \XXgl \to \AAgl$ the
  universal abelian scheme with a $\Gamma^0(\ell)$-level structure. Then the
  universal isogeny $f\from \XXgl \to \XXg \times_{\AAg} \AAgl$ is
  separable over~$\Z[1/\ell]$. Let $s_1\from \AAg\to \XXg$ and
  $s_{\ell}\from\AAgl\to \XXgl$ be the zero sections. Then
  \begin{displaymath}
    \PPHI_{\ell,2}=\PPHI_{\ell,1} \circ \pi_1 \times_{\AAg} \AAgl
    \circ f \circ s_{\ell}.
  \end{displaymath}
  so $\PPHI_{\ell,2}\from \AAgl \to \AAg$ is finite étale as well.
\end{proof}

The next proposition accounts for step~\eqref{overview:stabilizer} in the
overview. From now on, if~$x$ is a point of~$\AAgl$ or~$\AAg$, we denote
by~$\mathbf{x}$ its reduction to the coarse moduli space.

\begin{prop}
  \label{prop:moduli-stabilizer}
  Let~$i=1$ or~$2$.  Let~$x$ be a $k$-point of~$\AAgl$, and assume
  that~$\PPHI_{\ell,i}$ is stabilizer-preserving at~$x$. Then~$\PHI_{\ell,i}$
  is strongly étale at~$\mathbf{x}$, in other words we have étale-locally
  around~$\mathbf{x}$
  \begin{displaymath}
    \AAgl = \Agl \underset{\Ag}{\times} \AAg.
  \end{displaymath}
  The point~$\mathbf{x}$ is smooth in~$\Agl$ if and only if~$\PHI_{\ell,i}(\mathbf{x})$
  is smooth in~$\Ag$.
\end{prop}

\begin{proof}
  By \cref{prop:moduli-etale}, $\PPHI_{\ell,i}$ is finite étale. The étaleness
  of~$\PHI_{\ell,i}$ at a stabilizer-preserving point then comes from Luna's
  fundamental lemma: see e.g.~\cite[Prop.~6.5 and
  Thm.~6.10]{rydhExistencePropertiesGeometric2013}. Strong étaleness comes
  from the cartesian diagram in
  \cite[Thm.~6.10]{rydhExistencePropertiesGeometric2013}, and directly implies
  the last statement in the proposition.
\end{proof}

Under the assumptions of \cref{prop:moduli-stabilizer}, if~$\PPHI_{\ell,1}(x)$
is represented by an abelian variety~$A$ defined over~$k$, then the isogeny
$\isog\from A \to A'$ representing~$x$ is also defined over~$k$ by the same
reasoning as \cite[\S VI.3.1]{deligneSchemasModulesCourbes1973}. Indeed, if
$(A,K)$ represents~$x$ over~$\kbar$, the obstruction for $(A,K)$ to descend
over~$k$ is given by an element in $H^2(\Spec k, \Aut(x))$.  But this
obstruction vanishes since~$\PPHI_{\ell,1}(x)$ is represented by~$A/k$, and the
automorphism groups of~$x$ and~$\PPHI_{\ell,1}(x)$ are equal.

\begin{rem}
  \label{rem:stabpreserving}
  Concretely, \cref{prop:moduli-stabilizer} could be used in computations as
  follows.  Let~$x$ be a $k$-point of~$\AAgl$ where both~$\PPHI_{\ell,1}$
  and~$\PPHI_{\ell,2}$ are stabilizer-preserving, and let~$\mathbf{x}$ be its
  image in~$\Agl$.  For~$i\in\{1,2\}$, let
  $\mathbf{y}_i = \PHI_{\ell,i}(\mathbf{x})$, and let $y_i$ be a lift of
  $\mathbf{y}_i$ to~$\AAg$.  Let~$G=I_x$ be the common automorphism group of
  these objects. Finally, suppose that $\mathbf{x}$ is smooth in $\Agl$ (equivalently,
  $\mathbf{y}_1$ or $\mathbf{y}_2$ is smooth in $\Ag$).  By strong étaleness, the maps
  \begin{displaymath}
    d\PHI_{\ell,i}: T_{\mathbf{x}}(\Agl)\to T_{\mathbf{y}_i}(\Ag)
  \end{displaymath}
  for $i\in\{1,2\}$ are isomorphisms.

  Let~$B_1$ be the completed local ring of~$\AAg$ at~$y_1$. By \cite[\S
  I.8.2.1]{deligneSchemasModulesCourbes1973}, the completed local ring
  of~$\Ag$ at~$\mathbf{y}_1$ is~$B_1^G$.  Therefore, given $m = g(g+1)/2$
  uniformizers $u_1,\ldots, u_m$ of~$\AAg$ at~$y_1$, we obtain $g(g+1)/2$
  uniformizers of~$\Ag$ at~$\mathbf{y}_1$ as $G$-invariant polynomials in
  $u_1,\ldots,u_m$. Assume that such uniformizers of~$\Ag$ have been computed
  at both~$\mathbf{y}_1$ and~$\mathbf{y}_2$. Then we can recover the
  deformation map at the level of stacks from the maps~$d\PHI_{\ell,i}$ up to
  an action of non-generic elements of~$G$, i.e.~up to choosing other
  lifts~$y_1$ and~$y_2$.
  
  In practice, it may be more convenient to work at the level of stacks to
  recover the deformation map directly rather than using $G$-invariant
  uniformizers on~$\Ag$. A key factor in this choice is the degree of the field
  extension we have to consider in order to rigidify the stack. For instance,
  if~$A$ is an abelian surface and~$k$ is a finite field, we can give~$A$ a
  full level~$2$ structure over an extension of degree at most~$6$; over a
  number field, this could take an extension of degree up to~$720$.
\end{rem}

Under the additional assumption of generic
automorphisms~\eqref{overview:generic}, computing the deformation map becomes
considerably easier.

\begin{prop}
  \label{prop:moduli-generic}
  Let~$x$ be a $k$-point of~$\AAgl$, and assume that both~$\PPHI_{\ell,1}(x)$
  and~$\PPHI_{\ell,2}(x)$ have generic automorphisms. Then:
  \begin{enumerate}
  \item \label{it:stab-pres}
    Both~$\PPHI_{\ell,1}$ and~$\PPHI_{\ell,2}$ are stabilizer-preserving at~$x$.
  \item \label{it:smooth1}
    Both~$\PHI_{\ell,1}(\mathbf{x})$ and~$\PHI_{\ell,2}(\mathbf{x})$ are smooth points
    of~$\Ag$, and the map $\AAg\to \Ag$ is étale at these points.
  \item \label{it:smooth2}
    The point~$\mathbf{x}$ is smooth in~$\Agl$, and the map $\AAgl\to \Agl$ is étale
    at~$\mathbf{x}$.
  \item \label{it:diagram}
    We have a commutative diagram
    \begin{displaymath}
      \begin{tikzcd}
        \tangent{\PPHI_{\ell,1}(x)}{\AAg} \ar[d] &
        \tangent{x}{\AAgl}
        \ar[r, "d\PPHI_{\ell,2}"] \ar[l, swap, "d\PPHI_{\ell,1}"] \ar[d]&
        \tangent{\PPHI_{\ell,2}(x)}{\AAg} \ar[d] \\
        \tangent{\PHI_{\ell,1}(\mathbf{x})}{\Ag} &
        \tangent{\mathbf{x}}{\Agl}
        \ar[r, "d\PHI_{\ell,2}"] \ar[l, swap, "d\PHI_{\ell,1}"]&
        \tangent{\PHI_{\ell,2}(\mathbf{x})}{\Ag}
      \end{tikzcd}
    \end{displaymath}
    where the vertical arrows are isomorphisms induced by $\AAgl\to\Agl$ and
    $\AAg\to\Ag$. In particular, the deformation map of the isogeny~$\isog$
    attached to~$x$ is
    $\defo(\isog)= d{\PHI_{\ell,2}}(\mathbf{x}) \circ d{\PHI_{\ell,1}}^{-1}(\mathbf{x})$.
  \end{enumerate}  
\end{prop}

\begin{proof} Item~\eqref{it:stab-pres} follows from the
  definitions. For~\eqref{it:smooth1}, let~$y = \PPHI_{\ell,1}(x)$. Since~$y$
  has generic automorphisms, the map $[\AAg/\mu_2]\to \Ag$ is an isomorphism
  étale-locally around~$\mathbf{y}$, by general facts on the étale-local structure of
  stacks \cite[Lem.~2.2.3]{abramovichCompactifyingSpaceStable2002},
  \cite[Thm.~2.12]{olssonUnderlineRmHom2006}.  The conclusion follows
  since~$\AAg \to [\AAg/\mu_2]$ is étale. Item~\eqref{it:smooth2}
  similarly follows from the fact that~$\AAgl\to \Agl$ is an isomorphism
  étale-locally around~$\mathbf{x}$. Finally, \eqref{it:smooth1} and \eqref{it:smooth2}
  imply \eqref{it:diagram}.
\end{proof}

In the setting of \cref{prop:moduli-generic}, performing a change of
uniformizers as sketched in \cref{rem:stabpreserving} is no longer necessary.

We finally proceed to step~\eqref{overview:modeq} in the overview, and
investigate the relations between the coarse map~$\PHI_\ell$ and modular
equations. The map~$\PHI_\ell$ is not injective, but reasoning as in \cite[\S
VI.6]{deligneSchemasModulesCourbes1973} shows that it induces a birational
isomorphism to its image. The open subscheme of~$\Agl$ where~$\PHI_{\ell}$ is
an embedding is dense in every fiber of characteristic $p \nmid \ell$. We
denote by~$\PHII$ the schematic image of~$\PHI_\ell$, and denote by
$p_1,p_2\from\PHII\to\Ag$ the two projections. When~$g=2$, the modular
equations~$\Psi_{\ell,i}$ from §\ref{subsec:modpol} are equations for the image
of $\PHII \cap (\mathbf{U}\times \mathbf{U})$ in $\Avar^3 \times \Avar^3$ via
the Igusa invariants $j_1,j_2,j_3$.

\begin{prop}
  \label{prop:moduli_normal}
  The scheme~$\Agl$ is the normalization of~$\PHII$.  Thus, if~$\mathbf{x}_0$
  is a point of~$\PHII$, then~$\PHI_\ell\from \Agl \to \PHII$ induces a local
  isomorphism around~$\mathbf{x}_0$ if and only if~$\mathbf{x}_0$ is normal
  in~$\Psi_0$.
\end{prop}

\begin{proof}
  The map $\Agl \to \PHII$ is separated and quasi-finite, and is birational by
  the above discussion. The scheme~$\Agl$ is normal because~$\AAgl$ is normal,
  as seen from the description of its completed local rings~\cite[\S
  I.8.2.1]{deligneSchemasModulesCourbes1973}. We deduce that~$\Agl$ is the
  normalization of~$\PHII$ by Zariski's main theorem
  \cite[Cor.~IV.8.12.11]{grothendieckElementsGeometrieAlgebrique1964}.
\end{proof}

Combining \cref{prop:moduli-generic,prop:moduli_normal}, we obtain
the following conclusion.

\begin{cor}
  \label{cor:moduli-modeq}
  Let~$x$ be a $k$-point of~$\A_g(\ell)$ corresponding to an
  $\ell$-isogeny~$\isog$, and let~$\mathbf{x}_0=\PHI_\ell(\mathbf{x})$. Assume
  that both~$\PPHI_{\ell,1}(x)$ and~$\PPHI_{\ell,2}(x)$ have generic
  automorphisms and that~$\Psi_0$ is normal at~$\mathbf{x}_0$. Then the
  deformation map $\defo(\isog)$ can be computed as
  $d{p_2}(\mathbf{x}_0) \circ d{p_1}(\mathbf{x}_0)^{-1}$. If
  further~$\mathbf{x}_0\in \mathbf{U}\times \mathbf{U}$, then~$\defo(\isog)$
  can be computed from the derivatives of the Siegel modular equations at the
  point~$\mathbf{x}_0$ seen in~$\Avar^3\times\Avar^3$.
\end{cor}

\begin{rem}
  We have the following characterization of non-normal points on~$\PHII$,
  generalizing the remark of
  \cite[p.\,248]{schoofCountingPointsElliptic1995}. Let~$k$ be a field of
  characteristic~$p>0$, and let~$\mathbf{x}_0$ be a $k$-point of~$\PHII$.  We
  remark that $\Psi_0 \otimes k$ is reduced (because the generic automorphisms
  over $k$ are $\{\pm 1\}$ hence the generic points are smooth), so satisfies
  Serre's conditions $S_1$ and $R_0$ \stackcite{031R}.  Normality is equivalent
  to Serre's conditions $S_2$ and~$R_1$ \stackcite{031S}.  Let $\xi$ be a point
  specializing to~$x_0$ and of codimension~$1$ (resp.~2). If $\xi$ is of
  characteristic~$p$, it is of codimension~$0$ (resp.~$1$) in
  $\Psi_0 \otimes k$, hence satisfies Serre's conditions.  So $\mathbf{x}_0$ is
  normal in $\Psi_0\otimes k$ if and only if every lift~$\xi$ of~$\mathbf{x}_0$
  of characteristic~$0$ is normal.

  Now assume that $\PPHI_{\ell,1}$ is stabilizer-preserving at~$x \in \AAgl$,
  let $\mathbf{x}_0=\PHI_{\ell}(\mathbf{x})$ and assume that
  $\PHI_{\ell,1}(\mathbf{x}) \in \Ag$ is smooth.  Then by
  \cref{prop:moduli-stabilizer,prop:moduli_normal}, $\mathbf{x}_0$ is smooth in
  $\PHII$ if and only if $p_1$ is étale at $\mathbf{x}_0$, if and only if
  $\mathbf{x}_0$ is normal in $\PHII$. Hence, by the above discussion,
  $\mathbf{x}_0$ is singular if and only if it is the reduction of a singular
  point in characteristic~$0$.
\end{rem}

\subsection{The Kodaira--Spencer isomorphism}
\label{subsec:kodaira-spencer}

Let~$A\to S$ be a proper abelian scheme, and assume for simplicity that~$S$ is
smooth over~$\Z[1/2]$. Its associated \emph{Kodaira--Spencer map} was first
introduced in \cite{kodairaDeformationsComplexAnalytic1958}; we refer to
\cite[\S III.9]{faltingsDegenerationAbelianVarieties1990} and \cite[\S
1.3]{andreKodairaSpencerMap2017} for more details. This map takes the form
\begin{displaymath}
  \kappa\from T_S \to \Sym^2 \Lie_S (A)
  %= \Hom_{\Sym}(\Omega^1_{A/S}, \Omega^{1\,\dualv}_{A^\dualv/S})
  = \Hom_{\Sym}\bigl(\Lie_S(A)^\dualv, \Lie_S(A^\dualv)\bigr),
\end{displaymath}
where~$T_S$ denotes the tangent bundle on~$S$.  If we apply this construction
to the universal abelian scheme $\XXg \to \AAg$ (or rather, the pullback
of~$\XXg$ to an étale presentation~$S$ of~$\AAg$), the Kodaira--Spencer map is
an isomorphism \cite[\S 2.1.1]{andreKodairaSpencerMap2017}.  In particular, if
$x$ is a $k$-point of~$\AAg$ represented by a p.p.~abelian variety~$A/k$, we
have a canonical isomorphism $\tangent{x}{\AAg} \iso \Sym^2\tangent{0}{A}$.

As a consequence, if~$j$ is a modular invariant (i.e.~a rational
map~$\AAg\to \Avar^1$), then via the Kodaira--Spencer isomorphism, its
differential~$dj$ naturally becomes a Siegel modular function of
weight~$\Sym^2$ in the sense of §\ref{subsec:moduliav}.

Over~$\C$, the Kodaira--Spencer isomorphism can be described explicitly.

\begin{prop}
  \label{prop:analytic_kodaira}
  Let~$V$ be the trivial vector bundle~$\C^g$ on~$\Half_g$, identified
  with the tangent space at~$0$ of the universal abelian
  variety~$A(\tau)$ over~$\Half_g$. Then the pullback of the
  Kodaira--Spencer map $\kappa\from T_{\AAg} \to \Sym^2 \Lie_S \XXg$
  by $\Half_g \to \A_g^{\mathrm{an}}$ is an isomorphism
  $T_{\Half_g} \iso \smash{\Sym^2} V$ given by
  \begin{displaymath}
    \kappa \Bigl( \frac{1+\smash{\delta_{jk}}}{2\pi i} \frac{\partial}{\partial \tau_{jk}} \Bigr)
    = \frac{1}{(2\pi i)^2} \frac{\partial}{\partial z_j}\otimes\frac{\partial}{\partial z_k}.
  \end{displaymath}
  for all $1\leq j,k\leq g$, where $\delta_{jk}$ is the Kronecker symbol.
\end{prop}

\begin{proof}
  The pullback of the Kodaira--Spencer map is an isomorphism by
  \cite[§2.2]{andreKodairaSpencerMap2017}. Its expression can be obtained by
  looking at the deformation of a section~$s$ of the line bundle on~$\XXg$
  giving the principal polarization.  On $\Half_g \times \C^g \to \Half_g$, we
  can take the Riemann theta function~$\theta$ as a section, and its
  deformation along~$\tau$ is given by the heat equation
  \cite[p.\,9]{cilibertoModuliSpaceAbelian2000}:
  \begin{displaymath}
    2 \pi i (1+\delta_{jk}) \frac{\partial\theta}{\partial \tau_{jk}} =
    \frac{\partial^2
      \theta }{ \partial z_j \partial z_k}. \qedhere
  \end{displaymath}
\end{proof}

From \cref{prop:analytic_kodaira}, we recover that the derivatives of modular
invariants have weight~$\smash{\Sym^2}$ in the sense
of~§\ref{sec:mf}. Moreover, the basis of differential forms~$\omega(\tau)$
from~§\ref{subsec:siegel} and the matrix~$\djdtau$ defined
in~§\ref{subsec:norm-matrix} are correctly normalized.

The Kodaira--Spencer isomorphism allows us to define deformation matrices of
$\ell$-isogenies in an algebraic context, and \Cref{prop:norm-matrix} remains
valid.

\begin{defn}
  \label{def:siegel_defo_matrix}
  Let $k$ be a field of characteristic not~$2$ or~$\ell$, let
  $\isog\from A\to A'$ be an $\ell$-isogeny representing a $k$-point
  of~$\AAgl$, and fix bases of~$\tangent{0}{A}$ and~$\tangent{0}{A'}$ as
  $k$-vector spaces. We call the matrix of the tangent map~$d\isog$ in these
  bases the \emph{tangent matrix} of~$\isog$. By functoriality, this choice of
  bases induces bases of~$\tangent{x}{\AAg}$ and~$\tangent{x'}{\AAg}$ over~$k$,
  where~$x,x'$ are the $k$-points of~$\AAg$ corresponding to~$A$ and~$A'$. We
  call the matrix of the deformation map~$\defo(\isog)$ in these bases the
  \emph{deformation matrix} of~$\isog$. We still denote these matrices
  by~$d\isog$ and~$\defo(\isog)$ when the above choice of bases is understood.
\end{defn}

\begin{prop}
  \label{prop:defo_siegel}
  Let~$\isog$ be as in \cref{def:siegel_defo_matrix}, and let $d\isog$
  and~$\defo(\isog)$ be its tangent and deformation matrices in a choice of
  bases of~$\tangent{0}{A}$ and~$\tangent{0}{A'}$. Then
  \begin{displaymath}
    \Sym^2(d\isog) = \ell \defo(\isog).
  \end{displaymath}
\end{prop}

\begin{proof}
  It suffices to prove this relation for the universal $\ell$-isogeny
  \begin{displaymath}
    \isog\from \XXgl \to \XXg \times_{\AAg} \AAgl
  \end{displaymath}
  over~$\Z[1/2\ell]$. All the line bundles involved are locally free on smooth
  stacks, so are flat over~$\Z$; therefore, since $\Z\to\C$ is injective, it
  suffices to prove the relation over~$\C$.  By rigidity \cite[Prop.~6.1 and
  Thm.~6.14]{mumfordGeometricInvariantTheory1994}, it suffices to prove the
  relation on each fiber. Hence we may assume that $\isog\from A \to A'$ is an
  $\ell$-isogeny over~$\C$. There exists~$\tau\in\Half_g$ such that~$A$ is
  isomorphic to $\C^g/(\Z^g + \tau \Z^g)$ and~$A'$ is isomorphic
  to~$\C^g/(\Z^g + \tau/\ell \Z^g)$, with $\isog$ induced by the identity
  on~$\C^g$. In this case, the deformation map at $\isog$ is given by
  $\tau \to \tau/\ell$, so the result follows from the description of the
  Kodaira--Spencer map over~$\C$ in \cref{prop:analytic_kodaira}.
\end{proof}

\subsection{Modular forms and covariants}
\label{subsec:moduli-covariants}

In §\ref{subsec:kodaira-spencer}, we showed that the differentials of modular
invariants are algebraic Siegel modular functions of weight~$\Sym^2$. In the
case of the Igusa invariants when~$g=2$ over~$\C$, \Cref{thm:vector-identification}
identifies these modular functions with explicit covariants of genus~$2$ curve
equations. We now prove an algebraic analogue of this statement.  As a
consequence, all the computations of \cref{sec:cov} remain valid over every field
of characteristic not~$2$ or~$\ell$.

Note that covariants make sense over every ring~$R$, replacing~$\C$ by~$R$ in
\cref{def:cov}.  In order to relate them with algebraic Siegel modular forms,
we consider the Torelli morphism
\begin{displaymath}
  \tau_g\from \MM_g \to \AAg
\end{displaymath}
where~$\MM_g$ denotes the moduli stack of smooth curves of genus~$g$. Let
$\CC_g\to\MM_g$ denote the universal curve. Then the
pullback~$\smash{\tau_g^\ast \Hodge}$ of the Hodge bundle by~$\tau_g$ is
$\pi_{\ast} \Omega^1{\CC_g/\MM_g}$, and both vector bundles carry compatible
actions of $\GL_g$.

Now assume that $g=2$. Over $\Z[1/2]$, the moduli stack $\MM_2$ is identified
with the moduli stack of nondegenerate binary forms of degree~$6$.  Let
$V=\Z x \oplus \Z y$, let $X=\det^{-2} V \otimes \Sym^6 V$, and
let~$U\subset X$ be the open locus of binary forms with nonzero discriminant.
Then~$U\to \MM_2$ is naturally identified with the Hodge frame bundle
on~$\MM_2$, by sending the binary form $\poly$ to the curve $y^2=\poly(x,1)$
with the basis of differential forms $(x\,dx/y, dx/y)$ \cite[\S
4]{cleryCovariantsBinarySextics2017}. In other words,~$U$ is the moduli space
of genus~$2$ hyperelliptic curves $\pi\from C \to S$ endowed with a
rigidification $\OO_S^{\oplus 2} \iso \pi_\ast \Omega^1_{C/S}$. Therefore,
over~$\Z[1/2]$, every Siegel modular form of weight~$\rho$ pulls back to a
fractional covariant of weight~$\rho$.

Write $\smash{\Cov(f)}$ for the covariant attached to a Siegel modular
function~$f$, and denote by~$C$ the canonical covariant of
weight~$\det^{-2}\Sym^6$, i.e.~the binary sextic form itself. We now show that
\cref{prop:f86} remains true in the algebraic setting.

\begin{prop}
  \label{prop:identification}
  The equality $\Cov(\chi_{10})\,C = \Cov(\chi_{6,8})$ holds over~$\Z[1/2]$.
\end{prop}

\begin{proof}
  The covariants~$\Cov(\chi_{10})$ and~$C$ have integer coefficients, so they
  are defined over~$\Z[1/2]$. Since the Hodge bundle is without torsion, it is
  enough to check equality over~$\C$, which is the content of \cref{prop:f86}.
\end{proof}

As a consequence of \cref{prop:identification}, the identification of
the derivatives of the Igusa invariants as explicit covariants
(\cref{thm:vector-identification}) still holds over~$\Z[1/2]$.

\begin{rem}
  In fact, one can show as in~\cref{thm:cov-mf}, by considering suitable
  compactifications, that a Siegel modular form pulls back to a polynomial
  covariant over every ring~$R$ in which~$2$ is invertible. Using Igusa's
  universal form \cite[\S 2]{igusaArithmeticVarietyModuli1960}, one can also
  use binary forms of degree~$6$ to describe the moduli stack of genus~$2$
  curves even in characteristic~$2$.  This suggests another, entirely algebraic
  proof of \cref{prop:identification}. By dimension considerations, we have
  $\Cov(\chi_{10})\, C= \lambda \Cov(\chi_{6,8})$ for
  some~$\lambda\in\Q^\times$. The covariant $\Cov(\chi_{10})\, C$ is defined
  over~$\Z$ and primitive; therefore, if we can show that the Fourier
  coefficients of~$\chi_{6,8}$ are globally coprime integers, we will
  have~$\lambda=\pm 1$. An algebraic way to obtain~$\lambda=1$ could be to
  study degenerations from hyperelliptic curves to elliptic curves using
  \cite[Thm.~1.II]{liuCourbesStablesGenre1993}.
\end{rem}

\begin{rem}
  Let~$k$ be a field of characteristic different from~$2$ and~$3$, and let~$A$ be a
  p.p.~abelian surface over~$k$ such that $\Aut(A) = \{\pm 1\}$
  and~$j_3(A)\neq 0$. Let~$E$ be a genus~$2$ curve equation for~$A$. Then as a
  consequence of \cref{thm:vector-identification} over~$\Z[1/2]$, we obtain an
  explicit Kodaira--Spencer isomorphism at~$A$: it is equivalent to give
  \begin{enumerate}
  \item A deformation~$\widetilde{E}$ of~$E$ over $k[\epsilon]/(\epsilon^2)$,
  \item The Igusa invariants of~$\Jac(\Crv_{\widetilde{E}})$
    in~$k[\epsilon]/(\epsilon^2)$,
  \item A
    vector~$\alpha w_1^2 + \beta w_1 w_2 + \gamma w_2^2\in \Sym^2
    \Omega^1(\Crv_E)$, where $(w_1,w_2) = \omega_E$ is the canonical basis of
    differential forms on~$\Crv_E$.
  \end{enumerate}
  Switching between representations can be done in~$O(1)$ operations in~$k$.
\end{rem}

\subsection{Hilbert--Blumenthal stacks}
\label{subsec:moduli-hilbert}

There exists a similar algebraic
interpretation of the results of \cref{sec:cov} for isogenies of Hilbert type
in every dimension. This reformulation is based on \emph{Hilbert--Blumenthal
  stacks}, which classify abelian schemes with a real multiplication
structure~\cite{rapoportCompactificationsEspaceModules1978,
  chaiArithmeticMinimalCompactification1990}. We will simply outline the
main results, as the proof methods are similar to the Siegel case.

Let~$K$ be a real number field of dimension~$g$, and let~$\Z_K$ be its maximal
order.  We say that an abelian scheme $A \to S$ has \emph{real multiplication
  by~$\Z_K$} if it is endowed with a morphism $\iota\from \Z_K \to \End(A)$
such that $\Lie(A)$ is locally free of rank~$1$ as
a~$\Z_K \otimes \OO_S$-module.  The stack~$\HHg$ of p.p.~abelian schemes with
real multiplication by~$\Z_K$ is algebraic and smooth of relative dimension~$g$
over $\Spec \Z$ \cite[Thm.~1.14]{rapoportCompactificationsEspaceModules1978}.
Moreover,~$\HHg$ is connected and its generic fiber is geometrically connected
\cite[Thm.~1.28]{rapoportCompactificationsEspaceModules1978}.  Forgetting the
real multiplication yields the \emph{Hilbert embedding} $\HHg \to \AAg$, which
is an $\Aut(K)$-gerbe over its image, the \emph{Humbert stack}.  The map
$\HHg \to \AAg$ is finite \cite[EGA
IV.15.5.9]{grothendieckElementsGeometrieAlgebrique1964},
\cite[Lem~1.19]{deligneSchemasModulesCourbes1973}, and we described its
analytification in \cref{sec:mf}.

If~$\beta$ is a totally positive prime of~$\Z_K$, we can also construct the
stack~$\HHgbeta$ of abelian schemes with real multiplication endowed with the
kernel of a $\beta$-isogeny over~$\Z[1/N_{K/\Q}(\beta)]$. We are interested in
the map
\begin{align*}
  \PPHI_{\beta} = (\PPHI_{\beta,1},\PPHI_{\beta,2})\from \HHgbeta &\to \HHg \times \HHg\\
  A &\mapsto (A, A/K)  .
\end{align*}
As above, we use bold characters to denote the associated coarse maps and
spaces. We then have the following analogue of \cref{prop:moduli-generic}.

\begin{prop}
  \label{prop:hilbert-criteria}
  Let~$k$ be a field of characteristic not dividing $N_{K/\Q}(\beta)$.
  Let~$x$ be a $k$-point of $\HHgbeta$, and assume that
  both~$\PPHI_{\beta,1}(x)$ and~$\PPHI_{\beta,2}(x)$ have generic
  automorphisms. Then~$x$ maps to a smooth point of~$\Hgbeta$,
  both~$\PPHI_{\beta,1}(x)$ and~$\PPHI_{\beta,2}(x)$ map to smooth points
  of~$\Hg$, and we have a commutative diagram
  \begin{displaymath}
    \begin{tikzcd}
      \tangent{\PPHI_{\beta,1}(x)}{\HHg} \ar[d] &
      \tangent{x}{\HHgbeta}
      \ar[r, "d\PPHI_{\beta,2}"] \ar[l, swap, "d\PPHI_{\beta,1}"] \ar[d]&
      \tangent{\PPHI_{\beta,2}(x)}{\HHg} \ar[d] \\
      \tangent{\PHI_{\beta,1}(x)}{\Hg} &
      \tangent{x}{\Hgbeta}
      \ar[r, "d\PHI_{\beta,2}"] \ar[l, swap, "d\PHI_{\beta,1}"]&
      \tangent{\PHI_{\beta,2}(x)}{\Hg}
    \end{tikzcd}
  \end{displaymath}
  where the vertical arrows are isomorphisms.
\end{prop}

We deduce the following sufficient conditions to ensure the genericity of an
isogeny as in~§\ref{subsec:explicit-hilbert}. Let~$\PHIIb\subset\Hg\times\Hg$
be the image of~$\PHI_\beta$, and let $\PHIIH \subset \Ag\times\Ag$ denote the
image of~$\PHIIb$ under the Hilbert embedding.

\begin{cor}
  \label{cor:generic-hilbert}
  Let~$x$ be a $k$-point of $\HHgbeta$ such that both
  $x_1 = \PPHI_{\beta,1}(x)$ and $x_2 = \PPHI_{\beta,2}(x)$ only have generic
  automorphisms.  Assume further that~$(x_1, x_2)$ does not lie in the image of
  $\PPHI_{\betabar}$, in other words the corresponding abelian varieties are
  $\beta$-isogenous but not $\smash{\betabar}$-isogenous, and that $(x_1,x_2)$
  maps to a normal point of $\PHIIb$. Let~$\mathbf{y}$ the image of
  $\mathbf{x}$ by the forgetful morphism $\Hg \times \Hg \to \Ag \times \Ag$,
  and assume finally that~$\mathbf{y}$ lies
  in~$\mathbf{U}\times\mathbf{U}$. Then the $\beta$-isogeny corresponding
  to~$x$ is generic in the sense of~§\ref{subsec:explicit-hilbert}.
\end{cor}

To obtain an algebraic interpretation of \cref{prop:norm-matrix-hilbert}, we
invoke the Hilbert analogue of the Kodaira--Spencer isomorphism \cite[Prop.~1.6
and Prop.~1.9]{rapoportCompactificationsEspaceModules1978}. If $A \to
S$ is an abelian scheme corresponding to a point~$x$ of
$\HHg$, this isomorphism is
\begin{displaymath}
  \tangent{x}{\HHg} \iso \Hom_{\Z_K \otimes \OO_S}(\Lie(A)^\dualv,
  \Lie(A^\dualv)).
  %= \Lie(A^\dualv) \otimes_{\Z_K \otimes \OO_S} \Lie(A)
  %\otimes_{\Z_K} \different.
\end{displaymath}
Thus, on Hilbert--Blumenthal stacks, the deformation map is represented by an
element of $\Z_K \otimes \OO_S$ rather than a matrix in
$\OO_S$.  By \cite[\S~1.5]{rapoportCompactificationsEspaceModules1978}, the
Kodaira--Spencer isomorphisms
at~$A$ in the Hilbert and Siegel case fit in a commutative diagram with the
forgetful maps:
\begin{displaymath}
  \begin{tikzcd}
    \tangent{x}{\HHg} \ar[r] \ar[d]  & \tangent{x}{\AAg} \ar[d] \\
    \Hom_{\Z_K\otimes \OO_S}(\Lie(A)^\dualv, \Lie(A^\dualv)) \ar[r] &
    \Hom_{\Sym}(\Lie_S(A)^\dualv, \Lie_S(A^\dualv)).
  \end{tikzcd}
\end{displaymath}

In view of \cref{prop:analytic_kodaira} and the analytic description of the
forgetful map in~§\ref{subsec:hilbert-siegel} (easily generalized to every
dimension~$g$), the Kodaira--Spencer isomorphism in the Hilbert case takes the
following form over~$\C$.

\begin{prop}
  \label{prop:analytic_kodaira_hilbert}
  The pullback of the Kodaira--Spencer isomorphism
  % $\kappa\from T_{\HHg}\to \Sym^2 \Lie_S \XXg$
  under the analytic cover $\Half_1^g \to \HHg^{\text{an}}$ satisfies for every $1\leq j\leq g$:
  \begin{displaymath}
    \kappa \Bigl(\frac{1}{\pi i}  \frac{\partial}{\partial t_{j}}\Bigr)
    = \frac{1}{(2\pi i)^2}
    \frac{\partial}{\partial z_j} \otimes \frac{\partial}{\partial z_j}.
  \end{displaymath}  
\end{prop}

This result gives an algebraic interpretation for the presence of the
matrix~$T$ in \cref{prop:norm-matrix-hilbert}: in genus~$2$, the part of
$\tangent{x}{\AAtwo}$ coming from the Hilbert space is the span of
$dz_1\otimes dz_1$ and $dz_2\otimes dz_2$. We deduce from
\cref{prop:analytic_kodaira_hilbert} a relation between the tangent and
deformation matrices in the Hilbert case.

\begin{prop}
  \label{prop:defo_hilbert}
  Let $\isog\from A \to A'$ be a $\beta$-isogeny between abelian schemes with
  real multiplication over a base~$S \to \Z[1/N_{K/\Q}(\beta)]$. Denote
  by~$d\isog$ and $\defo(\isog)$ its associated tangent and deformation maps,
  seen as elements of $\Z_K\otimes \OO_S$-modules. Then under the
  Kodaira--Spencer isomorphism, we have $(\disog)^2 = \beta \defo(\isog)$.
\end{prop}

The last remaining step to prove that the computations
of~§\ref{subsec:explicit-hilbert} remain valid over every field is to give an
algebraic interpretation of the notion of (potentially) Hilbert-normalized
bases and the method to construct them in \cref{algo:hilb-curve-2}.

Let~$k$ be a field. Provided that $\chr k\nmid \Delta$, and up to taking an
étale extension of~$k$, we may assume that~$k$ splits~$\Z_K$, and fix a
trivialization $ \Z_K \otimes k \iso k^g$. Let~$A$ be an abelian variety
representing a $k$-point of~$\HHg$. Then~$\Lie(A)$ is a free
$\Z_K \otimes k$-module of rank~$1$, and a Hilbert-normalized basis
of~$\tangent{0}{A}$ is simply a basis of~$\Lie(A)$ as a $k$-vector space on
which~$\Z_K$ acts diagonally.  Let $(v_1, \dots, v_g)$ be a Hilbert-normalized
basis of~$\Lie(A)$, let $(w_1, \dots, w_g)$ be another $k$-basis and let~$M$ be
the base-change matrix. Then $w_1 \otimes w_1, \dots, w_g \otimes w_g$ are
tangent to the Humbert variety if and only if they are in the image of the map
\begin{displaymath}
  \Hom_{\Z_K \otimes k}(\Lie(A)^\dualv, \Lie(A^\dualv)) \to
  \Hom_{\Sym}(\Lie(A)^\dualv, \Lie(A^\dualv)).
\end{displaymath}
Therefore, the vectors $w_1 \otimes w_1, \dots, w_g \otimes w_g$ are tangent to
the Humbert variety if and only if~$M$ is diagonal up to a
permutation. When~$g=2$, this ensures that the basis $(w_1,\ldots,w_g)$ is
potentially Hilbert-normalized.

\section{Computing the isogeny from its tangent map}
\label{sec:alg}

Assume that we are given the tangent map~$d\isog$ of an isogeny
$\isog\from A \to A'$ between Jacobians of genus~$2$ curves defined over a
field~$k$, computed for instance from derivatives of modular equations as in
\cref{sec:cov}. We now describe how to compute~$\isog$ as a rational map by
solving a differential system with Newton iterations.

This approach is not new: \cite{elkiesEllipticModularCurves1998} introduces a
differential equation to compute isogenies in genus~$1$,
and~\cite{bostanFastAlgorithmsComputing2008} solves it with Newton
iterations. These ideas were extended to genus~$2$
in~\cite[§6.2]{couveignesComputingFunctionsJacobians2015}
and~\cite[§5.2]{costaRigorousComputationEndomorphism2019}. (Note that~$d\isog$
is obtained there in totally different ways, respectively using the kernel
of~$\isog$ as input and via a numerical approach whose complexity is hard to
control.) We will indicate the relevant differences between these references
and the differential system we set~up. Mainly, Newton iterations allow us to
reach a quasi-linear complexity in~$\ell$ instead of (at best) quasi-quadratic
using an iterative method.

\subsection{General strategy}
\label{subsec:introalg}

In general, the task of computing~$\isog$ may be specified as follows: given
models of~$A$ and~$A'$, that is given very ample line bundles~$\pol_A$
and~$\pol_{A'}$ on~$A$ and~$A'$ and a choice of global sections~$(a_i)$
(resp.~$(a'_j)$) which give a projective embedding of~$A$ (resp.~$A'$), express
the functions~$\isog^\ast a'_j$ on~$A$ as rational fractions in terms of the
coordinates~$(a_i)$.

One method to determine~$\isog$ from~$d\isog$ is to work with formal groups.
Let $x_1, \ldots, x_g$ be local uniformizers at $0_A$. Knowing $d\isog$ allows
us to write a differential system satisfied by the functions~$\isog^\ast a'_j$,
and we can attempt to solve it with a multivariate Newton algorithm. Upon
success, we recover the functions~$\isog^\ast a'_j$ as power series
in~$k[[x_1, \ldots, x_g]]$ up to some precision.  The next step is to use a
multivariate rational reconstruction algorithm to obtain~$\isog$ as a rational
map, assuming that the power series precision is large enough compared to the
degrees of the functions~$\isog^\ast a'_j$ in the variables~$(a_i)$.  For the
whole method to work,~$\isog$ must be completely determined by its tangent
map. This will be the case when $\chr k$ is large with respect to the degree
of~$\isog$. In practice, Newton iterations fail to reach sufficiently high
power series precision if $\chr k$ is too small, hence the bound~$8\ell+1$
in~\cref{thm:main}.

In genus~$2$ and away from characteristic~$2$, nice simplifications occur.
Let~$E$ and~$F$ be genus~$2$ curve equations, let $A = \Jac(\Crv_E)$ and
$A' = \Jac(\Crv_{F})$, and assume that we are given the matrix of $d\isog$ in
the bases of~$T_0(A)$ and~$T_0(A')$ that are dual to~$\omega_E$
and~$\omega_{F}$ respectively (see §\ref{subsec:hyperelliptic}). Then~$\isog$
is determined by the composition
\begin{displaymath}
  \begin{tikzcd}
    \Crv_E \ar[rr, hook, "{Q\mapsto [Q-P]}" ] & & \Jac(\Crv_E) \ar[r, "{\isog}" ]
    & \Jac(\Crv_{F}) \ar[r, dashed, "\sim"] & \Crv_{F}^{2,\sym} \ar[r, dashed, "m"]
    & \Avar^4
  \end{tikzcd}
\end{displaymath}
where~$P$ is any point on~$\Crv_E$, the symbol $\Crv_{F}^{2,\sym}$ denotes the
symmetric square of the curve~$\Crv_{F}$, and $m$ is the rational map given by
\begin{displaymath}
  \{(x_1,y_1),(x_2,y_2)\} \mapsto \Bigl(x_1+x_2,\ x_1x_2,\ y_1y_2,\ \frac{y_2-y_1}{x_2-x_1}\Bigr).
\end{displaymath}
This composite map is a quadruple rational fractions $s,p,q,r\in k(u,v)$ that
we call the \emph{rational representation of~$\isog$ at the base point~$P$}.
We choose a uniformizer~$z$ of~$\Crv_E$ around~$P$ and perform the Newton
iterations and rational reconstruction over the univariate power series
ring~$k[[z]]$.

We explain how to solve the resulting differential system
in~§\ref{subsec:diffsyst}. One difficulty is that the differential system we
obtain is singular, so we need to use the geometry of the curves to find the
first few terms in the series before switching to Newton iterations.
In~§\ref{subsec:rational}, we estimate the degrees of the rational fractions
that we want to compute and present the rational reconstruction step.

\subsection{Solving the differential system}
\label{subsec:diffsyst}

We keep the notation used in~§\ref{subsec:introalg}, and assume that
the characteristic of~$k$ is not~$2$. Write the curve equations
$\Crv_E,\,\Crv_{F}$ and the tangent matrix as
\begin{displaymath}
  \Crv_E\defby v^2 = E(u),\quad \Crv_E\defby y^2 = F(x),
  \quad d\isog = \mat{m_{1,1}}{m_{1,2}}{m_{2,1}}{m_{2,2}}.
\end{displaymath}
We assume that~$\isog$ is separable, so~$d\isog$ is invertible.  Let
$P\in \Crv_E(k)$ be a base point on~$\Crv_E$ (enlarging~$k$ if necessary). We
denote by $\isog_P$ the associated map
$\Crv_E\to\Crv_{F}^{2,\sym}$. Since~$\isog_P(P)$ is zero in~$\Jac(\Crv_{F})$,
we have
\begin{displaymath}
  \isog_P(P) = \bigl\{Q, i(Q)\bigr\}
\end{displaymath}
for some~$Q\in\Crv_{F}$, where~$i$ denotes the hyperelliptic
involution. Below, we will choose~$P$ such a way that~$Q$ is not a Weierstrass
point on~$\Crv_{F}$. If~$z$ is a local uniformizer of~$\Crv_E$ at~$P$, and~$R$
is a finite extension of~$k[[z]]$, we define a \emph{local lift of~$\isog_P$
  with coefficients in~$R$} to be a tuple
$\widetilde{\isog}_P = (x_1, x_2, y_1, y_2) \in R^4$ such that we have a
commutative diagram
\begin{displaymath}
  \begin{tikzcd}
    \Spec R \ar[rr, "{(x_1,y_1), (x_2,y_2)}"]
    \ar[d] && \Crv_{F}\times \Crv_{F} \ar[d] \\
    \Spec k[[z]] \ar[r] & \Crv_E \ar[r, "\isog_P"] & \Crv_{F}^{\,2,\sym}.
  \end{tikzcd}
\end{displaymath}

Assume that~$Q$ is not a Weierstrass point on~$\Crv_{F}$. Since the unordered
pair $\{Q, i(Q)\}$ is defined over~$k$, $Q$ is defined over a quadratic
extension~$k'$ of~$k$. The map $\Crv_{F}\times \Crv_{F}\to \Crv_{F}^{\,2,\sym}$
is étale at $(Q, i(Q))$, and thus induces an isomorphism of completed local
rings. Therefore, a local lift of~$\isog_P$ exists over~$k'[[z]]$.

The basis~$\omega_{F}$ of~$\Omega^1(\Jac(\Crv_{F}))$ corresponds
to the pair of differential forms
\begin{displaymath}
  \Bigl( \frac{x_1\, dx_1}{y_1} + \frac{x_2\, dx_2}{y_2},
  \frac{dx_1}{y_1} + \frac{dx_2}{y_2}  \Bigr)
\end{displaymath}
on $\Crv_{F}^{\,2,\sym}$. Thus, every local lift $(x_1,x_2,y_1,y_2)$ satisfies
the differential system
\begin{equation*}
  \label{eq:diffsyst}
  \tag{$S$}
  \begin{cases}
    \begin{matrix}
      \dfrac{x_1}{y_1} \dfrac{dx_1}{dz} + \dfrac{x_2}{y_2} \dfrac{dx_2}{dz} & = & (m_{1,1}
      u + m_{1,2}) \dfrac{1}{v}\dfrac{du}{dz} \smallskip \\ 
      \dfrac{1}{y_1} \dfrac{dx_1}{dz} + \dfrac{1}{y_2} \dfrac{dx_2}{dz} & = & (m_{2,1} u +
      m_{2,2}) \dfrac{1}{v} \dfrac{du}{dz} \smallskip \\
      y_1^2 = F(x_1) & & \\
      y_2^2 = F(x_2), & &
    \end{matrix}
  \end{cases}
\end{equation*}
where we consider the coordinates $u,v$ on~$\Crv$ as elements of~$k[[z]]$,
and~$d/dz$ denotes differentiation with respect to~$z$. In the remainder of
this section, we focus on solving this system up to a given precision, starting
with the determination of~$Q$.

\begin{rem}
  In~\cite{couveignesComputingFunctionsJacobians2015}, a differential system
  is used to compute a local lift of~$\isog_P$ at a base point other
  than~$P$. In our context, it is unclear how one would initialize such a
  system, as it would require knowing the image of~$\isog$ at a non-zero point
  of~$\Jac(\Crv_E)$. In
  contrast,~\cite[§5]{costaRigorousComputationEndomorphism2019} (specialized
  to the genus~$2$ case) also uses the zero point as a base point. However,
  they consider a birational map~$\Crv_F^{2,\,\sym} \to \Jac(\Crv_F)$ coming
  from a degree~$2$ divisor~$2P_0$ where~$P_0$ is not a Weierstrass point
  (whereas we take the canonical divisor, in other words~$P_0$ is a Weierstrass
  point). This removes the question of determining~$Q$, but in exchange one has
  to work with Puiseux series.
\end{rem}

\begin{prop}
  \label{prop:imagept}
  The point~$Q$ is uniquely determined by the following property: if~$\omega_P$
  (resp.~$\omega'_Q$) is a nonzero differential form on~$\Crv_E$ (resp.~$\Crv_{F}$)
  vanishing at~$P$ (resp.~$Q$), then there exists~$\lambda\in k^\times$ such that
  \begin{displaymath}
    \isog^*\omega'_Q = \lambda\, \omega_P.
  \end{displaymath}
\end{prop}

\begin{proof}
  First, assume that~$Q$ is not a Weierstrass point, so that a local
  lift~$\widetilde{\isog}_P$ exists over~$k'[[z]]$, where~$k'$ is a quadratic
  extension of~$k$.  The tangent space of $\Crv_{F}\times \Crv_{F}$ at
  $(Q, i(Q))$ decomposes as
  \begin{displaymath}
    \tangent{(Q,i(Q))}{\Crv_{F}\times \Crv_{F}}
    = \tangent{Q}{\Crv_{F}} \oplus \tangent{i(Q)}{\Crv_{F}}
    \simeq \tangent{Q}{\Crv_{F}}^2,
  \end{displaymath}
  where the last map is given by the hyperelliptic involution on the second
  term. Now consider the tangent vector $d\widetilde{\isog}_P/dz$ at $z=0$, and
  write it as $(v+w, w)$ for some $v,w\in \tangent{Q}{\Crv_{F}}$. Then~$v\neq 0$: indeed the
  whole direction $(w,w)$ is contracted to zero in the Jacobian, so if~$v$ were
  zero, every differential form on the Jacobian would be pulled back to zero
  via~$\isog_P$, contradicting the separability of $\isog$. Let~$\omega'$ be
  the unique nonzero differential form pulled back to~$\omega_P$
  by~$\isog$. Then~$\omega'$ must vanish on $(v, 0)$, in other words~$\omega'$
  must vanish at~$Q$, as claimed.

  If~$Q$ is a Weierstrass point, we can still find a local lift $(x_1,y_1,x_2,y_2)$
  of~$\isog_P$ with coefficients in~$k'[[\sqrt{z}]]$, where~$k'/k$ is a
  quadratic extension \stackcite{09E8}. After a change of variables, we may
  assume that~$P$ and~$Q$ are not at infinity. Write $P = (u_0, v_0)$ and
  $Q = (x_0, 0)$. The equality in the proposition can be rewritten as
  \begin{equation}
    \label{eq:xQ}
    x_0 = \dfrac{m_{1,1} u_0 + m_{1,2}}{m_{2,1}u_0 + m_{2,2}}.
  \end{equation}
  To show this, we use the system~\eqref{eq:diffsyst}. Write
  \begin{displaymath}
    y_1 = v_1\sqrt{z} + t_1 z + O(z^{3/2}) , \quad y_2 =
    v_2\sqrt{z} + t_2 z + O(z^{3/2}).
  \end{displaymath}
  Then the relation $y^2 = F(x)$ in~\eqref{eq:diffsyst} forces $x_1, x_2$ to have
  no term in $\sqrt{z}$, so that
  \begin{displaymath}
    x_1 = x_0 + w_1 z + O(z^{3/2}),\quad x_2 = x_0 + w_2z + O(z^{3/2}).
  \end{displaymath}
  Using the relation $dx/y = 2 dy/F'(x)$ (where~$F'$ is the derivative of~$F$),
  we have
  \begin{displaymath}
    \begin{cases}
      \dfrac{2 x_1}{F'(x_1)}\dfrac{dy_1}{dz} + \dfrac{2
        x_2}{F'(x_2)} \dfrac{dy_2}{dz} = (m_{1,1} u +
      m_{1,2})\dfrac{1}{v}\dfrac{du}{dz}, \smallskip\\
      \dfrac{2}{F'(x_1)} \dfrac{dy_1}{dz} +
      \dfrac{2}{F'(x_2)} \dfrac{dy_2}{dz} = (m_{2,1} u +
      m_{2,2})\dfrac{1}{v}\dfrac{du}{dz}.
    \end{cases}
  \end{displaymath}
  Inspection of the~$(\sqrt{z})^{-1}$ term gives $v_1 = -v_2$. Write
  $e = F'(x_0)$. Then the constant terms of the series on the left hand
  side are respectively
  \begin{displaymath}
    2 x_0 \Bigl(\frac{t_1}{e} + \frac{t_2}{e}\Bigr)\quad\text{and} \quad
    2 \Bigl(\frac{t_1}{e} + \frac{t_2}{e} \Bigr).
  \end{displaymath}
  The differential forms on the right hand side do not vanish simultaneously
  at~$P$, so $m_{2,1}u_0 + m_{2,2}$ is nonzero, and quotienting the
  two lines gives the result.
\end{proof}

Using \Cref{prop:imagept}, specifically~\eqref{eq:xQ}, we choose a base
point~$P$ such that~$Q$ is not Weierstrass. Then a local lift
$\widetilde{\isog}_P = (x_1,x_2,y_1,y_2)$ of~$\isog_P$ exists over~$k'[[z]]$,
where~$k'$ is quadratic over~$k$, and knowing~$Q= (x_0,y_0)$
specifies its constant term.

The next step is to compute the power series
$x_1,x_2,y_1,y_2$ up to $O(z^2)$. Write
\begin{displaymath}
  x_1 = x_0 + v_1 z + O(z^2), \quad x_2 = x_0 + v_2 z + O(z^2).
\end{displaymath}
Using the curve equations, we can compute $y_1$ and~$y_2$ up to~$O(z^2)$ in
terms of~$v_1$ and~$v_2$ respectively. Let $u_0$ (resp.~$d_0$) be the constant
term of the power series~$u$ (resp.~$1/v \cdot
du/dz$). Then~\eqref{eq:diffsyst} gives
\begin{equation}
  \label{eq:vi1}
  v_1 + v_2 = \dfrac{y_0}{x_0}(m_{1,1}u_0 + m_{2,1}) d_0 = y_0
  (m_{2,1}u_0 + m_{2,2}) d_0.
\end{equation}
Combining the two lines of~\eqref{eq:diffsyst}, we also obtain
\begin{displaymath}
  (x_1 - x_0) \dfrac{dx_1}{y_1} + (x_2 - x_0) \dfrac{dx_2}{y_2} = R,
\end{displaymath}
where $R= r_1z + O(z^2)$ has no constant term. At order 1, this yields
\begin{equation}
  \label{eq:vi2}
  v_1^2 + v_2^2 = y_0 r_1.
\end{equation}
Equalities~\eqref{eq:vi1} and~\eqref{eq:vi2} yield a quadratic equation satisfied by
$v_1, v_2$. This gives the values of $v_1$ and $v_2$ in a quadratic
extension $k'/k$.

We are now ready to begin the Newton iteration procedure. Assume that the
series $x_1, x_2, y_1, y_2$ are known up to~$O(z^n)$ for some $n\geq 2$. The
system~\eqref{eq:diffsyst} is satisfied up to~$O(z^{n-1})$ for the first two
lines, and~$O(z^n)$ for the last two lines. We attempt to double the precision,
and write
\begin{displaymath}
  x_1 = x_1^0(z) + \delta x_1(z) + O(z^{2n}),\ \text{etc.}
\end{displaymath}
where $x_1^0$ is the polynomial of degree at most $n-1$ that has been
computed. The series~$\delta x_i$ and~$\delta y_i$ start at the
term~$z^n$. Linearizing~\eqref{eq:diffsyst}, we obtain the following.

\begin{prop}
  \label{prop:diffsyst-linearize}
  The power series~$\delta x_1$, $\delta x_2$ satisfy a linear
  differential equation of the first order
  \begin{equation}
    \label{eq:mde}
    \tag{$E_n$}
  %\mat{\dfrac{x_1^0}{y_1^0}}{\dfrac{x_2^0}{y_2^0}}{\dfrac{1}{y_1^0}}{\dfrac{1}{y_2^0}}
    M(z)
    \vectwo{d(\delta x_1)/dz}{d(\delta x_2)/dz} + N(z) \vectwo{\delta
      x_1}{\delta x_2} = R(z) + O(z^{2n-1})
  \end{equation}
  where $M, N, R$ are $2\times 2$ matrices with coefficients in $k'[[z]]$ and
  have explicit expressions in terms of $x_1^0$, $x_2^0$, $y_1^0$, $y_2^0$,
  $u$, $v$, $E$ and $F$. In particular,
  \begin{displaymath}
    M(z) =
    \mat{x_1^0/y_1^0}{x_2^0/y_2^0}{1/y_1^0}{1/y_2^0}
  \end{displaymath}
  and, writing $e = F'(x_0)$, the constant term of $N$ is
  \begin{displaymath}
    \mat{\dfrac{v_1}{y_0} - \dfrac{x_0 v_1}{2 y_0^3}
      e}{\dfrac{v_2}{y_0} - \dfrac{x_0 v_2}{2 y_0^3}
      e}{-\dfrac{v_1}{2 y_0^3} e}{-\dfrac{v_2}{2 y_0^3} e}.
  \end{displaymath}
\end{prop}

In order to solve~\eqref{eq:diffsyst} in quasi-linear time in the precision, it
is enough to solve equation~\eqref{eq:mde} in quasi-linear time in~$n$. One
difficulty here, that does not appear in similar
works~\cite{couveignesComputingFunctionsJacobians2015,
  costaRigorousComputationEndomorphism2019} and is related to our choice of
base point at $0_A$, is that the matrix~$M$ is not invertible in $k'[[z]]$. We
can nonetheless adapt the divide-and-conquer strategy
from~\cite[§13.2]{bostanAlgorithmesEfficacesCalcul2017}.

\begin{lem}
\label{lem:det-valuation}
  The determinant
  %\begin{displaymath}
    $\det M(z) = \dfrac{x_1^0 - x_2^0}{y_1^0 y_2^0}$
  %\end{displaymath}
  has valuation one in~$z$.
\end{lem}

\begin{proof}
  We know that~$y_1^0$ and~$y_2^0$ have constant term~$\pm y_0\neq 0$.  The
  polynomials~$x_1^0$ and~$x_2^0$ have the same constant term~$x_0$, but they
  do not coincide at order~$2$: if they did, then so would~$y_1$ and~$y_2$
  because of the curve equation, and~$\isog_P$ would pull back every
  differential form on~$\Crv_F$ to zero, a contradiction.
\end{proof}

By \Cref{lem:det-valuation}, we can find $I\in\M_2\bigl(k'[[z]]\bigr)$
such that $IM = \mat{z}{0}{0}{z}$.

\begin{lem}
  \label{lem:invertible}
  Let $\kappa\geq 1$, and assume that $\chr k > \kappa+1$. Let $A = IN$. Then
  the matrix $A+\kappa$ has an invertible constant term.
\end{lem}

\begin{proof}
  By \Cref{lem:det-valuation}, the leading term of $\det(M)$ is
  $\lambda z$ for some nonzero $\lambda\in k'$.  Using
  \Cref{prop:diffsyst-linearize}, we see that the constant term of
  $\det(A+\kappa)$ is~$\lambda^2 \kappa(\kappa+1)$.
\end{proof}

\begin{prop}
  \label{prop:newton-mde}
  Let $1\leq \prc\leq 2n-1$, and assume that $\chr k = 0$ or $\chr k > \prc$. Then we
  can solve~\eqref{eq:mde} to compute $\delta x_1$ and $\delta x_2$ up
  to precision $O(z^\prc)$ using $\Otilde(\prc)$ operations in~$k'$.
\end{prop}

\begin{proof}
  Write $\theta = \vectwo{\delta x_1}{\delta x_2}$.
  Multiplying~\eqref{eq:mde} by~$I$, we obtain the equation
  \begin{displaymath}
    z \dfrac{d\theta}{dz} + (A + \kappa) \theta = B + O(z^{\prc}), \quad \text{where }
    \kappa=0.
  \end{displaymath}
  We show that~$\theta$ can be computed from this kind of equation up
  to~$O(z^\prc)$ using a divide-and-conquer strategy. If~$\prc>1$, write
  $\theta = \theta_1 + z^{\prc_1}\theta_2$ where $\prc_1 = \lfloor \prc/2\rfloor$. Then
  \begin{displaymath}
    z\dfrac{d\theta_1}{dz} + (A + \kappa)\theta_1 = B + O(z^{\prc_1})
  \end{displaymath}
  for some other series~$B$. By induction, we recover~$\theta_1$
  up to~$O(z^{\prc_1})$. Then, we have
  \begin{displaymath}
    z\dfrac{d\theta_2}{dz} + (A + \kappa + \prc_1)\theta_2 = C + O(z^{\prc-\prc_1})
  \end{displaymath}
  where~$C$ has an expression in terms of~$\theta_1$. This is enough to
  recover~$\theta_2$ up to~$O(z^{\prc-\prc_1})$, so we can recover~$\theta$ up
  to~$O(z^{\prc})$. We initialize the induction with the case $d=1$, where we
  have to solve for the constant term in
  \begin{displaymath}
    (A + \kappa)\theta = B.
  \end{displaymath}
  Since~$\theta$ starts at~$z^2$, the values of~$\kappa$ that occur are
  $2,\ldots, \prc-1$ when computing the solution of~\eqref{eq:diffsyst} up to
  precision~$O(z^\prc)$. By \Cref{lem:invertible}, the constant term of
  $A+\kappa$ is invertible. This concludes the induction. The complexity
  estimate follows from standard lemmas in computer
  algebra~\cite[Lem.~1.12]{bostanAlgorithmesEfficacesCalcul2017}.
\end{proof}

As a consequence of \Cref{prop:newton-mde}, we can indeed
solve~\eqref{eq:diffsyst} in quasi-linear time.

\begin{prop}
  \label{prop:newton}
  Let $\prc\geq 1$, and let~$k$ be a field such that $\chr k = 0$ or
  $\chr k > \prc$. Let~$E$ and~$F$ be genus~$2$ curve equations over~$k$ such
  that there exists an isogeny $\isog:\Jac(\Crv_E)\to \Jac(\Crv_F)$, and assume
  that we are given the matrix $d\isog$ in the bases of~$T_0(\Jac(\Crv_E))$
  and~$T_0(\Jac(\Crv_F))$ associated with this choice of equations. Let~$P\in
  \Crv_E(k)$ be a base point such
  that~$\varphi_P(P) = \{Q,i(Q)\}$ for some non-Weierstrass point~$Q$
  on~$\Crv_F$. Let~$k'$ be the field of definition of~$Q$, and let~$z$ be a
  uniformizer of~$\Crv_E$ at~$P$.  Then one can compute the local
  lift~$\widetilde{\isog}_P$ as power series in~$k'[[z]]$ up to
  precision~$O(z^\prc)$ using~ $\Otilde(\prc)$ operations in $k'$.
\end{prop}

\subsection{Rational reconstruction}
\label{subsec:rational}

Finally, we want to recover the rational representation $(s, p, q, r)$
of~$\isog$ at~$P$ from its power series expansion~$\widetilde{\isog}_P$ at a
finite precision. For this, we need upper bounds on the degrees of these
rational fractions.

The degrees of $s, p, q, r$ as morphisms from~$\Crv_E$ to~$\Pvar^1$ can be
computed as intersection numbers of divisors on~$\Jac(\Crv_F)$,
namely~$\isog_P(\Crv_E)$ and the polar divisors of $s$, $p$, $q$ and $r$. They
are already known in the case of an~$\ell$-isogeny.

\begin{prop}[{\cite[§6.1]{couveignesComputingFunctionsJacobians2015}}]
  \label{prop:degree-siegel}
  Let $\isog\from\Jac(\Crv_E)\to\Jac(\Crv_F)$ be an $\ell$-isogeny, and let
  $P\in\Crv_E(k)$. Let $(s,p,q,r)$ be the rational representation of $\isog$ at
  the base point~$P$. Then the degrees of $s$, $p$, $q$ and $r$
  % as morphisms from $\Crv$ to $\Pvar^1$
  are $4\ell$, $4\ell$, $12\ell$, and $8\ell$ respectively.
\end{prop}

Now assume that~$\Jac(\Crv_E)$ and~$\Jac(\Crv_F)$ have real
multiplication by $\Z_K$ given by embeddings $\iota_E,\iota_F$, and that
%\begin{displaymath}
$\isog\from \bigl(\Jac(\Crv_E),\iota_E\bigr) \to \bigl(\Jac(\Crv_F),\iota_F\bigr)$
%\end{displaymath}
is a $\beta$-isogeny. Denote the theta divisors on~$\Jac(\Crv_E)$
and~$\Jac(\Crv_F)$ by~$\Theta_E$ and~$\Theta_F$ respectively, and denote by
$\eta_P\from\Crv_E\to\Jac(\Crv_E)$ the map
$Q\mapsto[Q-P]$. Then~$\eta_P(\Crv_E)$ is algebraically equivalent to~$\Theta_E$.

\begin{lem}
  \label{lem:spqr-poles}
  The polar divisors of $s, p, q, r$ as rational functions
  on~$\Jac(\Crv_F)$ are algebraically equivalent to $2\Theta_F$,
  $2\Theta_F$, $6\Theta_F$ and $4\Theta_F$ respectively.
\end{lem}

\begin{proof}
  See \cite[§6.1]{couveignesComputingFunctionsJacobians2015}. For
  instance, $s = x_1 + x_2$ has a pole of order~$1$ along each of the
  two divisors $\bigl\{(\infty_{\pm}, Q) \st Q\in \Crv_F\bigr\}$,
  where~$\infty_\pm$ are the two points at infinity on~$\Crv_F$,
  assuming that we choose a degree 6 hyperelliptic model
  for~$\Crv_F$. Each of these divisors is algebraically equivalent
  to~$\Theta_F$. The proof for $p$, $q$, and~$r$ is similar.
\end{proof}

By \Cref{thm:NS-End}, if $(A,\iota)$ is a p.p.~abelian surface with real
multiplication by~$\Z_K$, then we have an injective map $\Z_K\to \NS(A)$ given
by $\alpha\mapsto\Lb_{A}({\iota(\alpha)})$.

\begin{lem}
  \label{lem:image-divisor}
  Let~$\isog$ be a $\beta$-isogeny as above. Then the
  divisor~$\isog_P(\Crv_E)$ is algebraically equivalent to the divisor
  corresponding to the line bundle
  $\displaystyle\Lb_{\Jac(\Crv_F)}({\iota_F(\conj{\beta})})$.
\end{lem}

\begin{proof}
  Since~$\Jac(\Crv_F)$ is a smooth surface, the divisor~$\isog_P(\Crv_E)$ corresponds
  to a line bundle on~$\Jac(\Crv_F)$.  By \Cref{thm:NS-End}, this line bundle
  is algebraically equivalent to~$\Lb_{\Jac(\Crv_F)}({\iota_F(\alpha)})$ for
  some~$\alpha\in\smash{\End^\dagger}(\Jac(\Crv_F))$. Consider
  $\isog^*\bigl(\isog_P(\Crv_E)\bigr)$ as a divisor on~$\Jac(\Crv_E)$. By
  definition, we have
  \begin{displaymath}
    \isog^*\bigl(\isog_P(\Crv_E) \bigr) = \sum_{x\in\ker\isog} \bigl(x + \eta_P(\Crv_E) \bigr).
  \end{displaymath}
  Therefore, up to algebraic equivalence,
  \begin{displaymath}
    \isog^*\bigl(\isog_P(\Crv_E) \bigr) = (\#\ker\isog)\Theta_E = N_{K/\Q}(\beta)\Theta_E.
  \end{displaymath}
  By \Cref{def:beta-isog}, the pullback~$\isog^*\Theta_F$ corresponds to the
  line bundle~$\Lb_{\Jac(\Crv_E)}({\iota_E(\beta)})$ up to algebraic equivalence.
  Therefore, for every~$\gamma\in\Z_K$,
  \begin{displaymath}
    \isog^*\Lb_{\Jac(\Crv_F)}({\iota_F(\gamma)}) = \Lb_{\Jac(\Crv_E)}({\iota_E(\gamma\beta)}).
  \end{displaymath}
  By \Cref{thm:NS-End} applied on~$\Jac(\Crv_E)$, we have
  $\alpha\beta = N_{K/\Q}(\beta)$, so $\alpha = \conj{\beta}$.
\end{proof}

The next step is to compute the intersection number of~$\Theta_F$ and
the divisor corresponding to $\Lb_{\Jac(\Crv_F)}({\iota_F(\alpha)})$
on~$\Jac(\Crv_F)$, for every~$\alpha\in\Z_K$.

\begin{prop}
  \label{prop:theta-intersection}
  Let~$(A,\iota)$ be a p.p.~abelian surface with real
  multiplication by~$\Z_K$, and let~$\Theta$ be its theta
  divisor. Then for all~$\alpha\in\Z_K$, we have
  \begin{displaymath}
    \bigl(\Lb_A({\iota(\alpha)})\cdot \Theta\bigr)^2 = \Tr_{K/\Q}(\alpha)^2.
  \end{displaymath}
\end{prop}

\begin{proof}
  By \cite[Rem.~16]{kaniEllipticSubcoversCurve2019}, the quadratic form
  \begin{displaymath}
    D \mapsto (D\cdot\Theta)^2 - 2(D\cdot D)
  \end{displaymath}
  on~$\NS(A)$ corresponds via \cref{thm:NS-End} to
  the quadratic form on~$\Z_K$ given by
  \begin{displaymath}
    \alpha\mapsto 2\Tr_{K/\Q}(\alpha^2) - \dfrac{1}{2}\Tr_{K/\Q}(\alpha)^2.
  \end{displaymath}
  Thus, for every $\alpha = a + b\sqrt{\Delta}\in\Z_K$, we have
  \begin{displaymath}
    \bigl(\Lb_A({\iota(\alpha)})\cdot\Theta \bigr)^2
    - 2\, \bigl(\Lb_A({\iota(\alpha)})\cdot \Lb_A({\iota(\alpha)})\bigr) = 2\Tr(\alpha^2)
    - \dfrac{1}{2}\Tr(\alpha)^2 = 4b^2\Delta.
  \end{displaymath}
  On the other hand, the Riemann--Roch
  theorem~\cite[Thm.~13.3]{milneAbelianVarieties1986} gives
  \begin{displaymath}
    \bigl(\Lb_A({\iota(\alpha)})\cdot \Lb_A({\iota(\alpha)}) \bigr)
    = 2\, \chi\bigl(\Lb_A({\iota(\alpha)}) \bigr) = 2\sqrt{\deg(\iota(\alpha))} =
    2(a^2 - b^2\Delta). \qedhere
  \end{displaymath}
\end{proof}

\begin{prop}
  \label{prop:spqr-degrees}
  Let~$\isog$ be a $\beta$-isogeny as above, and let $(s, p, q, r)$ be the
  rational representation of~$\isog$ at~$P$. Then the degrees of $s$, $p$, $q$,
  and $r$ as morphisms from~$\Crv_F$ to~$\Pvar^1$ are $2\Tr_{K/\Q}(\beta)$,
  $2\Tr_{K/\Q}(\beta)$, $6\Tr_{K/\Q}(\beta)$ and $4\Tr_{K/\Q}(\beta)$
  respectively.
\end{prop}

\begin{proof}
  The degrees of $s,p,q$ and~$r$ can be computed as the intersection of the
  polar divisors from \Cref{lem:spqr-poles} and the
  divisor~$\isog_P(\Crv_E)$. By \Cref{lem:image-divisor}, the line
  bundle associated with~$\isog_P(\Crv_E)$, up to algebraic equivalence,
  is~$\Lb_{\Jac(\Crv_F)}({\iota_F(\conj{\beta})})$. Its intersection number
  with~$\Theta_F$ is nonnegative, hence by
  \Cref{prop:theta-intersection}, we have
  \begin{displaymath}
    \bigl(\isog_P(\Crv_E)\cdot\Theta_F\bigr) = \Tr_{K/\Q}(\conj{\beta}) = \Tr_{K/\Q}(\beta).
    \qedhere
  \end{displaymath}
\end{proof}

In order to reformulate \cref{prop:degree-siegel,prop:spqr-degrees} in terms of
concrete degrees of rational fractions, we use the following lemma.

\begin{lem}
  Let $s\from\Crv_E\to\Pvar^1$ be a morphism of degree~$d$.
  \begin{enumerate}
  \item \label{it:invariant}
    If~$s$ is invariant under the hyperelliptic involution~$i$,
    then we can write
    % \begin{displaymath}
     $ s(u,v) = X(u)$
     % \end{displaymath}
    where the degree of~$X$ is bounded by~$d/2$.
  \item \label{it:noninvariant}
    In general, let~$X$, $Y$ be the rational fractions such that
    \begin{displaymath}
      s(u,v) = X(u) + v\, Y(u).
    \end{displaymath}
    Then the degrees of~$X$ and~$Y$ are bounded by~$d$ and~$d-3$
    respectively.
  \end{enumerate}
\end{lem}

\begin{proof}
  For~\eqref{it:invariant}, use the fact that the function~$u$ has
  degree~$2$. For~\eqref{it:noninvariant}, write
  \begin{displaymath}
    s(u,v) + s(u,-v) = 2X(u),\quad \dfrac{s(u,v) - s(u,-v)}{v} = 2Y(u).
  \end{displaymath}
  The degrees of these morphisms are bounded by~$2d$ and~$2d-6$
  respectively.
\end{proof}

We can thus summarize the rational reconstruction step as follows.

\begin{prop}
  \label{prop:reconstruction}
  Let $\widetilde{\isog}_P$ and $\widetilde{\isog}_{i(P)}$ be local
  lifts of~$\isog_P$ at~$P$ and~$i(P)$ in the uniformizers~$z$
  and~$i(z)$. Let $\prc = 8\ell+1$ in the Siegel case, and
  $\prc = 4\Tr_{K/\Q}(\beta)+1$ in the Hilbert case.  Then,
  given~$\widetilde{\isog}_P$ and~$\widetilde{\isog}_{i(P)}$ to
  precision~$O(z^\prc)$, we can compute the rational representation
  of~$\isog$ at~$P$ within~$\Otilde(\prc)$ operations in~$k'$.
\end{prop}

\begin{proof}
  It is enough to recover the rational fractions~$s$ and~$p$;
  afterwards,~$q$ and~$r$ can be deduced from the equation
  of~$\Crv_F$.

  First, assume that~$P$ is a Weierstrass point of~$\Crv_E$. Then $s$ and~$p$
  are invariant under the hyperelliptic involution. Therefore, we have to
  recover rational fractions in~$u$ of degree $d\leq 2\ell$
  (resp.~$d\leq \Tr_{K/\Q}(\beta)$). This can be done in quasi-linear time from
  their power series expansion to precision
  $O(u^{2d+1})$~\cite[§7.1]{bostanAlgorithmesEfficacesCalcul2017}.  Since~$u$
  has valuation~$2$ in~$z$, it suffices to compute~$\widetilde{\isog}_P$ to
  precision~$O(z^{4d+1})$.

  Second, assume that~$P$ is not a Weierstrass point of~$\Crv_E$. Then the
  series defining $s(u,-v)$ and $p(u,-v)$ are given
  by~$\widetilde{\isog}_{i(P)}$. It is enough to compute rational fractions of
  degree $d\leq 4\ell$ (resp.~$d\leq 2\Tr_{K/\Q}(\beta)$) in~$u$. Since~$u$ has
  valuation~$1$ in~$z$, this can be done in quasi-linear time
  if~$\widetilde{\isog}_P$ and~$\widetilde{\isog}_{i(P)}$ are known up to
  precision~$O(z^{2d+1})$.
\end{proof}

\section{Summary of the algorithm}
\label{sec:summary}

Now let us summarize the isogeny algorithm and prove \cref{thm:main}. We also
state an analogous result in the case of $\beta$-isogenies
(\cref{thm:proved-main-hilbert}).

Let~$k$ be a field, and let~$A,A'$ be two p.p.~abelian surfaces~$A,A'$
over~$k$. We specify them by giving their Igusa invarants~$j$ and~$j'$, as well
as a genus~$2$ curve equation~$E$ such that $\Jac(\Crv_E) = A$ to resolve
twisting ambiguities. In the Siegel case, we assume that $A$ and~$A'$ are
$\ell$-isogenous over~$k$ for some prime~$\ell$. In the Hilbert case, we assume
that $A$ and~$A'$ have real multiplication by~$\Z_K$ for some real quadratic
field~$K$ and are $\beta$-isogenous for some totally positive
prime~$\beta\in\Z_K$. We then compute the isogeny $\isog:A\to A'$ as follows.

\begin{algo}~
  \label{algo:main}
  \begin{enumerate}
  \item Construct a genus~$2$ curve equation~$F$ over~$k$ such
    that~$A'= \Jac(\Crv_F)$ over~$\kbar$ using Mestre's
    algorithm~\cite{mestreConstructionCourbesGenre1991}. In the Hilbert case,
    use \Cref{algo:hilb-curve-2} to ensure that $E$ and~$F$ are
    potentially Hilbert-normalized.
  \item Compute at most~$4$ candidates for the tangent matrix~$d\isog$
    of~$\isog$ using \cref{prop:norm-matrix} or~\ref{prop:norm-matrix-hilbert}.
    Run the rest of the algorithm on each candidate.
  \item Make a change of basis to ensure that~$E$, $F$ and~$d\isog$ are defined
    over~$k$ (but not necessarily Hilbert-normalized.)
  \item Choose a suitable base point~$P$ on~$\Crv_E$ using \cref{prop:imagept}
    and compute the power series~$\widetilde{\isog}_P$
    and~$\widetilde{\isog}_{i(P)}$ to precision~$O(z^{8\ell+1})$
    or~$O(z^{4\Tr_{K/\Q}(\beta)+1})$ respectively, following \cref{prop:newton}.
  \item \label{step:reconstruction} Try to recover the rational
    representation of~$\isog$ at~$P$ using
    \Cref{prop:reconstruction}. Output the result if rational
    fractions of the correct degrees are found.
  \end{enumerate}
\end{algo}

\begin{thm}
  \label{thm:main_proved}
  Let~$\ell$ be a prime, and let~$k$ be a field such that $\chr k = 0$ or
  $\chr k > 8\ell+ 1$.  Let~$\mathbf{U}\subset\Atwo(k)$ be the open set
  consisting of p.p.~abelian surfaces~$\av$ such that $\Aut(\av)\iso\{\pm 1\}$
  and $j_3(A)\neq 0$.  Let $A,A' \in \mathbf{U}$, let $j,j'$ be their Igusa
  invariants, and let~$E$ be a genus~$2$ curve equation over~$k$ such
  that~$A = \Jac(\Crv_E)$. Assume that~$A$ and~$A'$ are $\ell$-isogenous
  over~$k$, and that the subvariety of $\Avar^3\times \Avar^3$ cut out by the
  Siegel modular equations $\Psi_{\ell,i}$ for $1\leq i\leq 3$ is normal
  at~$(j, j')$. Then, given~$j,j'$ and~$E$ as well as the derivatives
  of the Siegel modular equations of level~$\ell$ at~$(j,j')$, \Cref{algo:main}
  succeeds and returns
  \begin{enumerate}
  \item a genus~$2$ curve equation $F$ over~$k$ such that~$A' = \Jac(\Crv_F)$,
  \item a point $P\in \Crv_E(k')$ where~$k'/k$ is a quadratic extension,
  \item the rational representation $(s,p,q,r)\in k'(u,v)^4$ at the base
    point~$P$ of an $\ell$-isogeny $\isog\from \Jac(\Crv_E)\to\Jac(\Crv_F)$
    defined over~$k$.
  \end{enumerate}
  This algorithm costs~$\Otilde(\ell)$ elementary operations and~$O(1)$ square
  roots in~$k'$.
\end{thm}

\begin{proof}
  Mestre's algorithm returns a curve equation~$F$ defined over~$k$, and
  costs~$O(1)$ operations in~$k$ and~$O(1)$ square roots. Under our hypotheses,
  $\isog$ is generic by \cref{prop:moduli-generic}, so \cref{prop:norm-matrix}
  allows us to recover $\Sym^2(d\isog)$ using $O(1)$ operations in $k$, so we
  recover~$d\isog$ up to sign using~$O(1)$ square roots and elementary
  operations.  We can twist~$F$ in a unique way so that~$d\isog$ is defined
  over~$k$. Then we must have~$A = \Jac(\Crv_F)$ over~$k$. Given our hypothesis
  on~$\chr k$, we can compute the local lifts and perform the rational
  reconstruction in~$\Otilde(\ell)$ operations in~$k'$.
\end{proof}

In the Hilbert case, \cref{thm:main_proved} has the following analogue.

\begin{thm}
  \label{thm:proved-main-hilbert}
  Let~$K$ be a real quadratic field and~$\beta\in \Z_K$ a totally positive
  prime. Let~$k$ be a field such that $\chr k=0$ or
  $\chr k > 4\Tr_{K/\Q}(\beta)+1$. Let~$A,A'\in \mathbf{U}$ be p.p.~abelian
  surfaces over~$k$ with real multiplication by~$\Z_K$, let~$j,j'$ be their
  Igusa invariants, and let~$E$ be a curve equation over~$k$ such
  that~$A = \Jac(\Crv_E)$.  Assume that $A$ and~$A'$ are $\beta$-isogenous but
  not~$\betabar$-isogenous, and that the subvariety of~${\Avar^3\times\Avar^3}$
  cut out by the Hilbert modular equations of level~$\beta$ and the Humbert
  equation is normal at~$(j,j')$.  Then, given~$j,j',E$, and the derivatives of
  the Hilbert modular equations of level~$\beta$ at~$(j,j')$, \cref{algo:main}
  succeeds and returns
  \begin{enumerate}
  \item a genus~$2$ curve equation~$F$ over~$k$ such that $A'=\Jac(\Crv_F)$,
  \item a point $P\in \Crv_E(k')$ where~$k'/k$ is a quadratic extension,
  \item at most~$4$ quadruples $(s,p,q,r)\in k'(u,v)^4$, one of which is the rational representation
     at the base point~$P$ of a~$\beta$-isogeny
    $\isog\from \Jac(\Crv_E)\to\Jac(\Crv_F)$ defined over~$k$.
  \end{enumerate}
  This algorithm costs $\Otilde\bigl(\Tr_{K/\Q}(\beta)\bigr) + O_K(1)$
  elementary operations and~$O(1)$ square roots in~$k'$. The implied constants,
  except in~$O_K(1)$, are independent of~$K$.
\end{thm}

\begin{proof}
  By \cref{cor:generic-hilbert}, the isogeny $\isog:A\to A'$ is generic, and
  defined over~$k$. Using \cref{algo:hilb-curve-2}, we obtain potentially
  Hilbert-normalized curves equations~$E'$ and~$F'$ defined over a common
  quadratic extension of~$k$; this costs $O_K(1)$ elementary operations
  and~$O(1)$ square roots in~$k$. We obtain four candidates for $\pm
  d\isog$. For each candidate, we now make a change of variables to~$E$ and the
  (not necessarily Hilbert-normalized) curve equation~$F$ output by Mestre's
  algorithm, so that both~$\Crv_E$ and~$\Crv_F$ are defined over~$k$, and
  twist~$\Crv_F$ so that~$d\isog$ is also defined over~$k$. We then have
  $A' = \Jac(\Crv_F)$, and we continue as in the Siegel case. For the correct
  value of~$d\isog$, rational reconstruction will succeed and output fractions
  of the correct degrees.
\end{proof}

\begin{rem}
In the Hilbert case, we expect that the algorithm returns only one answer for the
rational representation of~$\isog$ at~$P$, as the incorrect candidates
for~$d\isog$ should lead to garbage in Step~\eqref{step:reconstruction} of the
algorithm. Note that testing for correctness of the output might be more
expensive than the isogeny algorithm itself.
\end{rem}

  \section{The case \texorpdfstring{$K = \Q(\sqrt{5})$}{K=Q(√5)}}
\label{sec:Qr5}

In this final section, we present a variant of our isogeny algorithm in the
case of p.p.~abelian varieties with real multiplication by~$\Z_K$ where
$K = \Q(\sqrt{5})$.  We work over~$\C$, but the methods of §\ref{sec:moduli}
show that the computations remain valid over a general base. The Humbert
surface attached to~$K$ is rational: its function field can be generated by
only two elements called the Gundlach invariants.  Having only two
coordinates reduces the size of modular equations, allowing us to illustrate
our algorithm with an example of a cyclic isogeny of degree~$11$ over a finite
field.

\subsection{Hilbert modular forms for \texorpdfstring{$K = \Q(\sqrt{5})$}{K=Q(√5)}}

We keep the notation used to describe the Hilbert embedding
in~§\ref{subsec:hilbert-siegel}. Hilbert modular forms have Fourier expansions
in terms of
\begin{displaymath}
  w_1 \defi \exp\bigl(2\pi i(e_1 t_1 + \conj{e_1}t_2)\bigr)
  \quad\text{and}\quad
  w_2 \defi \exp\bigl(2\pi i(e_2 t_1 + \conj{e_2}t_2)\bigr).
\end{displaymath}
We use this notation and the term \emph{$w$-expansions} to avoid any confusion with
$q$-expansions of Siegel modular forms.  Apart from the constant term, a term
in~$w_1^a w_2^b$ can appear with a nonzero coefficient only when~$ae_1 + be_2$
is a totally positive element of~$\Z_K$.  Since~$e_1 = 1$ and~$e_2$ has
negative norm, for a given~$a$, only finitely many~$b$'s appear. Therefore, we
can consider truncations of~$w$-expansions as elements
of~$\C[w_2,w_2^{-1}][[w_1]]$ modulo an ideal of the form~$(w_1^\prc)$.

\begin{thm}[{\cite{nagaokaRingIlbertModular1983}}]
  \label{thm:hilbert-structure}
  The graded $\C$-algebra of symmetric Hilbert modular forms of even parallel
  weight for~$K = \Q(\sqrt{5})$ is generated by three elements~$G_2$,
  $F_6$, $F_{10}$ of respective weights~$2$, $6$ and $10$,
  with~$w$-expansions
  \begin{align*}
    G_2(t) &= 1 + (120 w_2 + 120) w_1 \\
           &\quad + \bigl(120 w_2^3 + 600 w_2^2 + 720
             w_2 + 600 + 120w_2^{-1}\bigr)w_1^2 + O(w_1^3),\\
    F_{6}(t) &= (w_2 + 1) w_1 + \bigl(w_2^3 + 20 w_2^2 - 90 w_2 + 20 +
               w_2^{-1}\bigr) w_1^2 + O(w_1^3), \\
    F_{10}(t) &= (w_2^2 - 2 w_2 + 1) w_1^2 + O(w_1^3).
  \end{align*}
\end{thm}

Following~\cite{milioModularPolynomialsIlbert2020}, we define the
\emph{Gundlach invariants} for~$K = \Q(\sqrt{5})$ as
\begin{displaymath}
  g_1 \defi \dfrac{G_2^5}{F_{10}}\quad \text{and}\quad  g_2 \defi \dfrac{G_2^2 F_6}{F_{10}}.
\end{displaymath}
Recall that we denote by~$\sigma$ the involution $(t_1,t_2)\mapsto(t_2,t_1)$
of~$\Htwo(\C)$. The Gundlach invariants define a birational map
$\Htwo(\C)/\sigma\to\C^2$.

By \Cref{prop:mf-pullback}, the pullbacks of the Siegel modular forms
$\psi_4$, $\psi_6$, $\chi_{10}$ and~$\chi_{12}$ via the Hilbert
embedding~$H$ are symmetric Hilbert modular forms of even weight, so
they have expressions in terms of~$G_2, F_6, F_{10}$. These
expressions can be computed using linear algebra on Fourier
expansions~\cite[Prop.~3.2]{lauterComputingGenusCurves2011}: in our
case, the Hilbert embedding is defined by $e_1 = 1$,
$e_2 = (1 - \sqrt{5})/2$, so
\begin{displaymath}
  q_1 = w_1, \quad q_2 = w_2, \quad q_3 =
  w_1 w_2.
\end{displaymath}
As a corollary, we obtain the expression for the pullback of the Igusa
invariants.

\begin{prop}[{\cite[Prop.~4.5]{lauterComputingGenusCurves2011}}]
  \label{prop:igusa-pullback}
  In the case~$K = \Q(\sqrt{5})$, we have
  \begin{displaymath}
    \begin{aligned}
      H^*j_1 &= 8 g_1 \Bigl(3 \dfrac{g_2^2}{g_1} - 2\Bigr)^5, \\
      H^*j_2 &= \dfrac{1}{2} g_1 \Bigl( 3 \dfrac{g_2^2}{g_1} -
        2\Bigr)^3, \\
      H^*j_3 &= \dfrac{1}{8} g_1 \Bigl( 3 \dfrac{g_2^2}{g_1} -
        2\Bigr)^2 \Bigl( 4 \dfrac{g_2^2}{g_1} +
        2^53^2\dfrac{g_2}{g_1} - 3\Bigr).
    \end{aligned}
  \end{displaymath}
\end{prop}

Let $\beta\in\Z_K$ be a totally positive prime. We define the \emph{Hilbert
  modular equations of level~$\beta$} in terms of Gundlach invariants to be
the irreducible polynomials
$\Psi_{\beta,1},\Psi_{\beta,2} \in \Q[G_1, G_2,G_1',G_2']$ with the following
properties:
\begin{itemize}
\item $\Psi_{\beta,1}\in \Q[G_1,G_2,G_1']$ is the (non-monic) minimal
  polynomial of the meromorphic function~$g_1(t/\beta)$ over the field
  $\C\bigl(g_1(t), g_2(t)\bigr)$,
\item We have $\deg_{G_2'} \Psi_{\beta,2} = 1$ and an equality of meromorphic
  functions
  \begin{displaymath}
    \ g_2(t/\beta) = \Psi_{\beta,2}\bigl(g_1(t),g_2(t),g_1(t/\beta)\bigr).
  \end{displaymath}
\end{itemize}
These modular equations have been computed in full up to
$N_{K/\Q}(\beta)=41$~\cite{milioDatabaseModularPolynomials2016}.

\subsection{Hilbert-normalized curve equations.}

We give another method to reconstruct such equations using the pullback of the
modular form~$\chi_{6,8}$ as a Hilbert modular form.  We continue to use the
notation of~§\ref{subsec:hilbert-siegel}.

\begin{prop}
  \label{prop:f86-pullback}
  Define the functions~$b_i(t)$ for $0\leq i\leq 6$ on~$\Half_1^2$ by
  \begin{displaymath}
    \det\nolimits^8\Sym^6(R)\,
    \chi_{6,8}\bigl(H(t)\bigr)
    = \sum_{i = 0}^6 b_i(t)\, x^i.
  \end{displaymath}
  Then $b_2$ and $b_4$ are identically zero, and we have
  \begin{displaymath}
    \begin{aligned}
      b_3^2 &= 4 F_{10} F_6^2, \phantom{\dfrac{1}{25}}\\
      b_1 b_5 &= \dfrac{36}{25} F_{10} F_6^2 - \dfrac{4}{5} F_{10}^2 G_2,\\
      b_0 b_6 &= \dfrac{-4}{25} F_{10} F_6^2 + \dfrac{1}{5} F_{10}^2 G_2,\\
      b_3 \bigl(b_0^2 b_5^3 + b_1^3 b_6^2 \bigr) &= 123 F_{10}^3 F_6 -
      \dfrac{32}{25} F_{10}^2 F_6^2 G_2^2 + \dfrac{288}{125} F_{10}
      F_6^4 G_2 - \dfrac{3456}{3125} F_6^6.
    \end{aligned}
  \end{displaymath}
\end{prop}

\begin{proof}
  By \Cref{prop:mf-pullback}, each coefficient~$b_i$ is a Hilbert modular form
  for~$K$ of weight~$(8+i,14-i)$, and~$\sigma$ exchanges~$b_i$
  and~$b_{6-i}$. From the~$q$-expansion for~$\chi_{6,8}$, we compute
  the~$w$-expansions of the~$b_i$'s, and use linear algebra to identify
  symmetric combinations of the~$b_i$'s of even weight in terms of the
  generators~$G_2, F_6, F_{10}$. We find that $b_2b_4 = 0$, and thus both~$b_2$
  and~$b_4$ must be identically zero.
\end{proof}

By construction, for each~$t\in\Half_1^2$, the genus~$2$ curve equation
$\sum_{i=0}^6 b_i(t) x^i$ is potentially Hilbert-normalized. Thus, we obtain an
alternative to \cref{algo:hilb-curve-2} for the construction of a potentially
Hilbert-normalized curve equation given a tuple of Igusa invariants $(j_1,j_2,j_3)$ that
does not use the Humbert equation.

\begin{algo}~
  \label{algo:hilb-curve-1}
  \begin{enumerate}
  \item Compute the Gundlach invariants $(g_1, g_2)$ mapping to $(j_1,j_2,j_3)$ via
    $H$ with \Cref{prop:igusa-pullback}, and choose values for
    $G_2, F_6, F_{10}$ giving these invariants.
  \item Compute $b_3^2$, $b_1 b_5$, etc.\ using
    \Cref{prop:f86-pullback}.
  \item Recover values for the coefficients as follows. Choose a square root
    for~$b_3$. Choose a arbitrary value for~$b_1$, which gives~$b_5$. Finally, solve a
    quadratic equation to find~$b_0$ and~$b_6$.
  \end{enumerate}
\end{algo}

We can always choose values $G_2, F_6, F_{10}$ such that~$b_3^2$ is a square in
$k$; then, the output of \cref{algo:hilb-curve-1} is defined over a quadratic
extension of $k$.

\subsection{Computing the tangent matrix.} Using Gundlach invariants instead of
Igusa invariants, we can compute the tangent matrix of a $\beta$-isogeny
without any reference to the Hilbert embedding into the Siegel threefold. To
formulate this result, we develop a notion of covariant attached to a Hilbert
modular form that one can evaluate on a Hilbert-normalized curve equation, as
announced in~§\ref{subsec:explicit-hilbert}.

First, if~$(A,\iota)$ is a p.p.~abelian surface with real multiplication
by~$\Z_K$, if~$\omega$ is a Hilbert-normalized basis of~$\Omega^1(A)$, and
if~$f$ is a Hilbert modular form of weight~$(k_1,k_2)$, then the
quantity~$f(A,\iota,\omega)$ makes sense. To define it, choose~$t\in\Half_1^2$
and an isomorphism $\eta:(A,\iota)\to (A_K(t),\iota_K(t))$. Then the matrix
of~$\eta^*$ in the bases~$\omega_K(t)$ and~$\omega$ is a diagonal
matrix~$\Diag(r_1,r_2)$, and we set
\[
  f(A,\iota,\omega) \defi r_1^{k_1} r_2^{k_2} f(t).
\]
This allows us to define the ``covariant''~$\Cov_K(f)$ as the rule which, to
genus~$2$ curve equation~$E$ that is Hilbert-normalized for a real
multiplication embedding~$\iota$ on~$\Jac(\Crv_E)$,
associates~$f(\Jac(\Crv_E),\iota,\omega_E)$.

Next, we note that if~$f$ is a Hilbert modular function of weight~$0$, its
partial derivatives
\[
  \frac{1}{\pi i} \frac{\partial f}{\partial t_1}
  \quad\text{and} \quad \frac{1}{\pi i} \frac{\partial f}{\partial t_2},
\]
where~$(t_1,t_2)$ are the coordinates on~$\Half_1^2$, are Hilbert modular
functions of weight $(2,0)$ and~$(0,2)$ respectively. This is easily seen by
differentiating the equation $f(\gamma t) = f(t)$, for
all~$\gamma\in \Gamma_K$, with respect to~$t$. As a consequence, the function
\begin{displaymath}
  DG(t) \defi \Bigl( \frac{1}{\pi i} \frac{\partial g_k}{\partial t_l} \Bigr)_{1\leq k, l\leq 2}
\end{displaymath}
is a ``matrix-valued'' Hilbert modular function; its weight is the
representation~$\rho$ of~$\GL_1(\C)\times\GL_1(\C)$ on~$\Mat_{2\times 2}(\C)$ given by
\begin{displaymath}
  \rho(r_1,r_2): M\mapsto M \Diag(r_1^2, r_2^2).
\end{displaymath}
We will formulate the computation of the tangent matrix~$d\isog$ in terms of
the associated covariant~$\Cov_K(DG)$. This raises the question of how to
evaluate this covariant on a given potentially Hilbert-normalized curve
equation. Fortunately, we can directly relate this to our study of~$\Cov(DJ)$
on the Siegel threefold. Let~$M(g_1,g_2)$ be the $3\times 2$ matrix obtained by
differentiating \cref{prop:igusa-pullback}, so that
\begin{displaymath}
  DH^\ast\!J(t) \defi \Bigl(\frac{1}{\pi i} \frac{\partial H^*j_k}{\partial t_l} \Bigr)_{1\leq k\leq 3,\,1\leq l\leq 2}
  = M(g_1(t),g_2(t))\cdot  DG(t).
\end{displaymath}

\begin{prop}
  \label{prop:hilbert-cov}
  Let~$E$ be a potentially Hilbert-normalized genus~$2$ curve equation, and
  let~$(g_1,g_2)$ be the Gundlach invariants of~$\Jac(\Crv_E)$. Then we have
  \begin{displaymath}
    \Cov(DJ)(E)\cdot T = M(g_1,g_2)\cdot \Cov_K(DG)(E).
  \end{displaymath}
\end{prop}

\begin{proof}
  Equip~$\Jac(\Crv_E)$ with the real multiplication embedding for which~$E$ is
  Hilbert-normalized, and choose an isomorphism~$\eta: \Jac(\Crv_E)\to A_K(t)$
  for some~$t\in \Half_1^2$. Let~$r\in \GL_2(\C)$ be the matrix of~$\eta^*$ in
  the bases~$\omega_K(t)$ and~$\omega_E$, and let~$\tau = H(t)$. By the
  expression of the Hilbert embedding, the columns of~$DH^*\!J(t)$ contain the
  derivatives of the Igusa invariants at~$\tau$ in the directions
  \begin{displaymath}
    \frac{1}{\pi i} R^t \mat{1}{0}{0}{0} R \quad \text{and}\quad
    \frac{1}{\pi i} R^t \mat{0}{0}{0}{1} R.
  \end{displaymath}
  Therefore, we have
  \begin{align*}
    DH^*\!J(t) &= \djdtau(\tau)\cdot \Sym^2(R^t)\cdot T \quad &\text{by \cref{lem:sym2-DJ}}\\
               &= \Cov(DJ)(E) \cdot \Sym^2(r^{-t}) \cdot T \quad &\text{by \cref{lem:dt-dtau}}\\
               &= \Cov(DJ)(E) \cdot T \cdot r^{-2} \quad &\text{as~$r$ is diagonal.}
  \end{align*}
  On the other hand,
  \begin{displaymath}
    DH^*\!J(t) = M(g_1,g_2) \cdot DG(t) = M(g_1,g_2) \cdot \Cov_K(DG)(E) \cdot r^{-2}. \qedhere
  \end{displaymath}
\end{proof}

Since the Igusa invariants define a birational map from $\Htwo(\C)/\sigma$ to
the Humbert surface, the matrix $M(g_1,g_2)$ generically has rank~$2$. Thus we
can combine \cref{prop:hilbert-cov} with the expression of~$DJ$ as a covariant
to evaluate $\Cov_K(DG)(E)$.

Now we can formulate an alternative to \cref{prop:norm-matrix-hilbert} to
compute the tangent matrix~$d\isog$. We define the $2\times 2$ matrices
\begin{displaymath}
  D\Psi_{\beta,L} \defi \left(\dfrac{\partial\Psi_{\beta,n}}{\partial
      G_k}\right)_{1\leq n,k \leq 2} \quad\text{and}\quad
  D\Psi_{\beta,R} \defi \left(\dfrac{\partial\Psi_{\beta,n}}{\partial
      G_k'}\right)_{1\leq n,k \leq 2}.
\end{displaymath}

\begin{prop}
  Let~$\isog:A\to A'$ be a $\beta$-isogeny between p.p.~abelian surfaces with
  real multiplication by~$\Z_K$. Let~$g$ (resp.~$g'$) denote the Gundlach
  invariants of~$A$ (resp.~$A'$), and let~$E$ (resp.~$F$) be a
  Hilbert-normalized curve equations for~$A$ (resp.~$A'$). Assume that $(A,A')$
  is generic in the sense that the matrices $D\Psi_{\beta,L}(g,g')$,
  $D\Psi_{\beta,R}(g,g')$, $\Cov_K(DG)(E)$ and~$\Cov_K(DG)(F)$ are
  invertible. Then the only~$\beta$-isogenies from~$A$ to~$A'$ are $\pm\isog$,
  and we have
  \begin{displaymath}
    (d\isog)^2 = - \Diag(\beta, \conj{\beta}) \cdot \Cov(DG)(F)^{-1} \cdot D\Psi_{\beta,R}(g,g')^{-1}
    \cdot D\Psi_{\beta,L}(g,g') \cdot \Cov_K(DG)(E).
  \end{displaymath}
\end{prop}

\begin{proof}
  Left to the reader: one can follow the proof of \cref{prop:norm-matrix}.
\end{proof}

Using the formalism of §\ref{sec:moduli}, one can prove that $(A,A')$ is
generic if~$A$ and~$A'$ have only~$\Z_K^\times$ as automorphisms, have
$g_1\neq 0$, and if the modular equations in terms of Gundlach invariants cut
out a normal subvariety of~$\Avar^2\times\Avar^2$ at~$(g,g')$.

\subsection{An example of a cyclic isogeny}
\label{sec:ex}

We illustrate our algorithm in the Hilbert case with
$K = \Q(\sqrt{5})$ by computing a $\beta$-isogeny between Jacobians
with real multiplication by $\Z_K$, where
\begin{displaymath}
  \beta = 3 + \dfrac{1 + \sqrt{5}}{2} \in\Z_K,\quad N_{K/\Q}(\beta)=11, \quad \Tr_{K/\Q}(\beta)=7.
\end{displaymath}
We work over the prime finite field $k = \F_{56311}$, whose
characteristic is large enough for our purposes. We choose a
trivialization of $\Z_K\otimes k$, in other words a square root
of~$5$ in~$k$, such that $\beta=26213$.

Consider the Gundlach invariants
\begin{displaymath}
  (g_1,g_2) = \bigl(23, 56260\bigr), \quad (g_1', g_2') = \bigl(8, 36073\bigr).
\end{displaymath}
%The corresponding Igusa--Streng invariants are
%\begin{displaymath}
%  (j_1,j_2,j_3) = \Bigl(14030, 9041, 56122\Bigr),\quad (j_1',j_2',j_3') =
%  \Bigl(13752, 42980, 12538\Bigr);
%\end{displaymath}
%they lie on the Humbert surface, as expected.
\Cref{algo:hilb-curve-1} provides the Hilbert-normalized curve equations
\begin{displaymath}
  \begin{aligned}
    \Crv_E\defby v^2 &= E(u) = 13425 u^6 + 34724 u^5 + 102 u^3 + 54150u + 11111,
    \\
    \Crv_F \defby y^2 &= F(x) = 47601 x^6 + 35850 x^5 + 40476 x^3 + 24699 x + 40502.
  \end{aligned}
\end{displaymath}
The derivatives of the Gundlach invariants at these points are given by
\begin{displaymath}
  \Cov_K(DG)(E) = \mat{43658}{17394}{16028}{26656},
  \quad \Cov_K(DG)(F) = \mat{15131}{739}{50692}{49952}.
\end{displaymath}
Computing derivatives of the modular equations as in
\Cref{prop:norm-matrix-hilbert}, we find that the isogeny is compatible with
the real multiplication embeddings for which~$E$ and~$F$ are
Hilbert-normalized. We do not known whether~$\isog$ is a~$\beta$- or a
$\conj{\beta}$-isogeny, so we have four candidates for the tangent matrix up to
sign:
\begin{displaymath}
  \begin{aligned}
    d\isog_{\beta,\pm} &= \mat{38932\alpha + 19466}{0}{0}{\pm(53318\alpha + 26659)},\\
    d\isog_{\conj{\beta},\pm} &= \mat{50651\alpha +
      53481}{0}{0}{\pm(11076\alpha + 5538)}
  \end{aligned}
\end{displaymath}
where $\alpha^2 + \alpha + 2 = 0$. We see that for these choices of curve
equations, the isogeny~$\isog$ is only defined over a quadratic extension
of~$k$; we could take a quadratic twist of~$\Crv_F$ to find a tangent matrix
over~$k$ instead.

The curve~$\Crv$ has a rational Weierstrass
point~$\bigl(36392, 0\bigr)$. We can bring it to~$(0,0)$, so
that~$\Crv$ is of the standard form
\begin{displaymath}
  \Crv\defby v^2 = 33461 u^6 + 7399 u^5 + 16387 u^4 + 34825 u^3 +
  14713 u^2 + u.
\end{displaymath}
This multiplies the tangent matrix on the right by
\begin{displaymath}
  \mat{44206}{18649}{0}{7615}.
\end{displaymath}

Choose $P = (0,0)$ as a base point on~$\Crv$, and~$z = \sqrt{u}$ as a
uniformizer. We solve the differential system up to precision~$O(z^{29})$.  It
turns out that the correct tangent matrix is~$d\isog_{\conj{\beta}, +}$ as the
other series do not come from rational fractions of degrees prescribed by
\Cref{prop:spqr-degrees}. We obtain in particular
\begin{align*}
  s(u) &= \frac{50255u^6 + 40618u^5 + 17196u^4 + 9527u^3 + 22804u^2
         + 49419u + 11726}{u^6 + 40883u^5 + 22913u^4 + 41828u^3 +
         18069u^2 + 14612u + 7238}, \\[2pt]
  p(u) &= \frac{35444u^6 + 9569u^5 + 52568u^4 + 3347u^3 + 9325u^2 +
         32206u + 7231}{u^6 + 40883u^5 + 22913u^4 + 41828u^3 + 18069u^2 +
         14612u + 7238}.
\end{align*}

%The isogeny is $k$-rational at the level of Kummer surfaces, but not on the
%Jacobians themselves: $\alpha$ appears on the numerator of $r(u, v)$.

% \begin{displaymath}
%   \begin{aligned}
%      r(u, v) &= v\frac{(47538\alpha + 23769)u^{10} + (54736\alpha +
%        27368)u^9 + (10441\alpha + 33376)u^8 + (51740\alpha + 25870)u^7
%        + (27982\alpha + 13991)u^6 + (47619\alpha + 51965)u^5 +
%        (55801\alpha + 56056)u^4 + (12897\alpha + 34604)u^3 +
%        (36860\alpha + 18430)u^2 + (30245\alpha + 43278)u + (55909\alpha
%        + 56110)}{u^{13} + 25455u^{12} + 42413u^{11} + 8422u^{10} +
%        1295u^9 + 10859u^8 + 4334u^7 + 33200u^6 + 52976u^5 + 10154u^4 +
%        37792u^3 + 19196u^2 + 19414u} \\
%      t(u, v) &= v \frac{(21373\alpha + 38842)u^{10} + (52517\alpha +
%        54414)u^9 + (30517\alpha + 43414)u^8 + (26715\alpha + 41513)u^7
%        + (39071\alpha + 47691)u^6 + (31123\alpha + 43717)u^5 +
%        (13028\alpha + 6514)u^4 + (29429\alpha + 42870)u^3 +
%        (12405\alpha + 34358)u^2 + (26273\alpha + 41292)u + (359\alpha +
%        28335)}{u^{13} + 25455u^{12} + 42413u^{11} + 8422u^{10} +
%        1295u^9 + 10859u^8 + 4334u^7 + 33200u^6 + 52976u^5 + 10154u^4 +
%        37792u^3 + 19196u^2 + 19414u}\\ 
%   \end{aligned}
% \end{displaymath}
% with $t(u,v) = (x_1 y_2 - x_2 y_1)/(x_2 - x_1)$, and $q = t^2 + r^2 p + srt$.

\printbibliography

\smallskip

\end{document}